\documentclass[reqno,12pt]{amsart}

\usepackage[utf8]{inputenc}
\usepackage[english]{babel}

\usepackage{amsfonts,amssymb}

\usepackage[symbol]{footmisc}
\usepackage{longtable}
\usepackage{enumerate}
\usepackage{cite}
\usepackage{hyperref}

\usepackage{tkz-graph}

\usepackage[left=0.75in,right=0.75in,top=1in,bottom=1in]{geometry}

\DeclareFontFamily{U}{mathc}{}
\DeclareFontShape{U}{mathc}{m}{it}%
{<->s*[1.03] mathc10}{}
\DeclareMathAlphabet{\mathscr}{U}{mathc}{m}{it}

\theoremstyle{plain}
\newtheorem{Th}{Theorem}
\newtheorem{Lem}[Th]{Lemma}
\newtheorem{defn}[Th]{Definition}
\newtheorem{Remark}[Th]{Remark}
\newtheorem*{theoremA}{Theorem A}
\newtheorem*{theoremB}{Theorem B}
\newtheorem*{theoremC}{Theorem C}

\let\opn\operatorname
\renewcommand \a {\mathfrak{a}}

\newcommand \Z {\mathbb{Z}}
\newcommand{\jor}{\mathfrak{J}}
\def\U{\mathcal{U}}
\def\S{\mathcal{S}}
\def\B{\mathcal{B}}
\def\T{\mathcal{T}}

\newcommand \GL {\operatorname{GL}}
\newcommand \Aut {\operatorname{Aut}}
\newcommand \Hom {\operatorname{Hom}}
\newcommand \tr {\operatorname{tr}}

\usepackage{fouriernc} 

\begin{document}
\sloppy

\vspace*{2cm}

    {\Large The algebraic and geometric classification of Jordan superalgebras}\footnote{The first part of the work is supported by the Russian Science Foundation under grant 22-71-10001. The second part of this work is supported by FCT UIDB/MAT/00212/2020, UIDP/MAT/00212/2020, 2023.08952.CEECIND; and  by grant F-FA-2021-423, Ministry of Higher Education, Science and Innovations of the Republic of Uzbekistan.}

\medskip

\medskip

\medskip

\medskip
\begin{center}

 {\bf
Kobiljon Abdurasulov\footnote{CMA-UBI, University of  Beira Interior, Covilh\~{a}, Portugal;  \ 
Saint Petersburg  University, Russia; \ 
Institute of Mathematics Academy of Sciences of Uzbekistan, Tashkent, Uzbekistan;
abdurasulov0505@mail.ru} 
Roman Lubkov\footnote{Department of Mathematics and Computer Science, Saint Petersburg University,
7/9 Universitetskaya nab., 199034 St. Petersburg, Russia; r.lubkov@spbu.ru, romanlubkov@yandex.ru}
\&   
Azamat Saydaliyev \footnote{Institute of Mathematics Academy of
Sciences of Uzbekistan, Tashkent, Uzbekistan; National University of Uzbekistan, Tashkent, Uzbekistan; \ azamatsaydaliyev6@gmail.com}
}

\end{center}

\noindent {\bf Abstract.}
We give the algebraic and geometric classification of complex four-dimensional Jordan superalgebras.
In particular, we describe all irreducible components in the corresponding varieties.\smallskip


\bigskip


\section*{Introduction}

\vspace{0.3cm}

The  algebraic classification (up to isomorphism) of algebras of small dimensions from a certain variety defined by a family of polynomial identities is a classic problem in the theory of non-associative algebras. 
Another interesting approach to studying algebras of a fixed dimension is to study them from a geometric point of view (that is, to study the degenerations and deformations of these algebras). The results in which the complete information about degenerations of a certain variety is obtained are generally referred to as the geometric classification of the algebras of these varieties. There are many results related to the algebraic and geometric classification of 
Jordan, Lie, Leibniz, Zinbiel, and other algebras 
(see,~\cite{adashev1,adashev2,degsulie,jord3, jor2, contr11, deggraa, BC99, CK07, ckls, ikp20, fkkv , degjor, GRH,GRH2, jdk, jdk2,km, km14, kkl21,k23,   kkp20,l24,S90 }). Degenerations are related to deformations and cohomology~\cite{ben,kkl20,MS}.

Superalgebras emerged in physics to provide a unified framework for the study of super symmetry of elementary particles. Jordan algebras, originating from quantum mechanics, gained importance due to their close connection to Lie theory. Finite-dimensional simple Jordan superalgebras over an algebraically closed field of characteristic zero were classified by Kac~\cite{Kac01} in 1977. One case remained unresolved, which was considered by Kantor~\cite{Kant} in 1990. Recently Racine and Zelmanov~\cite{MZ02} presented a classification of finite-dimensional simple Jordan superalgebras over arbitrary fields of characteristic distinct from 2, focusing on cases where the even part is semisimple. For the opposite case, where the even part is not semisimple, a classification was obtained by Mart\'inez and Zelmanov~\cite{MZ01} in 2002, completing the entire project.

In~\cite{Martin}, the authors focus on the classification of Jordan superalgebras of dimension up to three over an algebraically closed field of characteristic distinct from 2.
The mail goal of this paper is to determine the minimal dimension of exceptional Jordan superalgebras, raised in~\cite{Shest}. In~\cite{lowcom}, authors obtained all four dimensional indecomposable Jordan superalgebras. 

In the present paper, we  obtain an algebraic and geometric classification of four-dimensional Jordan superalgebras and find all irreducible components within that variety. Section 1 outlines the foundational concepts and establishes several key preliminary results. In Section 2, we proceed to classify all four-dimensional Jordan superalgebras.  Based on this classification, we determine the irreducible components within the variety in Section 3.

Our main results are summarized below. 

\begin{theoremA} The variety of complex $4$-dimensional Jordan superalgebras of type  $(1,3)$ has dimension $12$. 
It is defined by  $11$ rigid superalgebras and can be described as the closure of the union of $\mathrm{GL}_4(\mathbb{C})$-orbits of the superalgebras given in Theorem \ref{geo1}.    
\end{theoremA}

\begin{theoremB} The variety of complex $4$-dimensional Jordan superalgebras of type  $(2,2)$ has dimension $13$. 
It is defined by  $24$ rigid superalgebras and one one-parametric families of superalgebras and can be described as the closure of the union of $\mathrm{GL}_4(\mathbb{C})$-orbits of the superalgebras given in Theorem \ref{geo2}.    
\end{theoremB}

\begin{theoremC}
The variety of complex $4$-dimensional Jordan superalgebras of type  $(3,1)$ has dimension $15$. 
It is defined by  $21$ rigid superalgebras and can be described as the closure of the union of $\mathrm{GL}_4(\mathbb{C})$-orbits of the superalgebras given in Theorem \ref{geo3}.    
\end{theoremC} 

\section{Preliminaries}

\subsection{Jordan superalgebras}
\begin{defn}
A commutative algebra is called a {\it  Jordan  algebra}  if it satisfies the identity
$$(x^2y)x=x^2(yx).$$
\end{defn}

\begin{defn}
     A superalgebra $\mathcal{A}$ is an algebra with a $\mathbb{Z}_2$-grading. So $\mathcal{A}=\mathcal{A}_0\oplus\mathcal{A}_1$ is a direct sum of two vector spaces and
        $$\mathcal{A}_i \mathcal{A}_j \subseteq \mathcal{A}_{i+j}, \ \ \text{where} \ \ i,j\in \mathbb{Z}_2.$$
\end{defn}

Let $G$ be the Grassmann algebra over $\mathbb{F}$ given by the generators $1, \xi_1, \ldots , \xi_n, \ldots$  and the defining relations $\xi_i^2=0$ and $\xi_i\xi_j=-\xi_j\xi_i.$
The elements $1, \xi_{i_1} \xi_{i_2} \ldots  \xi_{i_k},\ i_1<i_2<\ldots <i_k,$ form a basis of the algebra $G$ over $\mathbb{F}$. Denote by $G_0$ and $G_1$ the subspaces spanned by the products
of even and odd lengths, respectively; then $G$ can be represented as the direct sum of these subspaces,
$G = G_0 \oplus G_1.$
Here the relations $G_iG_j \subseteq G_{i+j (mod \ 2)}$, $i,j = 0, 1$, hold.
In other words, $G$ is a $\mathbb{Z}_2$-graded algebra (or a superalgebra) over $\mathbb{F}$.
Suppose now that $A = A_0 \oplus A_1$ is an arbitrary superalgebra over $\mathbb{F}$. Consider the tensor product $G \otimes A$ of $\mathbb{F}$-algebras.
The subalgebra
$$G(A) = G_0 \otimes A_0 + G_1  \otimes A_1$$
of $G\otimes A$ is referred to as the Grassmann envelope of the superalgebra $A.$
Let $\Omega$ be a variety of algebras over $\mathbb{F}.$
A superalgebra $A = A_0 \oplus A_1$ is referred to as an
$\Omega$-superalgebra if its Grassmann envelope $G(A)$ is an algebra in $\Omega.$
In particular, $A = A_0  \oplus  A_1$ is
referred to as a Jordan superalgebra if its Grassmann envelope $G(A)$ is a Jordan algebra.
\begin{defn}
    A Jordan superalgebra is a superalgebra $\mathfrak{J}=\mathfrak{J}_0 + \mathfrak{J}_1$ satisfying the graded identities:
    $${xy=(-1)^{|x||y|}yx,}$$
    
$${((xy)z)t+(-1)^{|y||z|+|y||t|+|z||t|}((xt)z)y+(-1)^{|x||y|+|x||z|+|x||t|+|z||t|}((yt)z)x=}$$
$${(xy)(zt)+(-1)^{|t|(|y|+|z|)}(xt)(yz)+(-1)^{|y||z|}(xz)(yt),}$$
where $|x|=i$ for $x \in \mathfrak{J}_i$.
\end{defn}

For convenience, we use the following notation in the next sections.
\begin{center}
    $J(x,y,z,t)=((xy)z)t+(-1)^{|y||z|+|y||t|+|z||t|}((xt)z)y+(-1)^{|x||y|+|x||z|+|x||t|+|z||t|}((yt)z)x-(xy)(zt)-(-1)^{|t|(|y|+|z|)}(xt)(yz)-(-1)^{|y||z|}(xz)(yt)$.
\end{center}

\begin{defn}
For arbitrary elements $x,y \in \mathfrak{J}_0 \cup \mathfrak{J}_1$ of a Jordan superalgebra, the operator $D\colon \mathfrak{J} \to \mathfrak{J}$ satisfying
$$D(xy)=D(x)y+(-1)^{|D||x|}xD(y)$$
is a derivation of $\mathfrak{J}$, where $|D|=0$ if $D$ preserves the gradation and $|D|=1$ otherwise.
\end{defn}

In 2013, M.E.Martin~\cite{Martin2} described all Jordan algebras up to dimension four over an algebraically closed field. Based on that paper, we present the list of indecomposable Jordan algebras and superalgebras of dimension less than or equal to 3 in the following tables, as they will be used in the sequel.

\renewcommand{\arraystretch}{1.2}
\begin{longtable}{lllllllll}
\caption{Indecomposable Jordan algebras}
\label{tb:joral}\\
$\jor$ & & \multicolumn{6}{c}{Multiplication table} & dim \\
\hline
$\mathcal{U}_1$ & $:$ & ${e^{2}=e}$ & & & & & & 1 \\ 
$\mathcal{U}_2$ & $:$ & ${e^{2}=0}$ & & & & & & 1 \\
$\mathcal{B}_1$ & $:$ & $e_{1}^{2}=e_1$ &  $e_1 e_2 =e_2$ &  $e_{2}^{2}=0$ & & & & 2 \\
$\mathcal{B}_2$ & $:$ & $e_{1}^{2}=e_1$ &  $e_1 e_2 = \frac{1}{2} e_2$ &  $e_{2}^{2}=0$ & & & & 2 \\
$\mathcal{B}_3$ & $:$ & $e_{1}^{2}=e_2$ &  $e_1 e_2 =0$ &  $e_{2}^{2}=0$ & & & & 2 \\
$\mathcal{T}_1$ & $:$ & $e_{1}^{2}=e_1$ &  $e_{2}^{2}=e_3$ &  $e_{3}^{2}=0$ &  $ e_1 e_2 =e_2$ &  $e_1 e_3 =e_3$ &  $e_2 e_3 =0$ & 3 \\
$\mathcal{T}_2$ & $:$ & $e_{1}^{2}=e_1$ &  $e_{2}^{2}=0$ &  $e_{3}^{2}=0$ &  $e_1 e_2 =e_2$ &  $e_1 e_3 =e_3$ &  $e_2 e_3 =0$ & 3 \\
$\mathcal{T}_3$ & $:$ & $e_{1}^{2}=e_2$ &  $e_{2}^{2}=0$ &  $e_{3}^{2}=0$ &  $e_1 e_2 =e_3$ &  $e_1 e_3 =0$ &  $e_2 e_3 =0$ & 3 \\
$\mathcal{T}_4$ & $:$ & $e_{1}^{2}=e_2$ &  $e_{2}^{2}=0$ &  $e_{3}^{2}=0$ &  $e_1 e_2 =0$ &  $e_1 e_3 =e_2$ &  $e_2 e_3 =0$ & 3 \\
$\mathcal{T}_5$ & $:$ & $e_{1}^{2}=e_1$ &  $e_{2}^{2}=e_2$ &  $e_{3}^{2}=e_1 +e_2$ &  $e_1 e_2 =0$ &  $e_1 e_3 =\frac{1}{2}e_3$ &  $e_2 e_3 =\frac{1}{2} e_3$ & 3 \\
$\mathcal{T}_6$ & $:$ & $e_{1}^{2}=e_1$ &  $e_{2}^{2}=0$ &  $e_{3}^{2}=0$ &  $e_1 e_2 =\frac{1}{2} e_2$ &  $e_1 e_3 =  e_3$ &  $e_2 e_3 =0$ & 3 \\
$\mathcal{T}_7$ & $:$ & $e_{1}^{2}=e_1$ &  $e_{2}^{2}=0$ &  $e_{3}^{2}=0$ &  $e_1 e_2 =\frac{1}{2} e_2$ &  $e_1 e_3 =\frac{1}{2} e_3$ &  $e_2 e_3 =0$ & 3 \\
$\mathcal{T}_8$ & $:$ & $e_{1}^{2}=e_1$ &  $e_{2}^{2}=e_3$ &  $e_{3}^{2}=0$ &  $e_1 e_2 =\frac{1}{2} e_2$ &  $e_1 e_3 =0$ &  $e_2 e_3 =0$ & 3 \\
$\mathcal{T}_9$ & $:$ & $e_{1}^{2}=e_1$ &  $e_{2}^{2}=e_3$ &  $e_{3}^{2}=0$ &  $e_1 e_2 =\frac{1}{2} e_2$ &  $e_1 e_3 =e_3$ &  $e_2 e_3 =0$ & 3 \\
$\mathcal{T}_{10}$ & $:$ & $e_{1}^{2}=e_1$ &  $e_{2}^{2}=e_2$ &  $e_{3}^{2}=0$ &  $e_1 e_2 =0$ &  $e_1 e_3 =\frac{1}{2}e_3$ &  $e_2 e_3 =\frac{1}{2} e_3$ & 3 \\
\end{longtable}

Below we list indecomposable superalgebras denoted by $\S_j^i$ where the exponent $i$ represents its dimension.

\renewcommand{\arraystretch}{1.2}
\begin{longtable}{lllll||llllll}
\caption{Indecomposable superalgebras}
\label{tb:supal}\\
$\jor$ & & \multicolumn{3}{c}{Multiplication table} & $\jor$ & & \multicolumn{4}{c}{Multiplication table} \\
\hline
$\S_1^1$ & $:$ & $f^2=0$ & & & $\S_6^3$ & $:$ & $e^2=e$ & $e f_1=f_1$ & $e f_2=f_2$ \\
$\S_1^2$ & $:$ & $e^2=e$ & $e f=\frac{1}{2} f$  &                           &   $\S_7^3$ & $:$ & $e^2=e$ & $e f_1=\frac{1}{2} f_1$ & $e f_2=\frac{1}{2}f_2$ & $f_1 f_2=e$  \\ 
$\S_2^2$ & $:$ & $e^2=e$ & $e f =f$ &                                         &   $\S_8^3$ & $:$ & $e^2=e$ & $ef_1=f_1$ & $e f_2=f_2$ & $f_1 f_2=e$  \\ 
$\S_1^3$ & $:$ & $ef_1=f_2$ & $f_1 f_2=e$  &                                    &   $\S_9^3$ & $:$ &  $e_1^2=e_1$ & $e_1e_2=e_2$ & $e_1 f=\frac{1}{2} f$ \\
$\S_2^3$ & $:$ & $f_1f_2=e$   & &                                              &   $\S_{10}^3$ & $:$ &  $e_1^2=e_1$ & $e_1 e_2=e_2$ & $e_1f=\frac{1}{2}f$ \\ 
$\S_3^3$ & $:$ & $e f_1=f_2$  & &                                              &   $\S_{11}^3$ & $:$ & $e_1^2=e_1$ & $e_1e_2=\frac{1}{2}e_2$ & $e_1 f=\frac{1}{2}f $  \\ 
$\S_4^3$ & $:$ & $e^2=e$ & $e f_1=f_1$ & $e f_2=\frac{1}{2} f_2$              &   $\S_{12}^3$ & $:$ & $e_1^2=e_1$ & $e_1 e_2=\frac{1}{2} e_2$ & $e_1f=f$  \\ 
$\S_5^3$ & $:$ & $e^2=e$ & $e f_1=\frac{1}{2}f_1$ & $e f_2=\frac{1}{2}f_2$    &   $\S_{13}^3$ & $:$ & $e_1^2=e_1$ & $e_2^2=e_2$ & $e_1f=\frac{1}{2}f$ & $e_2 f=\frac{1}{2} f$  \\
\end{longtable}

\subsection{Degenerations}
Given an $(m,n)$-dimensional vector superspace $V=V_0\oplus V_1$, the set
$$\Hom(V \otimes V,V)=(\Hom(V \otimes V,V))_0\oplus (\Hom(V \otimes V,V))_1$$ is a vector superspace of dimension $m^3+3mn^2$. This space has a structure of the affine variety $\mathbb{C}^{m^3+3mn^2}.$ If we fix a basis $\{e_1,\dots,e_m,f_1,\dots,f_n\}$ of $V$, then any $\mu\in \Hom(V \otimes V,V)$ is determined by $m^3+3mn^2$ structure constants
$\alpha_{i,j}^k,\beta_{i,j}^q,\gamma_{i,j}^q, \delta_{p,q}^k \in\mathbb{C}$ such that
$$\mu(e_i\otimes e_j)=\sum\limits_{k=1}^m\alpha_{i,j}^ke_k,
\quad \mu(e_i\otimes f_p)=\sum\limits_{q=1}^n\beta_{i,p}^qf_q,
\quad \mu(f_p\otimes e_i)=\sum\limits_{q=1}^n\gamma_{p,i}^qf_q,
\quad \mu(f_p\otimes f_q)=\sum\limits_{k=1}^m\delta_{p,q}^ke_k.$$
A subset $\mathbb{L}(T)$ of $\Hom(V \otimes V,V)$ is {\it Zariski-closed} if it can be defined by a set of polynomial equations $T$ in the variables
$\alpha_{i,j}^k,\beta_{i,p}^q, \gamma_{p,i}^q, \delta_{p,q}^k$ ($1\le i,j,k\le m,\ 1\leq p,q\leq n$).

Let $\mathcal{S}^{m,n}$ be the set of all superalgebras of dimension $(m,n)$ defined by the family of polinomial super-identities $T$, understood as a subset $\mathbb{L}(T)$ of an affine variety $\Hom(V\otimes V, V)$. Then one can see that $\mathcal{S}^{m,n}$ is a Zariski-closed subset of the variety $\Hom(V\otimes V, V).$
The group $G=(\Aut V)_0\simeq\GL(V_0)\oplus\GL(V_1)$ acts on $\mathcal{S}^{m,n}$ by conjugations:
$$ (g * \mu )(x\otimes y) = g\mu(g^{-1}x\otimes g^{-1}y)$$
for $x,y\in V$, $\mu\in\mathbb{L}(T)$ and $g\in G$.

Thus, $\mathcal{S}^{m,n}$ is decomposed into $G$-orbits that correspond to the isomorphism classes of superalgebras. The dimension of the orbit is found via 
\begin{equation}\label{eq1}
\operatorname{dim}O(\mu)=n^2-\operatorname{dim}\mathfrak{Der}(J),
\end{equation}
where $\mathfrak{Der}(J)$ denotes the Lie algebra of derivations of $J$ corresponding to $\mu$.

Denote by $O(\mu)$ the orbit of $\mu\in\mathbb{L}(T)$ under the action of $G$ and by $\overline{O(\mu)}$ the Zariski closure of $O(\mu)$. Let $J, J' \in \mathcal{S}^{m,n}$  and $\lambda,\mu\in \mathbb{L}(T)$ represent $J$ and $J'$, respectively. We say that $\lambda$ degenerates to $\mu$ and write $\lambda\to \mu$ if $\mu\in\overline{O(\lambda)}$. Note that in this case we have $\overline{O(\mu)}\subset\overline{O(\lambda)}$. Hence, the definition of a degeneration does not depend on the choice of $\mu$ and $\lambda$, and we right indistinctly $J\to J'$ instead of $\lambda\to\mu$ and $O(J)$ instead of $O(\lambda)$. If $J\not\cong J'$, then the assertion $J\to J'$ is called a {\it proper degeneration}. We write $J\not\to J'$ if $J'\not\in\overline{O(J)}$.

Let $J$ be represented by $\lambda\in\mathbb{L}(T)$. Then $J$ is  {\it rigid} in $\mathbb{L}(T)$ if $O(\lambda)$ is an open subset of $\mathbb{L}(T)$.  Recall that a subset of a variety is called irreducible if it cannot be represented as a union of two non-trivial closed subsets. A maximal irreducible closed subset of a variety is called {\it irreducible component}.  In particular, $J$ is rigid in $\mathcal{S}^{m,n}$ iff $\overline{O(\lambda)}$ is an irreducible component of $\mathbb{L}(T)$. It is well known that any affine variety can be represented as a finite union of its irreducible components in a unique way. We denote by $\opn{Rig}(\mathcal{S}^{m,n})$ the set of rigid superalgebras in $\mathcal{S}^{m,n}$.

\begin{Remark}
There is no degeneration between two superalgebras of types $(n-i,i)$ and $(n-j,j)$ if $i\neq j$.
\end{Remark}

This statement follows from the fact that superalgebras of different types correspond to affine varieties of different dimensions. Suppose that superalgebras of types $(n-i,i)$ and $(n-j,j)$ correspond to the same variety $\mathbb{C}^{p^3+3pq^2}$ when $i\neq j$. In that case we have the following equality:
$$(n-i)^3+3(n-i)i^2=(n-j)^3+3(n-j)j^2.$$
Therefore, we obtain
$$(i-j)(3n^2-6n(i+j)+4(i^2+ij+j^2))=0.$$
Since $i\neq j$, the second bracket, which contains quadratic term with respect to $n$, must be zero. However, this requires
$$36(i+j)^2-48(i^2+ij+j^2)=24ij-12(i^2+j^2) \geq 0,$$
leading to
$$(i-j)^2\leq 0,$$
which cannot hold unless $i=j.$

\subsection{Principal notation}
Let $\mathscr{JS}^{m,n}$ be the set of all Jordan superalgebras of dimension $(m,n).$
Let $J$ be a Jordan superalgebra with a fixed basis $\{e_1,\dots,e_m,f_1,\dots f_n\}$, defined by
\[e_ie_j=\sum_{k=1}^m\alpha_{ij}^ke_k,\quad e_if_j=\sum_{k=1}^n\beta_{ij}^kf_k,\quad f_if_j=\sum_{k=1}^m\gamma_{ij}^ke_k.\]
In the sequel, we use the following notation:
\begin{enumerate}
\item $\a(J)$ is the Jordan superalgebra with the same underlying vector superspace as $J$ and defined by $f_if_j=\displaystyle\sum_{k=1}^n\gamma_{ij}^ke_k$. 
\item $J^1=J$, $J^r=J^{r-1}J+J^{r-2}J^2+\dots+ JJ^{r-1}$, and in every case $J^r=(J^r)_0\oplus (J^r)_1$.
\item $c_{i,j}=\displaystyle\frac{\tr (L(x)^i)\cdot\tr(L(y)^j)}{\tr( L(x)^i\cdot L(y)^j)}$ 
is the Burde invariant, where $L(x)$
is the left 
multiplication. This invariant $c_{i,j}$ 
is defined as a quotient of two polynomials in the structure constants of $J$, for all $x,y\in J$ such that both polynomials are not zero and $c_{i,j}$ 
is independent of the choice of $x,y$.
\end{enumerate}

\subsection{Methods}

First of all, if $J\to J'$ and $J\not\cong J'$, then $\dim\Aut(J)<\dim\Aut(J')$, where $\Aut(J)$ is the space of automorphisms of $J$. Secondly, if $J\to J'$ and $J'\to J''$, then $J\to J''$. If there is no $J'$ such that $J\to J'$ and $J'\to J''$ are proper degenerations, then the assertion $J\to J''$ is called a {\it primary degeneration}. If $\dim\Aut(J)<\dim\Aut(J'')$ and there are no $J'$ and $J'''$ such that $J'\to J$, $J''\to J'''$, $J'\not\to J'''$ and one of the assertions $J'\to J$ and $J''\to J'''$ is a proper degeneration,  then the assertion $J \not\to J''$ is called a {\it primary non-degeneration}. It suffices to prove only primary degenerations and non-degenerations to describe degenerations in the variety under consideration. It is easy to see that any superalgebra degenerates to the superalgebra with zero multiplication. From now on we use this fact without mentioning it.

Let us describe the methods for proving primary non-degenerations. The main tool for this is the following lemma~\cite{jord3}.

\begin{Lem}\label{lema:inv}
If $J\to J'$, then the following hold:
\begin{enumerate}

\item $\dim (J^r)_i\geq\dim (J'^r)_i$, for $i\in\Z_2;$
\item $(J)_0\to (J')_0;$
\item $\a(J)\to\a(J');$

\item If the Burde invariant exists for $J$ and $J'$, then both superalgebras have the same Burde invariant$;$
\item If $J$ is associative, then $J'$ must be associative. In fact, if $J$ satisfies a P.I. then $J'$ must satisfy the same P.I.
\end{enumerate}
\end{Lem}
In the cases where all of these criteria can't be applied to prove $J\not\to J'$, we define $\mathcal{R}$ by a set of polynomial equations and give a basis of $V$, in which the structure constants of $\lambda$ give a solution to all these equations. Further on, we omit the verification of the fact that $\mathcal{R}$ is stable under the action of the subgroup of upper triangular matrices and of the fact that $\mu\not\in\mathcal{R}$ for any choice of a basis of $V$. These verifications can be done by direct calculations.

{\bf Degenerations of Graded algebras}.  Let
 $G$ be a trivial group and  let  $\mathcal {V}(\mathcal{F})$ be a variety of algebras defined by a family of  polynomial identities $\mathcal{F}$. It is important to notice that degeneration on the  $G$-graded variety $G\mathcal{V}( \mathcal{F})$ is a more restrictive notion than degeneration on the variety $\mathcal{V}(\mathcal{F})$, In fact, consider $A, A^\prime \in  G\mathcal  {V}(\mathcal{F})$  such that  $ A,  A^\prime  \in \mathcal{V}(\mathcal{F})$, a degeneration between the algebras $A$ and $A^\prime$  may not give rise  to a degeneration  between  the $G$-graded algebras $A$ and $A^\prime$, since  the matrices describing the  basis changes in  $G\mathcal  {V}(\mathcal{F})$  must preserve the $G$-graduation. Hence,  we have the following  result.

\begin{Lem}
 Let  $ A, A^\prime \in G\mathcal {V}(\mathcal{F}) \cap \mathcal {V}(\mathcal{F})$. If $A  \not \to  A^\prime $ as algebras, then $A \not \to A^\prime $ as $G$-graded algebras.

\end{Lem}

Additionally, we need the following results from~\cite{degjor}
     
\begin{Th}\label{2d}
The graph of primary degenerations for two-dimensional Jordan algebras  has the following form:
\end{Th}

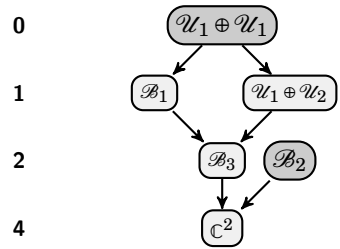
\begin{footnotesize}

$${\begin{tikzpicture}[->,>=stealth',shorten >=0.0cm,auto,node distance=0.9cm,
                    thick,main node/.style={rectangle,draw,fill=gray!12,rounded corners=1ex,font=\sffamily \tiny
                    \bfseries },rigid node/.style={rectangle,draw,fill=black!20,rounded corners=1.5ex,font=\sffamily \bf \bfseries },style={draw,font=\sffamily \scriptsize \bfseries }]
\node (0)   {0};

\node (00a1) [right of=0] {};
\node (00a2) [right of=00a1] {};
\node (00a3) [right of=00a2] {};
\node (00a4) [right of=00a3] {};

\node (01) [below of=0] {1};

\node (01a1) [right of=01] {};
\node (01a2) [right of=01a1] {};
\node (01a3) [right of=01a2] {};
\node (01a4) [right of=01a3] {};

\node (02) [below of=01] {2};

\node (02a1) [right of=02] {};
\node (02a2) [right of=02a1] {};
\node (02a3) [right of=02a2] {};
\node (02a4) [right of=02a3] {};

\node (04) [below of=02] {4};

\node (04a1) [right of=04] {};
\node (04a2) [right of=04a1] {};
\node (04a3) [right of=04a2] {};
\node (04a4) [right of=04a3] {};

\node  [rigid node] (b4) [right of=00a2] {$\mathcal{U}_{1} \oplus \mathcal{U}_{1}$};
\node  [main node] (b1) [right of=01a1] {$\mathcal{B}_{1}$};	
\node  [rigid node] (b2) [right of=02a3] {$\mathcal{B}_{2}$};	
\node  [main node] (b5) [right of=01a3] {$\mathcal{U}_{1} \oplus \mathcal{U}_{2}$};
\node  [main node] (b3) [right of=02a2] {$\mathcal{B}_{3}$};
\node  [main node] (c2) [right of=04a2] {$\mathbb{C}^{2}$};

\path[every node/.style={font=\sffamily\small}]

(b4) edge   node[above] {} (b1)
(b4) edge   node[above] {} (b5)
(b5) edge   node[above] {} (b3)
(b1) edge   node[above] {} (b3)
(b2) edge   node[above] {} (c2)
(b3) edge   node[above] {} (c2);

\end{tikzpicture}}$$

\end{footnotesize}
\begin{normalsize}
\begin{Th}
\label{3d}
The graph of primary degenerations for three-dimensional Jordan algebras  has the following form:
\end{Th}
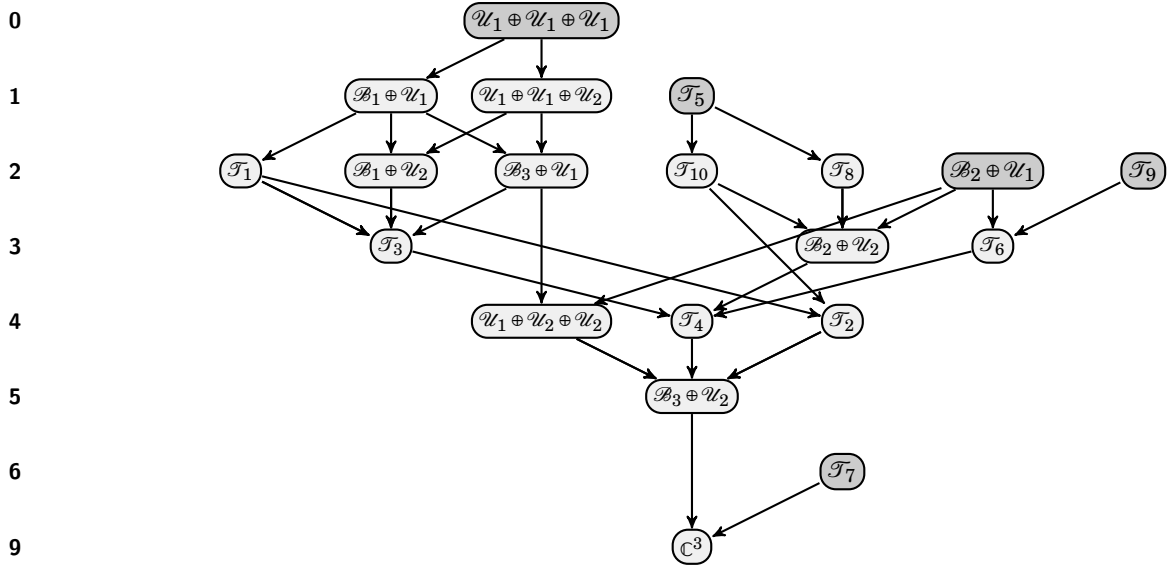
\scriptsize

\begin{center}
\begin{tikzpicture}[->,>=stealth',shorten >=0.0cm,auto,node distance=1cm,thick,
                    main node/.style={rectangle,draw,fill=gray!12,rounded corners=1.5ex,font=\sffamily \tiny \bfseries },
                    rigid node/.style={rectangle,draw,fill=black!20,rounded corners=1.5ex,font=\sffamily \bf \bfseries },
                    style={draw,font=\sffamily \scriptsize \bfseries }]
\node (0)   {0};

\node (00a1) [right of=0] {};
\node (00a2) [right of=00a1] {};
\node (00a3) [right of=00a2] {};
\node (00a4) [right of=00a3] {};
\node (00a5) [right of=00a4] {};
\node (00a6) [right of=00a5] {};
\node (00a7) [right of=00a6] {};
\node (00a8) [right of=00a7] {};
\node (00a9) [right of=00a8] {};
\node (00a10) [right of=00a9] {};
\node (00a11) [right of=00a10] {};
\node (00a12) [right of=00a11] {};
\node (00a13) [right of=00a12] {};
\node (00a14) [right of=00a13] {};
\node (00a15) [right of=00a14] {};
\node (00a16) [right of=00a15] {};
\node (00a17) [right of=00a16] {};
\node (00a18) [right of=00a17] {};
\node (00a19) [right of=00a18] {};
\node (00a20) [right of=00a19] {};
\node (00a21) [right of=00a20] {};
\node (00a22) [right of=00a21] {};
\node (00a23) [right of=00a22] {};
\node (00a24) [right of=00a23] {};
\node (00a25) [right of=00a24] {};
\node (00a26) [right of=00a25] {};
\node (00a27) [right of=00a26] {};

\node (01) [below of=0] {1};

\node (01a1) [right of=01] {};
\node (01a2) [right of=01a1] {};
\node (01a3) [right of=01a2] {};
\node (01a4) [right of=01a3] {};
\node (01a5) [right of=01a4] {};
\node (01a6) [right of=01a5] {};
\node (01a7) [right of=01a6] {};
\node (01a8) [right of=01a7] {};
\node (01a9) [right of=01a8] {};
\node (01a10) [right of=01a9] {};
\node (01a11) [right of=01a10] {};
\node (01a12) [right of=01a11] {};
\node (01a13) [right of=01a12] {};
\node (01a14) [right of=01a13] {};
\node (01a15) [right of=01a14] {};

\node (02) [below of=01] {2};

\node (02a1) [right of=02] {};
\node (02a2) [right of=02a1] {};
\node (02a3) [right of=02a2] {};
\node (02a4) [right of=02a3] {};
\node (02a5) [right of=02a4] {};
\node (02a6) [right of=02a5] {};
\node (02a7) [right of=02a6] {};
\node (02a8) [right of=02a7] {};
\node (02a9) [right of=02a8] {};
\node (02a10) [right of=02a9] {};
\node (02a11) [right of=02a10] {};
\node (02a12) [right of=02a11] {};
\node (02a13) [right of=02a12] {};
\node (02a14) [right of=02a13] {};
\node (02a15) [right of=02a14] {};
\node (02a16) [right of=02a15] {};
\node (02a17) [right of=02a16] {};
\node (02a18) [right of=02a17] {};
\node (02a19) [right of=02a18] {};
\node (02a20) [right of=02a19] {};
\node (02a21) [right of=02a20] {};
\node (02a22) [right of=02a21] {};
\node (02a23) [right of=02a22] {};
\node (02a24) [right of=02a23] {};
\node (02a25) [right of=02a24] {};
\node (02a26) [right of=02a25] {};
\node (02a27) [right of=02a26] {};

\node (03)[below of=02]{3};

\node (03a1) [right of=03] {};
\node (03a2) [right of=03a1] {};
\node (03a3) [right of=03a2] {};
\node (03a4) [right of=03a3] {};
\node (03a5) [right of=03a4] {};
\node (03a6) [right of=03a5] {};
\node (03a7) [right of=03a6] {};
\node (03a8) [right of=03a7] {};
\node (03a9) [right of=03a8] {};
\node (03a10) [right of=03a9] {};
\node (03a11) [right of=03a10] {};
\node (03a12) [right of=03a11] {};
\node (03a13) [right of=03a12] {};
\node (03a14) [right of=03a13] {};
\node (03a15) [right of=03a14] {};
\node (03a16) [right of=03a15] {};
\node (03a17) [right of=03a16] {};
\node (03a18) [right of=03a17] {};
\node (03a19) [right of=03a18] {};
\node (03a20) [right of=03a19] {};
\node (03a21) [right of=03a20] {};
\node (03a22) [right of=03a21] {};
\node (03a23) [right of=03a22] {};
\node (03a24) [right of=03a23] {};
\node (03a25) [right of=03a24] {};
\node (03a26) [right of=03a25] {};
\node (03a27) [right of=03a26] {};
*\node (03a28) [right of=03a27] {};

\node (04) [below of=03] {4};

\node (04a1) [right of=04] {};
\node (04a2) [right of=04a1] {};
\node (04a3) [right of=04a2] {};
\node (04a4) [right of=04a3] {};
\node (04a5) [right of=04a4] {};
\node (04a6) [right of=04a5] {};+
\node (04a7) [right of=04a6] {};
\node (04a8) [right of=04a7] {};
\node (04a9) [right of=04a8] {};
\node (04a10) [right of=04a9] {};
\node (04a11) [right of=04a10] {};
\node (04a12) [right of=04a11] {};
\node (04a13) [right of=04a12] {};
\node (04a14) [right of=04a13] {};
\node (04a15) [right of=04a14] {};
\node (04a16) [right of=04a15] {};
\node (04a17) [right of=04a16] {};
\node (04a18) [right of=04a17] {};
\node (04a19) [right of=04a18] {};
\node (04a20) [right of=04a19] {};
\node (04a21) [right of=04a20] {};
\node (04a22) [right of=04a21] {};
\node (04a23) [right of=04a22] {};
\node (04a24) [right of=04a23] {};
\node (04a25) [right of=04a24] {};
\node (04a26) [right of=04a25] {};
\node (04a27) [right of=04a26] {};
\node (04a28) [right of=04a27] {};

\node (05) [below of=04] {5};

\node (05a1) [right of=05] {};
\node (05a2) [right of=05a1] {};
\node (05a3) [right of=05a2] {};
\node (05a4) [right of=05a3] {};
\node (05a5) [right of=05a4] {};
\node (05a6) [right of=05a5] {};
\node (05a7) [right of=05a6] {};
\node (05a8) [right of=05a7] {};
\node (05a9) [right of=05a8] {};
\node (05a10) [right of=05a9] {};
\node (05a11) [right of=05a10] {};
\node (05a12) [right of=05a11] {};
\node (05a13) [right of=05a12] {};
\node (05a14) [right of=05a13] {};
\node (05a15) [right of=05a14] {};
\node (05a16) [right of=05a15] {};
\node (05a17) [right of=05a16] {};
\node (05a18) [right of=05a17] {};
\node (05a19) [right of=05a18] {};
\node (05a20) [right of=05a19] {};
\node (05a21) [right of=05a20] {};
\node (05a22) [right of=05a21] {};
\node (05a23) [right of=05a22] {};
\node (05a24) [right of=05a23] {};
\node (05a25) [right of=05a24] {};
\node (05a26) [right of=05a25] {};
\node (05a27) [right of=05a26] {};
\node (05a28) [right of=05a27] {};

\node (06)[below of=05] {6};

\node (09) [below of=06] {9};

\node (06a1) [right of=06] {};
\node (06a2) [right of=06a1] {};
\node (06a3) [right of=06a2] {};
\node (06a4) [right of=06a3] {};
\node (06a5) [right of=06a4] {};
\node (06a6) [right of=06a5] {};
\node (06a7) [right of=06a6] {};
\node (06a8) [right of=06a7] {};
\node (06a9) [right of=06a8] {};
\node (06a10) [right of=06a9] {};
\node (06a11) [right of=06a10] {};
\node (06a12) [right of=06a11] {};
\node (06a13) [right of=06a12] {};
\node (06a14) [right of=06a13] {};
\node (06a15) [right of=06a14] {};
\node (06a16) [right of=06a15] {};
\node (06a17) [right of=06a16] {};
\node (06a18) [right of=06a17] {};
\node (06a19) [right of=06a18] {};
\node (06a20) [right of=06a19] {};
\node (06a21) [right of=06a20] {};
\node (06a22) [right of=06a21] {};
\node (06a23) [right of=06a22] {};
\node (06a24) [right of=06a23] {};
\node (06a25) [right of=06a24] {};
\node (06a26) [right of=06a25] {};
\node (06a27) [right of=06a26] {};
\node (06a28) [right of=06a27] {};

\node  [main node] (t1) [right of = 02a2] {$\T_{1}$};

\node  [main node] (t12) [right of=01a4] {$\mathcal{B}_1 \oplus \mathcal{U}_1$};	
\node  [main node] (t15) [right of=02a4] {$\mathcal{B}_1 \oplus \mathcal{U}_2$};	
\node  [main node] (t3)  [right of=03a4] {$\T_3$};

\node  [rigid node] (t11) [right of=00a6] {$\mathcal{U}_1 \oplus \mathcal{U}_1 \oplus \mathcal{U}_1$};	
\node  [main node] (t13) [right of=01a6] {$\mathcal{U}_1 \oplus \mathcal{U}_1 \oplus \mathcal{U}_2$};	
\node  [main node] (t16) [right of=02a6] {$\mathcal{B}_3 \oplus \mathcal{U}_1$};
\node  [main node] (t17) [right of=04a6] {$\mathcal{U}_1 \oplus \mathcal{U}_2 \oplus \mathcal{U}_2$};

\node  [rigid node] (t5)  [right of=01a8] {$\T_{5}$};		
\node  [main node] (t10) [right of=02a8] {$\T_{10}$};
\node  [main node] (t4)  [right of=04a8] {$\T_{4}$};
\node  [main node] (t19) [right of=05a8] {$\mathcal{B}_3 \oplus \mathcal{U}_2$};

\node  [main node] (t8)  [right of=02a10] {$\T_{8}$};
\node  [main node] (t18) [right of=03a10] {$\mathcal{B}_2 \oplus \mathcal{U}_2$};
\node  [main node] (t2)  [right of=04a10] {$\T_{2}$};
\node  [rigid node] (t7) [right of=06a10] {$\T_{7}$};

\node  [rigid node] (t14) [right of=02a12] {$\mathcal{B}_2 \oplus \mathcal{U}_1$};
\node  [main node] (t6)  [right of=03a12] {$\T_6$};

\node  [rigid node] (t9) [right of=02a14] {$\T_9$};

\node  [main node] (t20) [below of=06a9] {$\mathbb{C}^3$};

\path[every node/.style={font=\sffamily\small}]

(t1) edge   node[above] {} (t2)
(t1) edge   node[above] {} (t3)
(t1) edge   node[above] {} (t3)

(t2) edge   node[above] {} (t19)
(t2) edge   node[above] {} (t19)

(t3) edge   node[above] {} (t4)

(t4) edge   node[above] {} (t19)

(t5) edge   node[above] {} (t8)
(t5) edge   node[above] {} (t10)

(t6) edge   node[above] {} (t4)

(t7) edge   node[above] {} (t20)

(t8) edge   node[above] {} (t18)
(t8) edge   node[above] {} (t18)

(t9) edge   node[above] {} (t6)

(t10) edge   node[above] {} (t18)
(t10) edge   node[above] {} (t2)

(t11) edge   node[above] {} (t12)
(t11) edge   node[above] {} (t13)

(t12) edge   node[above] {} (t1)
(t12) edge   node[above] {} (t15)
(t12) edge   node[above] {} (t16)

(t13) edge   node[above] {} (t15)
(t13) edge   node[above] {} (t16)

(t14) edge   node[above] {} (t6)
(t14) edge   node[above] {} (t17)
(t14) edge   node[above] {} (t18)

(t15) edge   node[above] {} (t3)
(t15) edge   node[above] {} (t3)

(t16) edge   node[above] {} (t3)
(t16) edge   node[above] {} (t17)

(t17) edge   node[above] {} (t19)
(t17) edge   node[above] {} (t19)

(t18) edge   node[above] {} (t4)

(t19) edge   node[above] {} (t20);

\end{tikzpicture}
\end{center}
\end{normalsize}

\section{Algebraic classification of four dimensional Jordan superalgebras}
All four dimensional indecomposable Jordan superalgebras are obtained in~\cite{lowcom} while classifying low-dimensional commutative power-associative superalgebras. Below we present our classification of all four dimensional Jordan superalgebras. Dimensions of orbits in the theorems below are calculated using the formula~(\ref{eq1}).

\begin{Th}
    Up to isomorphism there are 19 Jordan superalgebras of the type $(1,3)$, which are presented below with some additional information:

\begin{center}
\begin{longtable}{l|c|l|l}
     \textnumero & Orbit & Multiplication rules & Decomposition \\
     \hline
     $\bf J_1$ & 6 & $f_1f_2=e$ & $\mathcal{S}_2^3 \oplus \mathcal{S}_1^1$\\
     $\bf J_2$ & 6 & $ef_1=f_2$ & $\mathcal{S}_3^3 \oplus \mathcal{S}_1^1$\\
     $\bf J_3$ & 9 & $ef_1=f_2,$ $f_1f_2=e$  & $\mathcal{S}_1^3 \oplus \mathcal{S}_1^1$\\
     $\bf J_4$ & 9 & $ef_1=f_2,$ $f_1f_3=e$ & Indecomposable \\
     $\bf J_5$ & 12 & $ef_1=f_2,$ $f_2f_3=e$ & Indecomposable\\
     $\bf J_6$ & 11 & $ef_1=f_2,$ $ef_2=f_3$  & Indecomposable\\
     $\bf J_7$ & 7 & $e^2=e$ & $\mathcal{U}_1 \oplus \mathcal{S}_1^1 \oplus \mathcal{S}_1^1 \oplus \mathcal{S}_1^1$\\
     $\bf J_8$ & 10 & $e^2=e, \ ef_3=\frac{1}{2}f_3$ & $\mathcal{S}_1^2 \oplus \mathcal{S}_1^1 \oplus \mathcal{S}_1^1$\\
     $\bf J_{9}$ & 10 & $e^2=e, \ ef_3=f_3$ & $\mathcal{S}_2^2 \oplus \mathcal{S}_1^1 \oplus \mathcal{S}_1^1$\\
     $\bf J_{10}$ & 9 & $e^2=e, \ ef_2=\frac{1}{2}f_2, \ ef_3=\frac{1}{2}f_3$ & $\mathcal{S}_5^3 \oplus \mathcal{S}_1^1$\\
     $\bf J_{11}$ & 10 & $e^2=e, \ ef_2=\frac{1}{2}f_2, \ ef_3=\frac{1}{2}f_3, \ f_2f_3=e$ & $\mathcal{S}_7^3 \oplus \mathcal{S}_1^1$\\
     $\bf J_{12}$ & 12 & $e^2=e, \ ef_2=\frac{1}{2}f_2, \ ef_3=f_3$ & $\mathcal{S}_4^3 \oplus \mathcal{S}_1^1$\\
     $\bf J_{13}$ & 11 & $e^2=e, \ ef_2=f_2, \ ef_3=f_3$ & $\mathcal{S}_6^3 \oplus \mathcal{S}_1^1$\\
     $\bf J_{14}$ & 12 & $e^2=e, \ ef_2=f_2, \ ef_3=f_3, \ f_2f_3=e$ & $\mathcal{S}_8^3 \oplus \mathcal{S}_1^1$\\
     $\bf J_{15}$ & 4 & $e^2=e, \ ef_1=\frac{1}{2}f_1, \ ef_2=\frac{1}{2}f_2, \ ef_3=\frac{1}{2}f_3$ & Indecomposable\\
     $\bf J_{16}$ & 9 & $e^2=e, \ ef_1=\frac{1}{2}f_1, \ ef_2=\frac{1}{2}f_2, \ ef_3=f_3$ & Indecomposable\\
     $\bf J_{17}$ & 10 & $e^2=e, \ ef_1=\frac{1}{2}f_1, \ ef_2=f_2, \ ef_3=f_3$ & Indecomposable\\
     $\bf J_{18}$ & 7 & $e^2=e, \ ef_1=f_1, \ ef_2=f_2, \ ef_3=f_3$ & Indecomposable\\
     $\bf J_{19}$ & 10 & $e^2=e, \ ef_1=f_1, \ ef_2=f_2, \ ef_3=f_3, \ f_1f_2=e$ & Indecomposable

\end{longtable}
\end{center}
\end{Th}

\begin{proof}
As $\jor_0$ is a Jordan algebra, we have subcases $\jor_0 \cong \mathcal{U}_1$ and $\jor_0 \cong \mathcal{U}_2$.

Then we have the following multiplications for $e \in J_0, \  f_1,f_2,f_3 \in J_1$ 
\begin{center}
\begin{tabular}{ll}
  $e f_1= \alpha_1 f_1 + \alpha_2 f_2 + \alpha_3 f_3, \ $ & $f_1 f_2=\xi_1 e ,$\\
  $e f_2= \beta_1 f_1 + \beta_2 f_2 + \beta_3 f_3, \ $ & $f_1 f_3=\xi_2 e,$ \\
  $e f_3= \gamma_1 f_1 + \gamma_2 f_2 + \gamma_3 f_3, \ $ & $f_2 f_3=\xi_3 e.$
\end{tabular}
\end{center}

The linear operator $L_{x}: \jor \to \jor, x\in \jor $ such that $L_{x}(y)=xy$ is called a left multiplication operator.
It is obvious that:
$$x\in \jor_0, \  L_{x}: \jor_0 \to \jor_0, \  L_{x}: \jor_1 \to \jor_1,$$
$$x\in \jor_1, \  L_{x}: \jor_0 \to \jor_1, \  L_{x}: \jor_1 \to \jor_0.$$

For the action of the operator $L_{e}$ on $\jor_1$ we can write the following matrix
$$\begin{pmatrix}
  \alpha_1 & \alpha_2 & \alpha_3\\  
  \beta_1 & \beta_2 & \beta_3\\
  \gamma_1 & \gamma_2 & \gamma_3
\end{pmatrix}.$$
However, it is easy to prove that, by using a simple change of basis, the matrix of $L_{e}$ has one of the following forms:
$$\begin{array}{ccc}
 \left(\begin{array}{ccc}
 \mu_1 & 0 & 0\\
 0 & \mu_2 & 0\\
 0 & 0 & \mu_3
 \end{array}\right), &
  \left(\begin{array}{ccc}
\mu_1 & 1 & 0\\
0 & \mu_1 & 0\\
0 & 0 & \mu_3
 \end{array}\right), &
  \left(\begin{array}{ccc}
\mu_1 & 1 & 0\\
0 & \mu_1 & 1\\
0 & 0 & \mu_1
 \end{array}\right).
 \end{array}$$

\noindent
\underline{Let $\jor_0 \cong \mathcal{U}_2$.} 
\medskip

\noindent
\textbf{i)} Let $L_{e} \simeq \left(\begin{array}{ccc}
 \mu_1 & 0 & 0\\
 0 & \mu_2 & 0\\
 0 & 0 & \mu_3
 \end{array}\right),$ then the rule of multiplication can be written as follows:
 \begin{center}
 \begin{tabular}{llllll}
  $e f_1= \mu_1 f_1,$ & $e f_2= \mu_2 f_2,$ & $e f_3= \mu_3 f_3,$ & $f_1 f_2=\xi_1 e ,$ &  $f_1 f_3=\xi_2 e,$ &  $f_2 f_3=\xi_3 e.$
\end{tabular}
\end{center}
In this case, from $J(e, e, e, f_1)=0$, $J(e, e, e, f_2)=0$, and $J(e, e, e, f_3)=0$ we get $\mu_i=0, \  i=\overline{1,3}.$ 

If $(\xi_1, \xi_2, \xi_3)\neq(0, 0, 0)$, then by changing $e'=\xi_1e, \ f_3'=\xi_1f_3-\xi_2f_2+\xi_3f_1$ we obtian the Jordan superalgebra:

\begin{center}
  $e'^2=0, \  f_1f_2= e'.$
\end{center}
We denote this superalgebra by $\bf J_1$.

\noindent
\textbf{ii)} Let $L_{e} \simeq \left(\begin{array}{ccc}
\mu_1 & 1 & 0\\
0 & \mu_1 & 0\\
0 & 0 & \mu_3
 \end{array}\right),$ then the rule of multiplication can be written as follows:
 \begin{center}
 \begin{tabular}{llllll}
  $e f_1= \mu_1 f_1 + f_2 ,$ & $e f_2= \mu_1 f_2,$ & $e f_3= \mu_3 f_3,$ & $f_1 f_2=\xi_1 e ,$ & $f_1 f_3=\xi_2 e,$ & $f_2 f_3=\xi_3 e.$
\end{tabular}
\end{center}

In this case, from $J(e, e,e,f_2)=0$ and $J(e, e,e,f_3)=0$ we get $2 \mu_1^3=0, \  2 \mu_3^3=0$, respectively. So $\mu_1=\mu_3=0.$
When $(\xi_1, \xi_2, \xi_3)=(0, 0, 0)$ we get the superalgebra
$$ef_1=f_2,$$
which we denote by $\bf J_2.$
Hence, if $(\xi_1, \xi_2, \xi_3)\neq(0, 0, 0)$, then we can change the basis as follows:

\begin{center}
 \begin{tabular}{ll}
  $f_1'=a_1 f_1+ a_2f_2+a_3f_3,$ & \  $f_1'f_2'=a_1^2 \xi_1 e-a_1a_3\xi_3e=\xi_1'e,$ \\
  $f_2'=a_1f_2 ,$ & \  $f_1'f_3'=a_1b_1\xi_1e+a_1b_2\xi_2e+a_2b_2\xi_3e-a_3b_1\xi_3e=\xi_2'e,$\\
  $f_3'=b_1f_2+b_2f_3,$ & \  $ f_2'f_3'=a_1b_2\xi_3e.$ 
\end{tabular}
\end{center}
We proceed as follows:
\begin{enumerate}
    \item[a)] If $\xi_3=0$ then $\xi_1'=a_1^2\xi_1, \  \xi_2'=a_1b_1\xi_1+a_1b_2\xi_2.$
    \begin{enumerate}
        \item[a.1)] If $\xi_1 \neq 0$ then by choosing $b_1=-\frac{b_2 \xi_2}{\xi_1}, \  a_1=\frac{1}{\xi_1}$ we get $\xi_2'=0$ and $\xi_1'=1$ which gives the superalgebra $\bf J_3$:
        $$e^2=0, \  ef_1'=f_2', \  f_1'f_2'=e.$$
    
        \item[a.2)] If $\xi_1=0$ then $\xi_2'=a_1b_2\xi_2$. In this case by choosing $a_1b_2=\frac{1}{\xi_2}$ we get the superalgebra $\bf J_4$:
        $$e^2=0, \  ef_1'=f_2', \  f_1'f_3'=e.$$
    \end{enumerate}
    \item[b)] If $\xi_3 \neq 0$ then by choosing $a_3=\frac{a_1 \xi_1}{\xi_3}, \  b_2=\frac{1}{a_1 \xi_3}, \  a_2=-\frac{a_1 \xi_2}{\xi_3}$ and taking $b_2=1$ for simplicity we get the superalgebra $\bf J_5$: $e^2=0, \  ef_1'=f_2', \  f_2'f_3'=e.$
\end{enumerate}

\noindent
\textbf{iii)} Let $L_{e} \simeq \left(\begin{array}{ccc}
\mu_1 & 1 & 0\\
0 & \mu_1 & 1\\
0 & 0 & \mu_1
 \end{array}\right),$ then the rule of multiplication can be written as follows:
 \begin{center}
 \begin{tabular}{llllll}
  $e f_1= \mu_1 f_1 + f_2,$ & $e f_2= \mu_1 f_2 + f_3,$ & $e f_3= \mu_1 f_3,$ & $f_1 f_2=\xi_1 e,$ & $f_1 f_3=\xi_2 e,$ & $f_2 f_3=\xi_3 e.$
\end{tabular}
\end{center}
In this case from $J(e, e,e,f_1)=0$ we get $\mu_1=0$. Moreover, from $J(e, e,f_1,f_2)=0$ and $J(e, f_1,f_1,f_2)=0$ we get $\xi_3=0$ and $\xi_2=\xi_1=0$, respectively. As a result, we have the superalgebra $\bf J_6:$
\begin{center}

  $e f_1= f_2,  \  e f_2= f_3.$ 
\end{center}

\noindent
\underline{Let $\jor_0 \cong \mathcal{U}_1$. }
\medskip

\noindent
\textbf{i)} Let $L_{e} \simeq \left(\begin{array}{ccc}
 \mu_1 & 0 & 0\\
 0 & \mu_2 & 0\\
 0 & 0 & \mu_3
 \end{array}\right),$ then the rule of multiplication can be written as follows:
 \begin{center}
 \begin{tabular}{llllll}
  $e f_1= \mu_1 f_1,$ & $e f_2= \mu_2 f_2,$ & $e f_3= \mu_3 f_3,$ & $f_1 f_2=\xi_1 e,$ & $f_1 f_3=\xi_2 e,$ & $f_2 f_3=\xi_3 e.$
\end{tabular}
\end{center}

From $J(e, e,e,f_1)=0,$ $J(e, e,e,f_2)=0,$ and $J(e, e,e,f_3)=0$ we obtain equations below
\begin{center}
\begin{tabular}{cc}
  $(\mu_i-1)\mu_i(2\mu_i-1)=0,$ & $i=\overline{1,3}.$
\end{tabular}
\end{center}
Up to permutation of $f_1, f_2$ and $f_3$ we have ten possibilities:

$(\mu_1,\mu_2,\mu_3) \in \{(0,0,0), (0,0,\frac{1}{2}), (0,0,1), (0,\frac{1}{2}, \frac{1}{2}), (0,\frac{1}{2},1), (0,1,1), (\frac{1}{2}, \frac{1}{2}, \frac{1}{2}), (\frac{1}{2}, \frac{1}{2}, 1), (\frac{1}{2},1,1), (1,1,1) \}.$

\begin{enumerate}
    \item[1.] $(\mu_1,\mu_2,\mu_3)=(0,0,0)$. 
    In this case, we have the following results
    \begin{center}
        \begin{tabular}{ccc}
         $J(e,e,f_1,f_2)=0 \  \Rightarrow \  \xi_1=0,$ &
         $J(e,e,f_1,f_3)=0 \  \Rightarrow \  \xi_2=0,$ &
         $J(e,e,f_2,f_3)=0 \  \Rightarrow \  \xi_3=0.$
        \end{tabular}
    \end{center}
    Hence, the obtained superalgebra is $\bf J_7.$

    \item[2.] $(\mu_1,\mu_2,\mu_3)=(0,0,\frac{1}{2})$.
    In this case, we have the following results
    \begin{center}
        \begin{tabular}{ccc}
         $J(e,e,f_1,f_2)=0 \  \Rightarrow \  \xi_1=0,$ &
         $J(e,e,f_3,f_1)=0 \  \Rightarrow \  \xi_2=0,$ &
         $J(e,e,f_3,f_2)=0 \  \Rightarrow \  \xi_3=0.$
        \end{tabular}
    \end{center}
    Hence, the the obtained superalgebra is $\bf J_8$.
    
    \item[3.] $(\mu_1,\mu_2,\mu_3)=(0,0,1)$.
    In this case, we have the following results
    \begin{center}
        \begin{tabular}{ccc}
         $J(e,e,f_1,f_2)=0 \  \Rightarrow \  \xi_1=0,$ &
         $J(e,e,f_1,f_3)=0 \  \Rightarrow \  \xi_2=0,$ &
         $J(e,e,f_2,f_3)=0 \  \Rightarrow \  \xi_3=0.$ 
        \end{tabular}
    \end{center}
    Hence, the the obtained superalgebra is $\bf J_9$.
    
    \item[4.] $(\mu_1,\mu_2,\mu_3)=(0,\frac{1}{2},\frac{1}{2})$.
    In this case, we have the following results
    \begin{center}
        \begin{tabular}{cc}
         $J(e,e,f_1,f_2)=0 \  \Rightarrow \  \xi_1=0,$ &
         $J(e,e,f_1,f_3)=0 \  \Rightarrow \  \xi_2=0.$ 
        \end{tabular}
    \end{center}
    Hence, the obtained superalgebras are $\bf J_{10}$
    and $\bf J_{11}$.
    
    \item[5.] $(\mu_1,\mu_2,\mu_3)=(0,\frac{1}{2},1)$
    In this case, we have the following results
    \begin{center}
        \begin{tabular}{ccc}
         $J(e,e,f_2,f_1)=0 \  \Rightarrow \  \xi_1=0,$ &
         $J(e,e,f_1,f_3)=0 \  \Rightarrow \  \xi_2=0,$ &
         $J(e,e,f_2,f_3)=0 \  \Rightarrow \  \xi_3=0.$
        \end{tabular}
    \end{center}
    Hence, the obtained superalgebra is $\bf J_{12}.$
    
    \item[6.] $(\mu_1,\mu_2,\mu_3)=(0,1,1)$.
    In this case, we have the following results
    \begin{center}
        \begin{tabular}{cc}
         $J(e,e,f_1,f_2)=0 \  \Rightarrow \  \xi_1=0,$ &
         $J(e,e,f_1,f_3)=0 \  \Rightarrow \  \xi_2=0.$
        \end{tabular}
    \end{center}
    Hence, the obtained superalgebras are $\bf J_{13}$
    and $\bf J_{14}.$
    
    \item[7.] $(\mu_1,\mu_2,\mu_3)=(0,\frac{1}{2},\frac{1}{2})$.
    In this case, we have the following results
    \begin{center}
        \begin{tabular}{c}
         $J(f_1,f_2,e,f_3)=0 \  \Rightarrow \  \xi_1=\xi_2=\xi_3=0.$ 
        \end{tabular}
    \end{center}
    Hence, the obtained superalgebra is $\bf J_{15}.$

    \item[8.] $(\mu_1,\mu_2,\mu_3)=(\frac{1}{2},\frac{1}{2},1)$.
    In this case, we have the following results
    \begin{center}
        \begin{tabular}{ll}
         $J(e,f_1,f_2,f_3)=0 \  \Rightarrow \  \xi_1=\xi_3=0,$ &
         $J(e,e,f_1,f_3)=0 \  \Rightarrow \  \xi_2=0.$
        \end{tabular}
    \end{center}
    Hence, the obtained superalgebra is $\bf J_{16}.$

    \item[9.] $(\mu_1,\mu_2,\mu_3)=(\frac{1}{2},1,1)$.
    In this case, we have the following results
    \begin{center}
        \begin{tabular}{ll}
         $J(e,f_1,f_2,f_3)=0 \  \Rightarrow \  \xi_1=\xi_3=0,$ &
         $J(e,e,f_1,f_3)=0 \  \Rightarrow \  \xi_2=0.$
        \end{tabular}
    \end{center}
    Hence, the obtained superalgebra is $\bf J_{17}.$ 

    \item[10.] $(\mu_1,\mu_2,\mu_3)=(1,1,1)$.
    In this case, the multiplication rules in the obtained superalgebra are $e^2=e,\quad e f_1=f_1, \  e f_2=f_2, \  e f_3=f_3, \  f_1 f_2=\xi_1 e, \  f_1f_3=\xi_2 e, \  f_2 f_3=\xi_3 e$. 
    When $(\xi_1, \xi_2,\xi_3)=(0,0,0)$ we get the superalgebra $\bf J_{18}.$
    
    However, when $(\xi_1, \xi_2,\xi_3)\neq(0,0,0)$ we can assume that $\xi_1\neq 0$ and by changing the basis as follows

    \begin{center}
        \begin{tabular}{lll}
            $f_1'= \frac{1}{\xi_1} f_1 + f_2,$ &
            $f_2'=f_2,$ &
            $f_3'=\xi_1 f_3-\xi_2 f_2 +\xi_3f_1,$
        \end{tabular}
    \end{center}
    we get the superalgebra $\bf J_{19}.$
\end{enumerate}
\noindent
\textbf{ii)} Let $L_{e} \simeq \left(\begin{array}{ccc}
\mu_1 & 1 & 0\\
0 & \mu_1 & 0\\
0 & 0 & \mu_3
 \end{array}\right),$ then the rule of multiplication can be written as follows:
 \begin{center}
 \begin{tabular}{llllll}
  $e f_1= \mu_1 f_1 + f_2,$ & $e f_2= \mu_1 f_2,$ & $e f_3= \mu_3 f_3,$ & $f_1 f_2=\xi_1 e,$ &
   $f_1 f_3=\xi_2 e,$ &
   $f_2 f_3=\xi_3 e.$
\end{tabular}
\end{center}

\noindent
However, the following two equations
    \begin{center}
    \begin{tabular}{ll}
         $J(e,e,e,f_1)=0 \  \Rightarrow \  1-6\mu_1+6\mu_1^2=0,$ &
         $J(e,e,e,f_2)=0 \  \Rightarrow \  (\mu_1-1)\mu_1(2\mu_1-1)=0,$
    \end{tabular}
    \end{center}
which can not be satisfied at the same time, give a contradiction, thereby no superalgebras can be found in this subcase.

\noindent
\textbf{iii)} Let $L_{e} \simeq \left(\begin{array}{ccc}
\mu_1 & 1 & 0\\
0 & \mu_1 & 1\\
0 & 0 & \mu_1
 \end{array}\right),$ then the rule of multiplication can be written as follows:
 \begin{center}
 \begin{tabular}{llllll}
  $e f_1= \mu_1 f_1 + f_2,$ & $e f_2= \mu_1 f_2 + f_3,$ & $e f_3= \mu_1 f_3,$ & $f_1 f_2=\xi_1 e,$ & $f_1 f_3=\xi_2 e,$ & $f_2 f_3=\xi_3 e.$
\end{tabular}
\end{center}

\noindent
However, the following two equations
    \begin{center}
    \begin{tabular}{ll}
         $J(e,e,e,f_2)=0 \  \Rightarrow \  1-6\mu_1+6\mu_1^2=0,$ &
         $J(e,e,e,f_3)=0 \  \Rightarrow \  (\mu_1-1)\mu_1(2\mu_1-1)=0,$
    \end{tabular}
    \end{center}
which can not be satisfied at the same time, give a contradiction, thereby no superalgebras can be found in this subcase.
\end{proof}

\begin{Th}
Up to isomorphism there are 71 Jordan superalgebras of type $(2,2)$, which are presented below with some additional information:

\begin{center}
\renewcommand{\arraystretch}{1.2}
\begin{longtable}{l|c|l|l}
     \textnumero & Orbit & Multiplication rules & Decomposition \\
     \hline
     $\bf \mathcal{J}_{1}$ & 12 & $e_1^2=e_1, \ e_2^2=e_2$ & $\mathcal{U}_1 \oplus \mathcal{U}_1 \oplus \mathcal{S}_1^1 \oplus \mathcal{S}_1^1$\\
     $\bf \mathcal{J}_{2}$ & 13 & $e_1^2=e_1, \ e_2^2=e_2, \ e_2f_2=f_2$ & $\mathcal{U}_1 \oplus \mathcal{S}_2^2 \oplus \mathcal \mathcal{S}_1^1$ \\
     $\bf \mathcal{J}_{3}$ & 13 & $e_1^2=e_1, \ e_2^2=e_2, \ e_2f_2=\frac{1}{2}f_2$  & $\mathcal{U}_1 \oplus \mathcal{S}_1^2 \oplus \mathcal \mathcal{S}_1^1$ \\

     $\bf \mathcal{J}_{4}$ & 12 & $e_1^2=e_1, \ e_2^2=e_2, \ e_2f_1=f_1, \ e_2f_2=f_2$  & $\mathcal{U}_1 \oplus \mathcal{S}_6^3$ \\
     $\bf \mathcal{J}_{5}$ & 13 & $e_1^2=e_1, \ e_2^2=e_2, \ e_2f_1=f_1, \ e_2f_2=\frac{1}{2}f_2$ & $\mathcal{U}_1 \oplus \mathcal{S}_4^3 $\\
     $\bf \mathcal{J}_{6}$ & 13 & $e_1^2=e_1, \ e_2^2=e_2, \ e_2f_1=f_1, \ e_2f_2=f_2, \ f_1f_2=e_2$ & $\mathcal{U}_1 \oplus \mathcal{S}_8^3 $\\

     $\bf \mathcal{J}_{7}$ & 10 & $e_1^2=e_1, \ e_2^2=e_2, \ e_2f_1=\frac{1}{2}f_1, \ e_2f_2=\frac{1}{2}f_2$ & $\mathcal{U}_1 \oplus \mathcal{S}_5^3$\\
     $\bf \mathcal{J}_{8}$ & 11 & $e_1^2=e_1, \ e_2^2=e_2, \ e_2f_1=\frac{1}{2}f_1, \ e_2f_2=\frac{1}{2}f_2$, \ $f_1f_2=e_2$ & $\mathcal{U}_1 \oplus \mathcal{S}_7^3 $\\

     $\bf \mathcal{J}_{9}$ & 13 & $e_1^2=e_1, \ e_2^2=e_2, \ e_1f_2=\frac{1}{2}f_2, \ e_2f_2=\frac{1}{2}f_2$ & $\mathcal{S}_{13}^3 \oplus \mathcal{S}_1^1 $\\
     $\bf \mathcal{J}_{10}$ & 12 & $e_1^2=e_1, \ e_2^2=e_2, \ e_1f_2=\frac{1}{2}f_2, \ e_2f_1=f_1$ & $\mathcal{S}_1^2 \oplus \mathcal{S}_2^2 $\\
     $\bf \mathcal{J}_{11}$ & 13 & $e_1^2=e_1, \ e_2^2=e_2, \ e_1f_2=\frac{1}{2}f_2, \ e_2f_1=f_1, \ e_2f_2=\frac{1}{2}f_2$ & Indecomposable\\
     $\bf \mathcal{J}_{12}$ & 12 & $e_1^2=e_1, \ e_2^2=e_2, \ e_1f_2=\frac{1}{2}f_2, \ e_2f_1=\frac{1}{2}f_1$ & $\mathcal{S}_1^2 \oplus \mathcal{S}_1^2 $\\
     $\bf \mathcal{J}_{13}$ & 12 & $e_1^2=e_1, \ e_2^2=e_2, \ e_1f_2=\frac{1}{2}f_2, \ e_2f_1=\frac{1}{2}f_1, \ e_2f_2=\frac{1}{2}f_2$ & Indecomposable \\

     $\bf \mathcal{J}_{14}$ & 12  & $e_1^2=e_1, \ e_2^2=e_2, \ e_1f_2=f_2, \ e_2f_1=f_1$  & $\mathcal{S}_2^2 \oplus \mathcal{S}_2^2 $\\
     $\bf \mathcal{J}_{15}$ & 10 & $e_1^2=e_1, \ e_2^2=e_2, \ e_1f_1=\frac{1}{2}f_1,  \ e_1f_2=\frac{1}{2}f_2, e_2f_1=\frac{1}{2}f_1, \  e_2f_2=\frac{1}{2}f_2$ & Indecomposable \\
     $\bf \mathcal{J}_{16}^t$ & 12 & $e_1^2=e_1, \ e_2^2=e_2, \ e_1f_1=\frac{1}{2}f_1,  \ e_1f_2=\frac{1}{2}f_2, $ & Indecomposable \\
      &  & $e_2f_1=\frac{1}{2}f_1, \  e_2f_2=\frac{1}{2}f_2, \ f_1f_2=e_1+te_2$ & \\
     \hline
     $\bf \mathcal{J}_{17}$ & 6 & $f_1f_2=e_1$ & $\U_2 \oplus \S_2^3$ \\
     $\bf \mathcal{J}_{18}$ & 7 & $e_2f_1=f_2$ & $\U_2 \oplus \S_3^3$\\
     $\bf \mathcal{J}_{19}$ & 9 & $e_2f_1=f_2, \ f_1f_2=e_1$ & Indecomposable \\
     $\bf \mathcal{J}_{20}$ & 10 & $e_2f_1=f_2, \ f_1f_2=e_2$ & $\U_2 \oplus \S_1^3$ \\
      \hline
     $\bf \mathcal{J}_{21}$ & 7 & $e_1^2=e_1$ & $\mathcal{U}_1 \oplus \mathcal{U}_2 \oplus \mathcal{S}_1^1 \oplus \mathcal{S}_1^1$\\
     $\bf \mathcal{J}_{22}$ & 10 & $e_1^2=e_1, \ f_1f_2=e_2$ & $\mathcal{U}_1 \oplus \mathcal{S}_2^3$ \\
     $\bf \mathcal{J}_{23}$ & 11 & $e_1^2=e_1, \ e_2f_1=f_2$  & $\mathcal{U}_1 \oplus \mathcal{S}_3^3$ \\

     $\bf \mathcal{J}_{24}$ & 12 & $e_1^2=e_1, \ e_2f_1=f_2, \ f_1f_2=e_2$  & $\mathcal{U}_1 \oplus \mathcal{S}_1^3$ \\
     $\bf \mathcal{J}_{25}$ & 10 & $e_1^2=e_1, \ e_1f_2=\frac{1}{2}f_2$ & $\mathcal{U}_2 \oplus \mathcal{S}_1^2 \oplus \S_1^1 $\\
     $\bf \mathcal{J}_{26}$ & 10 & $e_1^2=e_1, \ e_1f_2=f_2$ & $\mathcal{U}_2 \oplus \mathcal{S}_2^2 \oplus \S_1^1 $\\
     $\bf \mathcal{J}_{27}$ & 9 & $e_1^2=e_1, \ e_1f_1=\frac{1}{2}f_1, \ e_1f_2=\frac{1}{2}f_2$ & $\mathcal{U}_2 \oplus \mathcal{S}_5^3 $\\

     $\bf \mathcal{J}_{28}$ & 10 & $e_1^2=e_1, \ e_1f_1=\frac{1}{2}f_1, \ e_1f_2=\frac{1}{2}f_2, \ f_1f_2=e_2$ & Indecomposable\\

     $\bf \mathcal{J}_{29}$ &10 & $e_1^2=e_1, \ e_1f_1=\frac{1}{2}f_1, \ e_1f_2=\frac{1}{2}f_2, \ f_1f_2=e_1$ & $\U_2 \oplus \S_7^3$\\
     $\bf \mathcal{J}_{30}$ & 11 & $e_1^2=e_1, \ e_1f_1=\frac{1}{2}f_1, \ e_1f_2=\frac{1}{2}f_2, \ f_1f_2=e_1+e_2$ & Indecomposable\\
     $\bf \mathcal{J}_{31}$ & 11 & $e_1^2=e_1, \ e_1f_1=\frac{1}{2}f_1, \ e_1f_2=\frac{1}{2}f_2, \ e_2f_1=f_2$ & Indecomposable\\
     $\bf \mathcal{J}_{32}$ & 12 & $e_1^2=e_1, \ e_1f_1=\frac{1}{2}f_1, \ e_1f_2=\frac{1}{2}f_2, \ e_2f_1=f_2, \ f_1f_2=e_2$ & Indecomposable\\
     $\bf \mathcal{J}_{33}$ & 12 & $e_1^2=e_1, \ e_1f_1=\frac{1}{2}f_1, \ e_1f_2=f_2$ & $\U_2 \oplus \S_4^3$ \\

     $\bf \mathcal{J}_{34}$ & 11 & $e_1^2=e_1, \ e_1f_1=f_1, \ e_1f_2=f_2$  & $\U_2 \oplus \S_6^3$\\
     $\bf \mathcal{J}_{35}$ & 12 & $e_1^2=e_1, \ e_1f_1=f_1, \ e_1f_2=f_2, \ f_1f_2=e_1$ & $\U_2 \oplus \S_8^3$\\
     \hline
     $\bf \mathcal{J}_{36}  $ & 11 & $e_1^2=e_1, \ e_1e_2=e_2$ & $\B_1 \oplus \S_1^1 \oplus \S_1^1$\\
     $\bf \mathcal{J}_{37}  $ & 12 & $e_1^2=e_1, \ e_1e_2=e_2, \ e_1f_2=\frac{1}{2}f_2$ & $\S_9^3 \oplus \S_1^1$\\
     $\bf \mathcal{J}_{38}  $ & 11 & $e_1^2=e_1, \ e_1e_2=e_2, \ e_1f_2=f_2$  & $\S_{10}^3 \oplus \S_1^1$\\
     $\bf \mathcal{J}_{39}  $ & 9 & $e_1^2=e_1, \ e_1e_2=e_2, \ e_1f_1=\frac{1}{2}f_1, \ e_1f_2=\frac{1}{2}f_2$ & Indecomposable \\
     $\bf \mathcal{J}_{40}$ & 10 & $e_1^2=e_1, \ e_1e_2=e_2, \ e_1f_1=\frac{1}{2}f_1, \ e_1f_2=\frac{1}{2}f_2, \ f_1f_2=e_2$  & Indecomposable\\
     $\bf \mathcal{J}_{41}  $ & 11 & $e_1^2=e_1, \ e_1e_2=e_2, \ e_1f_1=\frac{1}{2}f_1, \ e_1f_2=\frac{1}{2}f_2, \ e_2f_1=f_2$  & Indecomposable\\
     $\bf \mathcal{J}_{42}$ & 12 & $e_1^2=e_1, \ e_1e_2=e_2, \ e_1f_1=\frac{1}{2}f_1, \ e_1f_2=\frac{1}{2}f_2, \ e_2f_1=f_2, \ f_1f_2=e_2$  & Indecomposable\\
     $\bf \mathcal{J}_{43}  $ & 10 & $e_1^2=e_1, \ e_1e_2=e_2, \ e_1f_1=\frac{1}{2}f_1, \ e_1f_2=f_2$ & Indecomposable\\
     $\bf \mathcal{J}_{44}  $ & 7 & $e_1^2=e_1, \ e_1e_2=e_2, \ e_1f_1=f_1, \ e_1f_2=f_2$ & Indecomposable\\
     $\bf \mathcal{J}_{45}$ & 8 & $e_1^2=e_1, \ e_1e_2=e_2, \ e_1f_1=f_1, \ e_1f_2=f_2, \ f_1f_2=e_2$ & Indecomposable\\
     $\bf \mathcal{J}_{46}$ & 10 & $e_1^2=e_1, \ e_1e_2=e_2, \ e_1f_1=f_1, \ e_1f_2=f_2, \ f_1f_2=e_1$ & Indecomposable\\
     $\bf \mathcal{J}_{47}$ & 11 & $e_1^2=e_1, \ e_1e_2=e_2, \ e_1f_1=f_1, \ e_1f_2=f_2, \ f_1f_2=e_1+e_2$ & Indecomposable\\
     $\bf \mathcal{J}_{48}  $ & 10 & $e_1^2=e_1, \ e_1e_2=e_2, \ e_1f_1=f_1, \ e_1f_2=f_2, \ e_2f_1=f_2$ & Indecomposable\\
     $\bf \mathcal{J}_{49}$ & 11 & $e_1^2=e_1, \ e_1e_2=e_2, \ e_1f_1=f_1, \ e_1f_2=f_2, \ e_2f_1=f_2, \ f_1f_2=e_2$ & Indecomposable\\
     \hline
     $\bf \mathcal{J}_{50}  $ & 10 & $e_1^2=e_1, \ e_1e_2=\frac{1}{2}e_2$ & $\B_2 \oplus \S_1^1 \oplus \S_1^1$\\
     $\bf \mathcal{J}_{51}  $ & 9 & $e_1^2=e_1, \ e_1e_2=\frac{1}{2}e_2, \ e_1f_2=\frac{1}{2}f_2$ & $\S_{11}^3 \oplus \S_1^1$\\
     $\bf \mathcal{J}_{52}$ & 11 & $e_1^2=e_1, \ e_1e_2=\frac{1}{2}e_2, \ e_1f_2=\frac{1}{2}f_2, \ f_1f_2=e_2$  & Indecomposable\\
     $\bf \mathcal{J}_{53}  $ & 11 & $e_1^2=e_1, \ e_1e_2=\frac{1}{2}e_2, \ e_1f_2=\frac{1}{2}f_2, \ e_2f_2=f_1$  & Indecomposable\\
     $\bf \mathcal{J}_{54}$ & 12 & $e_1^2=e_1, \ e_1e_2=\frac{1}{2}e_2, \ e_1f_2=\frac{1}{2}f_2, \ e_2f_2=f_1, \ f_1f_2=e_2$  & Indecomposable\\
     $\bf \mathcal{J}_{55}  $ & 11 & $e_1^2=e_1, \ e_1e_2=\frac{1}{2}e_2, \ e_1f_2=\frac{1}{2}f_2, \ e_2f_1=f_2$  & Indecomposable\\
     $\bf \mathcal{J}_{56}$ & 13 & $e_1^2=e_1, \ e_1e_2=\frac{1}{2}e_2, \ e_1f_2=\frac{1}{2}f_2, \ e_2f_1=f_2, \ f_1f_2=e_2$  & Indecomposable\\
     $\bf \mathcal{J}_{57}  $ & 12 & $e_1^2=e_1, \ e_1e_2=\frac{1}{2}e_2, \ e_1f_2=f_2$ & $\S_{12}^3\oplus \S_1^1$ \\
     $\bf \mathcal{J}_{58}  $ & 4 & $e_1^2=e_1, \ e_1e_2=\frac{1}{2}e_2, \ e_1f_1=\frac{1}{2}f_1, \ e_1f_2=\frac{1}{2}f_2$ & Indecomposable\\
     $\bf \mathcal{J}_{59}  $ & 9 & $e_1^2=e_1, \ e_1e_2=\frac{1}{2}e_2, \ e_1f_1=\frac{1}{2}f_1, \ e_1f_2=f_2$ & Indecomposable\\
     $\bf \mathcal{J}_{60}$ & 10 & $e_1^2=e_1, \ e_1e_2=\frac{1}{2}e_2, \ e_1f_1=\frac{1}{2}f_1, \ e_1f_2=f_2, \ f_1f_2=e_2$ & Indecomposable\\
     $\bf \mathcal{J}_{61}  $ & 10 & $e_1^2=e_1, \ e_1e_2=\frac{1}{2}e_2, \ e_1f_1=\frac{1}{2}f_1, \ e_1f_2=f_2, \ e_2f_2=f_1$ & Indecomposable\\
     $\bf \mathcal{J}_{62}$ & 11 & $e_1^2=e_1, \ e_1e_2=\frac{1}{2}e_2, \ e_1f_1=\frac{1}{2}f_1, \ e_1f_2=f_2, \ e_2f_2=f_1, \ f_1f_2=e_2$ & Indecomposable\\
     $\bf \mathcal{J}_{63}  $ & 11 & $e_1^2=e_1, \ e_1e_2=\frac{1}{2}e_2, \ e_1f_1=\frac{1}{2}f_1, \ e_1f_2=f_2, \ e_2f_1=f_2$ & Indecomposable\\
     $\bf \mathcal{J}_{64}$ & 12 & $e_1^2=e_1, \ e_1e_2=\frac{1}{2}e_2, \ e_1f_1=\frac{1}{2}f_1, \ e_1f_2=f_2, \ e_2f_1=f_2, \ f_1f_2=e_2$ & Indecomposable\\
     $\bf \mathcal{J}_{65}  $ & 10 & $e_1^2=e_1, \ e_1e_2=\frac{1}{2}e_2, \ e_1f_1=f_1, \ e_1f_2=f_2$ & Indecomposable\\
     \hline
     $\bf \mathcal{J}_{66}  $ & 6 & $e_1^2=e_2$ & $\B_3 \oplus \S_1^1 \oplus \S_1^1$\\
     $\bf \mathcal{J}_{67}$ & 7 & $e_1^2=e_2, \ f_1f_2=e_2$ & Indecomposable\\
     $\bf \mathcal{J}_{68}$ & 10 & $e_1^2=e_2, \ f_1f_2=e_1$ & Indecomposable\\
     $\bf \mathcal{J}_{69}  $ & 11 & $e_1^2=e_2, \ e_2f_1=f_2$ & Indecomposable\\
     $\bf \mathcal{J}_{70}  $ & 8 & $e_1^2=e_2, \ e_1f_1=f_2$ & Indecomposable\\
     $\bf \mathcal{J}_{71}$ & 10 & $e_1^2=e_2, \ e_1f_1=f_2, \ f_1f_2=e_2$ & Indecomposable 
\end{longtable}
\end{center}
\end{Th}

\begin{proof}
\noindent
\underline{Let $\jor_0 \cong \mathcal{U}_1 \oplus \mathcal{U}_1$}.
Here we are looking for Jordan superalgebras such that $\jor=(\mathbb{F}e_1 +\mathbb{F}e_2 )+(\mathbb{F}f_1 +\mathbb{F}f_2)$ with multiplication rules

\begin{center}
    \begin{tabular}{l}
       $e_1^2=e_1,$ \ $e_2^2 =e_2,$ \ $e_1 f_1=\alpha_1 f_1 +\alpha_2 f_2,$ \ $e_1 f_2=\alpha_3 f_1 +\alpha_4 f_2,$ \\
       $e_2 f_1=\beta_1 f_1 +\beta_2 f_2,$ \ $f_1 f_2 = \xi_1 e_1+\xi_2 e_2,$ \ $e_2 f_2=\beta_3 f_1 +\beta_4 f_2.$
    \end{tabular}
\end{center}

For the action of the operator $L_{e_{1}}$ on $\jor_1$ we can write the following matrix
\begin{center}
    \begin{tabular}{c}
        $\begin{pmatrix}
        \alpha_1 & \alpha_2 \\  
        \alpha_3 & \alpha_4 
        \end{pmatrix}.$ 
    \end{tabular}
\end{center}

\noindent
However, it is easy to prove that, by using a simple change of basis, the matrix of $L_{e_{1}}$ have one of the following forms:
\begin{center}
    \begin{tabular}{cc}
        $\begin{pmatrix}
        \mu_1 & 0 \\  
        0 & \mu_2 
        \end{pmatrix},$ & 
        $\begin{pmatrix}
        \mu_1 & 1 \\  
        0 & \mu_1 
        \end{pmatrix}.$ 
    \end{tabular}
\end{center}

\noindent
\textbf{i)} Let $L_{e_{1}} \simeq       \begin{pmatrix}
        \mu_1 & 0 \\  
        0 & \mu_2 
    \end{pmatrix}$ then the rule of multiplication can be written as follows:
\begin{center}
    \begin{tabular}{l}
       $e_1^2=e_1,$ \ $e_2^2 =e_2,$ \ $e_1 f_1=\mu_1 f_1,$ \ $e_1 f_2=\mu_2 f_2,$ \\ 
       $e_2 f_1=\beta_1 f_1 +\beta_2 f_2,$ \
       $f_1 f_2 = \xi_1 e_1+\xi_2 e_2, $ \ $e_2 f_2=\beta_3 f_1 +\beta_4 f_2.$
    \end{tabular}
\end{center}
From $J(e_1, e_1,e_1,f_1)=0,$ $J(e_1, e_1,e_1,f_2)=0,$  we obtain equations below
\begin{center}
    \begin{tabular}{cc}
  $(\mu_1-1)\mu_1(2\mu_1-1)=0,$ & $i=\overline{1,2}.$
\end{tabular}
\end{center}
Up to permutation of $f_1$ and $f_2$ we have six possibilities: 
\begin{center}
$(\mu_1,\mu_2)\in \{(0,0), (0,\frac{1}{2}), (0,1), (\frac{1}{2}, \frac{1}{2}), (\frac{1}{2},1), (1,1)\}.$
\end{center}

\begin{enumerate}
\item[1.] $(\mu_1,\mu_2)=(0,0)$
   In this case, we can consider the action of $e_2.$
    By using a simple change of basis, the matrix of $L_{e_{2}}$ have one of the following forms:
    \begin{center}
        \begin{tabular}{cc}
            $\begin{pmatrix}
            \tau_1 & 0 \\  
            0 & \tau_2 
            \end{pmatrix},$ & 
            $\begin{pmatrix}
            \tau_1 & 1 \\  
            0 & \tau_1 
            \end{pmatrix}.$ 
        \end{tabular}
    \end{center}

    \noindent
    When $L_{e_{2}} \simeq       \begin{pmatrix}
        \tau_1 & 0 \\  
        0 & \tau_2 
    \end{pmatrix}$ we have $\xi_1=0, \ (\tau_1-1)\tau_1(2\tau_1-1)=0, \ (\tau_2-1)\tau_2(2\tau_2-1)=0$ from $J(e_1,e_1,f_1,f_2)=0, \ J(e_2,e_2,e_2,f_1)=0$ and $J(e_2,e_2,e_2,f_2)=0$, respectively.
    
    \begin{itemize}
        \item If $\tau_1=0$ then from 
            \begin{center}
                \begin{tabular}{ll}
                $J(e_2,e_2,f_1,f_2)=0 \  \Rightarrow \  \xi_2(2\tau_2-1)=0,$ &
                $J(e_2,e_2,f_2,f_1)=0 \  \Rightarrow \  \xi_2(\tau_2-1)=0.$  
                \end{tabular}
            \end{center}
            we get $\xi_2=0$ and $\tau_2 \in {0,1,\frac{1}{2}}$, which gives superalgebras $\mathcal{J}_{1}, \mathcal{J}_{2}$ and $\mathcal{J}_{3}.$
        \item If $\tau_1=1$ then we have
        
            \begin{center}
                \begin{tabular}{c}
                $J(e_2,e_2,f_2,f_1)=0 \  \Rightarrow \  \xi_2(\tau_2-1)=0.$ \\ 
                \end{tabular}
            \end{center}
            So, either $\xi_2=0$ and we have  Jordan superalgebras $\mathcal{J}_{4}$ and $\mathcal{J}_{5}$ with $\tau_2 \in \{1,\frac{1}{2}\}$, or $\xi_2 \neq 0$ and $\tau_2=1$ which gives us the Jordan superalgebra $\mathcal{J}_{6}$.
            
        \item If $\tau_1=\frac{1}{2}$ then from 
            \begin{center}
                \begin{tabular}{c}
                $J(e_2,e_2,f_1,f_2)=0 \  \Rightarrow \  \xi_2(2\tau_2-1)=0.$ \\ 
                \end{tabular}
            \end{center}
            So, either $\xi_2=0$ and we have a Jordan superalgebra $\mathcal{J}_{7}$ with $\tau_2=\frac{1}{2}$ ( $\tau_2 \in \{0,1\}$ repeat the previous superalgebras), or $\xi_2 \neq 0$ and $\tau_2=\frac{1}{2}$ which gives us the Jordan superalgebra $\mathcal{J}_{8}$.
    \end{itemize}

    \noindent
    When $L_{e_{2}} \simeq       \begin{pmatrix}
        \tau_1 & 1 \\  
        0 & \tau_1 
    \end{pmatrix}$ we have:
    \begin{center}
        \begin{tabular}{ll}
         $J(e_2,e_2,e_2,f_1)=0 \  \Rightarrow \  (1-6\tau_1+6\tau_1^2)=0,$ &
         $J(e_2,e_2,e_2,f_2)=0 \  \Rightarrow \  (\tau_1-1)\tau_1(2\tau_1-1)=0.$ 
        \end{tabular}
    \end{center}
    This contradiction implies that no superalgebras can be found in this case.
\item[2.] $(\mu_1,\mu_2)=(0,\frac{1}{2})$. In this case we have the following results
    \begin{center}
        
    $J(e_1,e_1,e_2,f_1)=0 \  \Rightarrow \  \beta_2=0,$ \ $J(e_1,e_1,f_2,e_2)=0 \  \Rightarrow \  \beta_3=0,$ \ $J(e_1,e_1,f_2,f_1)=0 \  \Rightarrow \  \xi_1=\xi_2=0,$ \ $J(e_1,e_2,e_2,f_2)=0 \  \Rightarrow \  \beta_4(2\beta_4-1)=0,$ \ $J(e_2,e_2,e_2,f_1)=0 \  \Rightarrow \  (\beta_1-1)\beta_1(2\beta_1-1)=0.$
    
    \end{center}
    Hence, in this case, we obtain 6 Jordan superalgebras with $\beta_4 \in \{0,\frac{1}{2}\}$ and $\beta_1 \in \{0,1,\frac{1}{2}\}.$ Among them we have $\mathcal{J}_{9}, \mathcal{J}_{10}, \mathcal{J}_{11}, \mathcal{J}_{12}, \mathcal{J}_{13}$, while $\beta_1=\beta_4=0$ gives a superalgebra isomorphic to previously obtained one.
    
\item[3.] $(\mu_1,\mu_2)=(0,1)$ In this case we have the following results
    \begin{center}
        
         $J(e_1,e_1,e_2,f_1)=0 \  \Rightarrow \  \beta_2=0,$ \ $J(e_1,e_1,e_2,f_2)=0 \  \Rightarrow \  \beta_3=\beta_4=0,$ \ $J(e_1,e_1,f_1,f_2)=0 \  \Rightarrow \  \xi_1=0,$ \ $J(e_1,e_1,f_2,f_1)=0 \  \Rightarrow \  \xi_2=0,$ \ $J(e_2,e_2,e_2,f_1)=0 \  \Rightarrow \  (\beta_1-1)\beta_1(2\beta_1-1)=0.$
        
    \end{center}
    Hence, in this case, we obtain 3 Jordan superalgebras with $\beta_1\in\{0,1,\frac{1}{2}\}.$ Only $\beta_1=1$ gives a new superalgebra $\mathcal{J}_{14}$.
\item[4.] $(\mu_1,\mu_2)=(\frac{1}{2},\frac{1}{2})$. In this case, we can consider the action of $e_2.$
    By using a simple change of basis, the matrix of $L_{e_{2}}$ has one of the following forms:
    \begin{center}
        \begin{tabular}{cc}
            $\begin{pmatrix}
            \tau_1 & 0 \\  
            0 & \tau_2 
            \end{pmatrix},$ & 
            $\begin{pmatrix}
            \tau_1 & 1 \\  
            0 & \tau_1 
            \end{pmatrix}.$ 
        \end{tabular}
    \end{center}

    \noindent
    When $L_{e_{2}} \simeq       \begin{pmatrix}
        \tau_1 & 0 \\  
        0 & \tau_2 
    \end{pmatrix}$ we have:
    \begin{center}
        \begin{tabular}{cc}
         $J(e_1,e_2,e_2,f_1)=0 \  \Rightarrow \  \tau_1(2\tau_1-1)=0,$ &
         $J(e_1,e_2,e_2,f_4)=0 \  \Rightarrow \  \tau_2(2\tau_2-1)=0.$
        \end{tabular}
    \end{center}
    \begin{itemize}
        \item If  $(\tau_1, \tau_2)=(\frac{1}{2},\frac{1}{2})$ then we get the well-known one parametric family of four-dimensional Jordan superalgebras, which we denoted as $\mathcal{J}_{16}^t.$
        \item If  $(\tau_1, \tau_2)=(0,\frac{1}{2})$ then from $J(e_1, e_2, f_1,            f_2)=0$ we get $\xi_1=\xi_2=0$ which gives a Jordan superalgebra isomorphic to $\mathcal{J}_{13}$.
        \item If $(\tau_1, \tau_2)=(0,0)$ then from $J(e_1, e_2, f_1, f_2)=0$ we get        $\xi_2=0$ and hence the following Jordan superalgebra:
            $e_1^2=e_1, \ e_2^2=e_2, \ e_1f_1=\frac{1}{2}f_1, \ e_1 f_2=\frac{1}{2}f_2, \ f_1f_2=\xi_1 e_1$.
        Though only $\xi_1=0$ gives us a new superalgebra $\mathcal{J}_{15}.$
    \end{itemize}

    \noindent
    When $L_{e_{2}} \simeq       \begin{pmatrix}
        \tau_1 & 1 \\  
        0 & \tau_1 
    \end{pmatrix}$ we have:
    \begin{center}
        \begin{tabular}{ll}
         $J(e_1,e_2,e_2,f_1)=0 \  \Rightarrow \  2\tau_1-\frac{1}{2}=0,$ &
         $J(e_1,e_2,e_2,f_4)=0 \  \Rightarrow \  \tau_1(2\tau_1-1)=0.$
        \end{tabular}
    \end{center}
    This is a contradiction.

\item[5.] $(\mu_1,\mu_2)=(\frac{1}{2},1)$. In this case we have the following results
    \begin{center}
        
         $J(e_1,e_1,e_2,f_2)=0 \  \Rightarrow \  \beta_3=\beta_4=0,$ \
         $J(e_1,e_1,f_1,e_2)=0 \  \Rightarrow \  \beta_2=0,$ \
         $J(e_1,e_1,f_1,f_2)=0 \  \Rightarrow \  \xi_1=\xi_2=0,$ \
         $J(e_1,e_2,e_2,f_3)=0 \  \Rightarrow \  \beta_1(2\beta_1-1)=0.$
        
    \end{center}
    Hence, we have two Jordan superalgebras with $\beta_1 \in \{0,\frac{1}{2}\}$, which are isomorphic to $\mathcal{J}_5$ and $\mathcal{J}_{11}.$

\item[6.] $(\mu_1,\mu_2)=(1,1)$. In this case, we can consider the action of $e_2.$
    By using a simple change of basis, the matrix of $L_{e_{2}}$ has one of the following forms:
    \begin{center}
        \begin{tabular}{cc}
            $\begin{pmatrix}
            \tau_1 & 0 \\  
            0 & \tau_2 
            \end{pmatrix},$ & 
            $\begin{pmatrix}
            \tau_1 & 1 \\  
            0 & \tau_1 
            \end{pmatrix}.$ 
        \end{tabular}
    \end{center}

    \noindent
    When $L_{e_{2}} \simeq       \begin{pmatrix}
        \tau_1 & 0 \\  
        0 & \tau_2 
    \end{pmatrix}$ we have:
    \begin{center}

         $J(e_1,e_1,e_2,f_1)=0 \  \Rightarrow \  \tau_1=0,$ \
         $J(e_1,e_1,e_2,f_2)=0 \  \Rightarrow \  \tau_2=0,$ \
         $J(e_1,e_1,f_1,f_2)=0 \  \Rightarrow \  \xi_2=0.$

    \end{center}
    Hence, we have the following Jordan superalgebra: $e_1^2=e_1, \ e_2^2=e_2, \ e_1f_1=f_1, \ e_1f_2=f_2, \ f_1f_2=\xi_1e_1$ which repeats $\mathcal{J}_2$ and $\mathcal{J}_6.$

    \noindent
    When $L_{e_{2}} \simeq       \begin{pmatrix}
        \tau_1 & 1 \\  
        0 & \tau_1 
    \end{pmatrix}$ we have:
    \begin{center}
        \begin{tabular}{c}
         $J(e_1,e_1,e_2,f_1)=0 \  \Rightarrow \  f_2=0.$
        \end{tabular}
    \end{center}
    Thus, there are no superalgebras in this case.
    
\end{enumerate}

\noindent
\textbf{ii)} Let $L_{e_{1}} \simeq       \begin{pmatrix}
        \mu_1 & 1 \\  
        0 & \mu_1 
    \end{pmatrix}$ then the rule of multiplication can be written as follows:
\begin{center}

       $e_1^2=e_1,$ \ $e_2^2 =e_2,$  \ $e_1 f_1=\mu_1 f_1+f_2,$ \ $e_1 f_2=\mu_1 f_2,$ \\
       $e_2 f_1=\beta_1 f_1 +\beta_2 f_2,$ \ $f_1 f_2 = \xi_1 e_1+\xi_2 e_2 ,$ \ $e_2 f_2=\beta_3 f_1 +\beta_4 f_2.$

\end{center}
From $J(e_1, e_1,e_1,f_1)=0,$ $J(e_1, e_1,e_1,f_2)=0,$  we obtain equations below
\begin{center}
    \begin{tabular}{ll}
  $1-6\mu_1+6\mu_1^2=0,$ & $(\mu_1-1)\mu_1(2\mu_1-1)=0.$ 
\end{tabular}
\end{center}
Evidently, this is a contradiction.

\noindent

\underline{Let $\jor_0 \cong \mathcal{U}_2 \oplus \mathcal{U}_2$.}
Here we are looking for Jordan superalgebras such that $\jor=(\mathbb{F}e_1 +\mathbb{F}e_2 )+(\mathbb{F}f_1+\mathbb{F}f_2)$ with multiplication rules

       $e_1 f_1=\alpha_1 f_1 +\alpha_2 f_2,$ \ $e_1 f_2=\alpha_3 f_1 +\alpha_4 f_2,$ \  $e_2 f_1=\beta_1 f_1 +\beta_2 f_2,$ \
       $f_1 f_2 = \xi_1 e_1+\xi_2 e_2, $ \ $e_2 f_2=\beta_3 f_1 +\beta_4 f_2.$

\noindent
For the action of the operator $L_{e_{1}}$ on $\jor_1$ we can write the following matrix
\begin{center}
    \begin{tabular}{c}
        $\begin{pmatrix}
        \alpha_1 & \alpha_2 \\  
        \alpha_3 & \alpha_4 
        \end{pmatrix}.$ 
    \end{tabular}
\end{center}

\noindent
However, it is easy to prove that, by using a simple change of basis, the matrix of $L_{e_{1}}$ has one of the following forms:
\begin{center}
    \begin{tabular}{cc}
        $\begin{pmatrix}
        \mu_1 & 0 \\  
        0 & \mu_2 
        \end{pmatrix},$ & 
        $\begin{pmatrix}
        \mu_1 & 1 \\  
        0 & \mu_1 
        \end{pmatrix}.$ 
    \end{tabular}
\end{center}

\noindent
\textbf{i)} Let $L_{e_{1}} \simeq       \begin{pmatrix}
        \mu_1 & 0 \\  
        0 & \mu_2 
    \end{pmatrix},$ then the rule of multiplication can be written as follows:
\begin{center}

       $e_1 f_1=\mu_1 f_1,$ \ $e_1 f_2=\mu_2 f_2,$ \ $e_2 f_1=\beta_1 f_1 +\beta_2 f_2,$ \ $f_1 f_2 = \xi_1 e_1+\xi_2 e_2, $ \ $e_2 f_2=\beta_3 f_1 +\beta_4 f_2.$

\end{center}

\noindent
Using Jordan super identity we get the following results:
    \begin{center}
        \begin{tabular}{cc}
         $J(e_1,e_1,e_1,f_1)=0 \  \Rightarrow \  \mu_1=0,$ &
         $J(e_1,e_1,e_1,f_2)=0 \  \Rightarrow \  \mu_2=0.$
        \end{tabular}
    \end{center}
Thus, we can consider the action of $e_2.$
    By using a simple change of basis, the matrix of $L_{e_{2}}$ has one of the following forms:
    \begin{center}
        \begin{tabular}{cc}
            $\begin{pmatrix}
            \tau_1 & 0 \\  
            0 & \tau_2 
            \end{pmatrix},$ & 
            $\begin{pmatrix}
            \tau_1 & 1 \\  
            0 & \tau_1 
            \end{pmatrix}.$ 
        \end{tabular}
    \end{center}

    \noindent
    When $L_{e_{2}} \simeq       \begin{pmatrix}
        \tau_1 & 0 \\  
        0 & \tau_2 
    \end{pmatrix}$ we have:
    \begin{center}
        \begin{tabular}{cc}
         $J(e_2,e_2,e_2,f_1)=0 \  \Rightarrow \  \tau_1=0,$ &
         $J(e_2,e_2,e_2,f_2)=0 \  \Rightarrow \  \tau_2=0.$
        \end{tabular}
    \end{center}

    \noindent
    Hence, we have the Jordan superalgebra: $f_1f_2=\xi_1e_1+\xi_2 e_2.$
    \begin{itemize}
        \item If $(\xi_1, \xi_2)=(0,0)$ the superalgebra is trivial.
        \item If $(\xi_1, \xi_2)\neq(0,0)$ then by changing the basis as follows
        $$e_1'=\xi_1e_1+\xi_2e_2, \ e_2'=\begin{cases}
            e_2, & \text{ if } \xi_1 \neq 0, \\
            e_1, & \text{ if } \xi_1=0. 
        \end{cases}$$
        we get the superalgebra $\mathcal{J}_{17}$.
    \end{itemize}

    \noindent
    When $L_{e_{2}} \simeq       \begin{pmatrix}
        \tau_1 & 1 \\  
        0 & \tau_1 
    \end{pmatrix}$ we have:
    \begin{center}
        \begin{tabular}{c}
         $J(e_2,e_2,e_2,f_1)=0 \  \Rightarrow \  \tau_1=0.$
        \end{tabular}
    \end{center}

    \noindent
    Hence, we have the Jordan superalgebra: $e_2f_1=f_2, \ f_1f_2=\xi_1e_1+\xi_2 e_2.$
    \begin{itemize}
        \item If $(\xi_1, \xi_2)=(0,0)$, we have the superalgebra $\mathcal{J}_{18}.$
        \item If $\xi_2=0, \xi_1 \neq 0$, then by changing $f_1'=\frac{1}{\sqrt{\xi_1}}f_1, \ f_2'=\frac{1}{\sqrt{\xi_1}}f_2$ we get the superalgebra $\mathcal{J}_{19}.$
        \item If $\xi_2 \neq 0$, then changing $f_1'=\frac{1}{\sqrt{\xi_2}}f_1, \ f_2'=\frac{1}{\sqrt{\xi_2}}f_2, \ \frac{\xi_1}{\xi_2}=t$ we have $$e_2f_1'=f_2', \ f_1'f_2'=te_1+e_2.$$
        Further, by changing $e_2'=te_1+e_2$ we obtain the superalgebra
        $\mathcal{J}_{20}.$
    \end{itemize}

\noindent
\underline{Let $\jor_0 \cong \mathcal{U}_1 \oplus \mathcal{U}_2$.}
Here we are looking for Jordan superalgebras such that $\jor=(\mathbb{F}e_1 +\mathbb{F}e_2)+(\mathbb{F}f_1+\mathbb{F}f_2)$ with multiplication rules

\begin{center}

       $e_1^2=e_1,$ \ $e_1 f_1=\alpha_1 f_1 +\alpha_2 f_2,$ \ $e_1 f_2=\alpha_3 f_1 +\alpha_4 f_2,$ \ $e_2 f_1=\beta_1 f_1 +\beta_2 f_2,$ \ $f_1 f_2 = \xi_1 e_1+\xi_2 e_2, $ \ $e_2 f_2=\beta_3 f_1 +\beta_4 f_2.$
    
\end{center}

\noindent
For the action of the operator $L_{e_{1}}$ on $\jor_1$ we can write the following matrix
\begin{center}
    \begin{tabular}{c}
        $\begin{pmatrix}
        \alpha_1 & \alpha_2 \\  
        \alpha_3 & \alpha_4 
        \end{pmatrix}.$ 
    \end{tabular}
\end{center}

\noindent
However, it is easy to prove that, by using a simple change of basis, the matrix of $L_{e_{1}}$ has one of the following forms:
\begin{center}
    \begin{tabular}{cc}
        $\begin{pmatrix}
        \mu_1 & 0 \\  
        0 & \mu_2 
        \end{pmatrix},$ & 
        $\begin{pmatrix}
        \mu_1 & 1 \\  
        0 & \mu_1 
        \end{pmatrix}.$ 
    \end{tabular}
\end{center}

\noindent
\textbf{i)} Let $L_{e_{1}} \simeq       \begin{pmatrix}
        \mu_1 & 0 \\  
        0 & \mu_2 
    \end{pmatrix}$ then the rule of multiplication can be written as follows:
\begin{center}
    \begin{tabular}{l l}
       $e_1^2=e_1,$  & $e_1 f_1=\mu_1 f_1,$ \ $e_1 f_2=\mu_2 f_2,$ \       $e_2 f_1=\beta_1 f_1 +\beta_2 f_2,$ \ $f_1 f_2 = \xi_1 e_1+\xi_2 e_2, $ \ $e_2 f_2=\beta_3 f_1 +\beta_4 f_2.$
    \end{tabular}
\end{center}

From $J(e_1, e_1,e_1,f_1)=0,$ $J(e_1, e_1,e_1,f_2)=0,$  we obtain equations below
\begin{center}
    \begin{tabular}{cc}
  $(\mu_i-1)\mu_i(2\mu_i-1)=0,$ & $i=\overline{1,2}.$
\end{tabular}
\end{center}
Up to permutation of $f_1$ and $f_2$ we have six possibilities: 
\begin{center}
$(\mu_1,\mu_2)\in \{(0,0), (0,\frac{1}{2}), (0,1), (\frac{1}{2}, \frac{1}{2}), (\frac{1}{2},1), (1,1)\}.$
\end{center}

\begin{enumerate}
    \item[1.] $(\mu_1, \mu_2)=(0,0)$. In this case, we can consider the action of $e_2.$
    By using a simple change of basis, the matrix of $L_{e_{2}}$ has one of the following forms:
    \begin{center}
        \begin{tabular}{cc}
            $\begin{pmatrix}
            \tau_1 & 0 \\  
            0 & \tau_2 
            \end{pmatrix},$ & 
            $\begin{pmatrix}
            \tau_1 & 1 \\  
            0 & \tau_1 
            \end{pmatrix}.$ 
        \end{tabular}
    \end{center}
    When $L_{e_{2}} \simeq       \begin{pmatrix}
        \tau_1 & 0 \\  
        0 & \tau_2 
    \end{pmatrix}$ we have:
    \begin{center}

         $J(e_1,e_1,f_1,f_2)=0 \  \Rightarrow \  \xi_1=0,$ \
         $J(e_2,e_2,e_2,f_1)=0 \  \Rightarrow \  \tau_1=0,$ \
         $J(e_2,e_2,e_2,f_2)=0 \  \Rightarrow \  \tau_2=0.$ 

    \end{center}
    So we have the superalgebras $\mathcal{J}_{21}$ and $\mathcal{J}_{22}$.
    When $L_{e_{2}} \simeq       \begin{pmatrix}
        \tau_1 & 1 \\  
        0 & \tau_1 
    \end{pmatrix}$ 
    we have 
    \begin{center}
        \begin{tabular}{ll}
         $J(e_1,e_1,f_1,f_2)=0 \  \Rightarrow \  \xi_1=0,$ &
         $J(e_2,e_2,e_2,f_2)=0 \  \Rightarrow \  \tau_1=0.$ 
         \end{tabular}
    \end{center}
    Which gives us the superalgebras $\mathcal{J}_{23}$ and $\mathcal{J}_{24}$.
    
    \item[2.] $(\mu_1, \mu_2)=(0,\frac{1}{2})$. In this case, we have the following results
    \begin{center}

         $J(e_1,e_1,e_2,f_1)=0 \  \Rightarrow \  \beta_2=0,$ \
         $J(e_1,e_1,f_2,e_2)=0 \  \Rightarrow \  \beta_3=0,$ \
         $J(e_1,e_1,f_2,f_1)=0 \  \Rightarrow \  \xi_1=\xi_2=0,$ \
         $J(e_1,e_2,e_2,f_2)=0 \  \Rightarrow \  \beta_4=0,$ \
         $J(e_2,e_2,e_2,f_1)=0 \  \Rightarrow \  \beta_1=0.$

    \end{center}
    Hence, we obtain the superalgebra $\mathcal{J}_{25}$.

    \item[3.] $(\mu_1, \mu_2)=(0,1)$. In this case, we have the following results
    \begin{center}

         $J(e_1,e_2,e_1,f_2)=0 \  \Rightarrow \  \beta_3=\beta_4=0,$ \
         $J(e_1,f_1,e_1,f_2)=0 \  \Rightarrow \  \xi_1=\xi_2=0,$ \
         $J(e_1,e_1,e_2,f_1)=0 \  \Rightarrow \  \beta_2=0,$ \
         $J(e_2,e_2,e_2,f_1)=0 \  \Rightarrow \  \beta_1=0.$

    \end{center}
    Hence, we get the superalgebra $\mathcal{J}_{26}$.

    \item[4.] $(\mu_1, \mu_2)=(\frac{1}{2},\frac{1}{2})$. In this case, we can consider the action of $e_2.$
    By using a simple change of basis, the matrix of $L_{e_{2}}$ has one of the following forms:
    \begin{center}
        \begin{tabular}{cc}
            $\begin{pmatrix}
            \tau_1 & 0 \\  
            0 & \tau_2 
            \end{pmatrix},$ & 
            $\begin{pmatrix}
            \tau_1 & 1 \\  
            0 & \tau_1 
            \end{pmatrix}.$ 
        \end{tabular}
    \end{center}
    When $L_{e_{2}} \simeq       \begin{pmatrix}
        \tau_1 & 0 \\  
        0 & \tau_2 
    \end{pmatrix}$ we have:
    \begin{center}

         $J(e_1,e_2,e_2,f_1)=0 \  \Rightarrow \  \tau_1=0,$ \
         $J(e_1,e_2,e_2,f_2)=0 \  \Rightarrow \  \tau_2=0.$ 

    \end{center}
    So we have the following superalgebra:
    \begin{center}

       $e_1^2=e_1,$ \ $e_1 f_1=\frac{1}{2}f_1,$ \ $e_1 f_2=\frac{1}{2} f_2,$ \ $f_1 f_2 = \xi_1e_1+\xi_2 e_2.$

    \end{center}
    \begin{itemize}
        \item If $\xi_1=\xi_2=0$ then we have the superalgebra $\mathcal{J}_{27}$.
        \item If $\xi_1=0, \xi_2\neq0$ then by the change $f_1'=\frac{1}{\xi_2} f_1$ we get the superalgebra $\mathcal{J}_{28}.$
        \item If $\xi_1\neq 0, \ \xi_2=0,$ then by changing $f_1'=\frac{1}{\xi_1} f_1$ we get the superalgebra $\mathcal{J}_{29}.$
        \item If  $\xi_1\neq 0, \ \xi_2 \neq 0,$ then by changing $f_1'=\frac{1}{\xi_1} f_1$ and $e_2'=\frac{\xi_2}{\xi_1}e_2$ we obtain the superalgebra $\mathcal{J}_{30}.$
    \end{itemize}
    When $L_{e_{2}} \simeq       \begin{pmatrix}
        \tau_1 & 1 \\  
        0 & \tau_1 
    \end{pmatrix}$ 
    we have 
    \begin{center}
   
         $J(e_1,e_2,e_2,f_1)=0 \  \Rightarrow \  \tau_1=0,$ \
         $J(e_1,e_2,f_1,f_1)=0 \  \Rightarrow \  \xi_1=0.$ 

    \end{center}
    So we have the superalgebras $\mathcal{J}_{31}$ and $\mathcal{J}_{32}$.

    \item[5.] $(\mu_1, \mu_2)=(\frac{1}{2},1)$. In this case, we have the following results
    \begin{center}

         $J(e_1,e_1,e_2,f_2)=0 \  \Rightarrow \  \beta_3=\beta_4=0,$ \
         $J(e_1,e_1,f_1,e_2)=0 \  \Rightarrow \  \beta_2=0,$ \
         $J(e_1,e_1,f_1,f_2)=0 \  \Rightarrow \  \xi_1=\xi_2=0,$\
         $J(e_1,e_2,e_2,f_1)=0 \  \Rightarrow \  \beta_1=0.$

    \end{center}
    Hence, the obtained superalgebra is $\mathcal{J}_{33}.$

    \item[6.] $(\mu_1, \mu_2)=(1,1)$. In this case, we can consider the action of $e_2.$
    By using a simple change of basis, the matrix of $L_{e_{2}}$ has one of the following forms:
    \begin{center}
    \begin{tabular}{cc}
        $\begin{pmatrix}
        \tau_1 & 0 \\  
        0 & \tau_2 
        \end{pmatrix},$ & 
        $\begin{pmatrix}
        \tau_1 & 1 \\  
        0 & \tau_1 
        \end{pmatrix}.$ 
    \end{tabular}
    \end{center}
    When $L_{e_{2}} \simeq       \begin{pmatrix}
        \tau_1 & 0 \\  
        0 & \tau_2 
    \end{pmatrix}$ we have:
    \begin{center}

         $J(e_1,e_1,e_2,f_1)=0 \  \Rightarrow \  \tau_1=0,$ \
         $J(e_1,e_1,e_2,f_2)=0 \  \Rightarrow \  \tau_2=0,$ \
         $J(e_1,e_1,f_1,f_2)=0 \  \Rightarrow \  \xi_2=0.$ 

    \end{center}
    So we have superalgebras $\mathcal{J}_{34}$ and $\mathcal{J}_{35}$.

    \noindent
    When $L_{e_{2}} \simeq       \begin{pmatrix}
        \tau_1 & 1 \\  
        0 & \tau_1 
    \end{pmatrix}$ 
    we have 
    \begin{center}
        \begin{tabular}{l}
         $J(e_1,e_1,e_2,f_1)=0 \  \Rightarrow \  1=0.$  
         \end{tabular}
    \end{center}
    Which is, obviously, a contradiction.
\end{enumerate}

\noindent

\underline{Let $\jor_0 \cong \mathcal{B}_1$.}
Here we are looking for Jordan superalgebras such that $\jor=(\mathbb{F}e_1 +\mathbb{F}e_2)+(\mathbb{F}f_1+\mathbb{F}f_2)$ with multiplication rules

\begin{center}

       $e_1^2=e_1,$ \ $e_1 f_1=\alpha_1 f_1 +\alpha_2 f_2,$ \
       $e_1 e_2 =e_2,$ \ $e_1 f_2=\alpha_3 f_1 +\alpha_4 f_2,$ \
       $e_2 f_1=\beta_1 f_1 +\beta_2 f_2,$ \
       $f_1 f_2 = \xi_1 e_1+\xi_2 e_2, $ \ $e_2 f_2=\beta_3 f_1 +\beta_4 f_2.$

\end{center}

For the action of the operator $L_{e_{1}}$ on $\jor_1$ we can write the following matrix
\begin{center}
    \begin{tabular}{c}
        $\begin{pmatrix}
        \alpha_1 & \alpha_2 \\  
        \alpha_3 & \alpha_4 
        \end{pmatrix}.$ 
    \end{tabular}
\end{center}

\noindent
However, it is easy to prove that, by using a simple change of basis, the matrix of $L_{e_{1}}$ has one of the following forms:
\begin{center}
    \begin{tabular}{cc}
        $\begin{pmatrix}
        \mu_1 & 0 \\  
        0 & \mu_2 
        \end{pmatrix},$ & 
        $\begin{pmatrix}
        \mu_1 & 1 \\  
        0 & \mu_1 
        \end{pmatrix}.$ 
    \end{tabular}
\end{center}

\noindent
\textbf{i)} Let $L_{e_{1}} \simeq       \begin{pmatrix}
        \mu_1 & 0 \\  
        0 & \mu_2 
    \end{pmatrix}$ then the rule of multiplication can be written as follows:
\begin{center}

       $e_1^2=e_1,$ \ $e_1 f_1=\mu_1 f_1,$ \ $e_1 e_2 =e_2,$ \ $e_1 f_2=\mu_2 f_2,$ \ $e_2 f_1=\beta_1 f_1 +\beta_2 f_2,$ \
       $f_1 f_2 = \xi_1 e_1+\xi_2 e_2, $ \ $e_2 f_2=\beta_3 f_1 +\beta_4 f_2.$

\end{center}

From $J(e_1, e_1,e_1,f_1)=0,$ $J(e_1, e_1,e_1,f_2)=0,$  we obtain equations below
\begin{center}
    \begin{tabular}{cc}
  $(\mu_1-1)\mu_1(2\mu_1-1)=0,$ & $i=\overline{1,2}.$
\end{tabular}
\end{center}
Up to permutation of $f_1$ and $f_2$ we have six possibilities:
\begin{center}
$(\mu_1,\mu_2)\in \{(0,0), (0,\frac{1}{2}), (0,1), (\frac{1}{2}, \frac{1}{2}), (\frac{1}{2},1), (1,1)\}.$
\end{center}

\begin{enumerate}
    \item[1.] $(\mu_1,\mu_2)=(0,0)$. In this case, we have the following results
    \begin{center}

         $J(e_1,e_1,e_2,f_1)=0 \  \Rightarrow \  \beta_1=\beta_2=0,$ \
         $J(e_1,e_1,e_2,f_2)=0 \  \Rightarrow \  \beta_3=\beta_4=0,$ \
         $J(e_1,e_1,f_1,f_2)=0 \  \Rightarrow \  \xi_1=\xi_2=0.$

    \end{center}
    Hence, the obtained superalgebra is $\mathcal{J}_{36}$.
    
    \item[2.] $(\mu_1,\mu_2)=(0,\frac{1}{2})$. In this case, we have the following results
    \begin{center}

         $J(e_1,e_1,e_2,f_1)=0 \  \Rightarrow \  \beta_1=\beta_2=0,$ \
         $J(e_1,e_1,f_2,e_2)=0 \  \Rightarrow \  \beta_3=0,$ \
         $J(e_1,e_2,e_2,f_2)=0 \  \Rightarrow \  \beta_4=0,$ \
         $J(e_1,e_1,f_2,f_1)=0 \  \Rightarrow \  \xi_1=\xi_2=0.$ 

    \end{center}
    Hence, the obtained superalgebra is $\mathcal{J}_{37}$.

    \item[3.] $(\mu_1,\mu_2)=(0,1).$ In this case, we have the following results
    \begin{center}

         $J(e_1,e_2,e_1,f_1)=0 \  \Rightarrow \  \beta_1=\beta_2=0,$ \
         $J(e_1,e_1,e_2,f_2)=0 \  \Rightarrow \  \beta_3=0,$ \
         $J(e_1,e_1,f_1,f_2)=0 \  \Rightarrow \  \xi_1=\xi_2=0,$ \
         $J(f_2,e_2,e_2,e_2)=0 \  \Rightarrow \  \beta_4=0.$

    \end{center}
    Hence, the obtained superalgebra is $\mathcal{J}_{38}$.

    \item[4.] $(\mu_1,\mu_2)=(\frac{1}{2},\frac{1}{2})$. In this case, we can consider the action of $e_2.$
    By using a simple change of basis, the matrix of $L_{e_{2}}$ have one of the following forms:
    \begin{center}
        \begin{tabular}{cc}
            $\begin{pmatrix}
            \tau_1 & 0 \\  
            0 & \tau_2 
            \end{pmatrix},$ & 
            $\begin{pmatrix}
            \tau_1 & 1 \\  
            0 & \tau_1 
            \end{pmatrix}.$ 
        \end{tabular}
    \end{center}
    When $L_{e_{2}} \simeq       \begin{pmatrix}
        \tau_1 & 0 \\  
        0 & \tau_2 
    \end{pmatrix}$ we have:
    \begin{center}

         $J(e_1,e_2,e_2,f_1)=0 \  \Rightarrow \  \tau_1=0,$ \
         $J(e_1,e_2,e_2,f_2)=0 \  \Rightarrow \  \tau_2=0,$ \
         $J(e_1,e_2,f_1,f_2)=0 \  \Rightarrow \  \xi_1=0.$ 
    \end{center}
    So we have the superalgebras $\mathcal{J}_{39}$ and $\mathcal{J}_{40}$.
    
    When $L_{e_{2}} \simeq       \begin{pmatrix}
        \tau_1 & 1 \\  
        0 & \tau_1 
    \end{pmatrix}$ 
    we have 
    \begin{center}

         $J(e_1,e_2,e_2,f_1)=0 \  \Rightarrow \  \tau_1=0$, \
         $J(e_1,e_2,f_1,f_2)=0 \  \Rightarrow \  \xi_1=0.$ 

    \end{center}
    which gives us the superalgebras $\mathcal{J}_{41}$ and $\mathcal{J}_{42}$.

    \item[5.] $(\mu_1,\mu_2)=(\frac{1}{2},1)$. In this case, we have the following results
    \begin{center}

         $J(e_1,e_1,e_2,f_2)=0 \  \Rightarrow \  \beta_3=0,$ \
         $J(e_1,e_1,f_1,e_2)=0 \  \Rightarrow \  \beta_2=0,$ \
         $J(e_1,e_1,f_1,f_2)=0 \  \Rightarrow \  \xi_1=\xi_2=0,$\
         $J(e_1,e_2,e_2,f_1)=0 \  \Rightarrow \  \beta_1=0,$\
         $J(e_2,e_2,e_2,f_2)=0 \  \Rightarrow \  \beta_4=0.$

    \end{center}
    Hence, the obtained superalgebra is $\mathcal{J}_{43}.$

    \item[6.] $(\mu_1,\mu_2)=(1,1).$ In this case, we can consider the action of $e_2.$
    By using a simple change of basis, the matrix of $L_{e_{2}}$ have one of the following forms:
    \begin{center}
    \begin{tabular}{cc}
        $\begin{pmatrix}
        \tau_1 & 0 \\  
        0 & \tau_2 
        \end{pmatrix},$ & 
        $\begin{pmatrix}
        \tau_1 & 1 \\  
        0 & \tau_1 
        \end{pmatrix}.$ 
    \end{tabular}
    \end{center}
    When $L_{e_{2}} \simeq       \begin{pmatrix}
        \tau_1 & 0 \\  
        0 & \tau_2 
    \end{pmatrix}$ we have:
    \begin{center}

         $J(e_2,e_2,e_2,f_1)=0 \  \Rightarrow \  \tau_1=0,$ \
         $J(e_2,e_2,e_2,f_2)=0 \  \Rightarrow \  \tau_2=0.$

    \end{center}
    So we have the following superalgebra:
    \begin{center}

       $e_1^2=e_1,$  \ $e_1 f_1=f_1,$ \
       $e_1 e_2 =e_2,$  \ $e_1 f_2=f_2,$ \
       $f_1 f_2 =\xi_1e_1 + \xi_2 e_2, $

    \begin{itemize}
        \item If $(\xi_1,\xi_2)=(0,0)$ we have the superalgebra $\mathcal{J}_{44}$.
        \item If $(\xi_1,\xi_2)\neq(0,0)$ 
        \begin{enumerate}
            \item When $\xi_1=0,$ by changing $f_1'=\frac{1}{\xi_2}f_1$ we get the superalgebra $\mathcal{J}_{45}$.
            \item When $\xi_1\neq 0,$ by changing $f_1'=\frac{1}{\xi_1}f_1$ and denoting $\frac{\xi_2}{\xi_1}=t$ we can write $$e_1^2=e_1, \ e_1e_2=e_2, \ e_1f_1=f_1, \ e_1f_2=f_2, \ f_1f_2=e_1+te_2.$$
            In this case, we get $\mathcal{J}_{46}$ when $t=0$ and $\mathcal{J}_{47}$ when $t\neq 0$. 
        \end{enumerate}
    \end{itemize}
    \end{center}
    When $L_{e_{2}} \simeq       \begin{pmatrix}
        \tau_1 & 1 \\  
        0 & \tau_1 
    \end{pmatrix}$ 
    we have 
    \begin{center}

         $J(e_2,e_2,e_2,f_2)=0 \  \Rightarrow \  \tau_1=0,$ \
         $J(e_2,e_2,f_1,f_1)=0 \  \Rightarrow \  \xi_1=0.$ 

    \end{center}
    So we have the superalgebras $\mathcal{J}_{48}$ and $\mathcal{J}_{49}$.
\end{enumerate}

\noindent
\textbf{ii)} Let $L_{e_{1}} \simeq       \begin{pmatrix}
        \mu_1 & 1 \\  
        0 & \mu_1 
    \end{pmatrix}$ then the rule of multiplication can be written as follows:
$$e_1^2=e_1, \ e_1 e_2 =e_2, \ e_1 f_1=\mu_1 f_1+f_2, \ e_1 f_2=\mu_1 f_2,$$
$$e_2 f_1=\beta_1 f_1 +\beta_2 f_2, \ f_1 f_2 = \xi_1 e_1+\xi_2 e_2, \ e_2 f_2=\beta_3 f_1 +\beta_4 f_2.$$

However, from
    \begin{center}

         $J(e_1,e_1,e_1,f_1)=0 \  \Rightarrow \  1-6\mu_1+6\mu_1^2=0,$ \
         $J(e_1,e_1,e_1,f_2)=0 \  \Rightarrow \  (\mu_1-1)\mu_1(2\mu_1-1)=0.$ 

    \end{center}
    we get a contradiction.

\noindent

\underline{Let $\jor_0 \cong \mathcal{B}_2$.}
Here we are looking for Jordan superalgebras such that $\jor=(\mathbb{F}e_1 +\mathbb{F}e_2)+(\mathbb{F}f_1+\mathbb{F}f_2)$ with multiplication rules
$$e_1^2=e_1, \ e_1 f_1=\alpha_1 f_1 +\alpha_2 f_2, \ e_1 e_2 =\frac{1}{2}e_2, \ e_1 f_2=\alpha_3 f_1 +\alpha_4 f_2,$$ 
$$e_2 f_1=\beta_1 f_1 +\beta_2 f_2, \ f_1 f_2 = \xi_1 e_1+\xi_2 e_2,  \ e_2 f_2=\beta_3 f_1 +\beta_4 f_2.$$

\noindent
For the action of the operator $L_{e_{1}}$ on $\jor_1$ we can write the following matrix
\begin{center}
    \begin{tabular}{c}
        $\begin{pmatrix}
        \alpha_1 & \alpha_2 \\  
        \alpha_3 & \alpha_4 
        \end{pmatrix}.$ 
    \end{tabular}
\end{center}

\noindent
However, it is easy to prove that, by using a simple change of basis, the matrix of $L_{e_{1}}$ has one of the following forms:
\begin{center}
    \begin{tabular}{cc}
        $\begin{pmatrix}
        \mu_1 & 0 \\  
        0 & \mu_2 
        \end{pmatrix},$ & 
        $\begin{pmatrix}
        \mu_1 & 1 \\  
        0 & \mu_1 
        \end{pmatrix}.$ 
    \end{tabular}
\end{center}

\noindent
\textbf{i)} Let $L_{e_{1}} \simeq       \begin{pmatrix}
        \mu_1 & 0 \\  
        0 & \mu_2 
    \end{pmatrix}$ then the rule of multiplication can be written as follows:
$$e_1^2=e_1, \ e_1 f_1=\mu_1 f_1, \ e_1 e_2 =\frac{1}{2}e_2, \ e_1 f_2=\mu_2 f_2,$$
$$e_2 f_1=\beta_1 f_1 +\beta_2 f_2, \ f_1 f_2 = \xi_1 e_1+\xi_2 e_2,  \ e_2 f_2=\beta_3 f_1 +\beta_4 f_2.$$

From $J(e_1, e_1,e_1,f_1)=0,$ $J(e_1, e_1,e_1,f_2)=0,$  we obtain equations below
\begin{center}
    \begin{tabular}{cc}
  $(\mu_1-1)\mu_1(2\mu_1-1)=0,$ & $i=\overline{1,2}.$
\end{tabular}
\end{center}
Up to permutation of $f_1$ and $f_2$ we have six possibilities:
\begin{center}
$(\mu_1,\mu_2)\in \{(0,0), (0,\frac{1}{2}), (0,1), (\frac{1}{2}, \frac{1}{2}), (\frac{1}{2},1), (1,1)\}.$
\end{center}

\begin{enumerate}
    \item[1.] $(\mu_1, \mu_2)=(0,0)$. In this case, we have the following results
    \begin{center}

         $J(e_1,e_1,e_2,f_1)=0 \  \Rightarrow \  \beta_1=\beta_2=0,$ \
         $J(e_1,e_1,e_2,f_2)=0 \  \Rightarrow \  \beta_3=\beta_4=0,$ \
         $J(e_1,e_1,f_1,f_2)=0 \  \Rightarrow \  \xi_1=\xi_2=0.$

    \end{center}
    Hence, the obtained superalgebra is $\mathcal{J}_{50}$.
    
    \item[2.] $(\mu_1, \mu_2)=(0,\frac{1}{2})$. In this case, we have the following results
    \begin{center}

         $J(e_1,e_1,e_2,f_1)=0 \  \Rightarrow \  \beta_1=0,$ \
         $J(e_1,e_1,f_2,f_1)=0 \  \Rightarrow \  \xi_1=0,$ \
         $J(e_1,e_2,e_1,f_2)=0 \  \Rightarrow \  \beta_4=0,$ \
         $J(e_1,e_2,e_2,f_1)=0 \  \Rightarrow \  \beta_2 \beta_3=0.$

    \end{center}
    We consider the following subcases:
    \begin{enumerate}
        \item[(a)] $\beta_2=0$. If $\beta_3=0$ then we take the superalgebra $\mathcal{J}_{51}$ when $\xi_2=0$, and when $\xi_2 \neq 0$ we can change $f_1'=\frac{1}{\xi_2} f_1$  and thereby obtain superalgebra $\mathcal{J}_{52}$.

        If $\beta_3 \neq 0$, we get the superalgebra $\mathcal{J}_{53}$ when $\xi_2=0$, and when $\xi_2 \neq 0$ then we change the basis by taking $f_1'=\sqrt{\frac{\beta_3}{\xi_2}}f_1, \ f_2'=\frac{1}{\sqrt{\beta_3 \xi_2}}f_2$, which gives us the superalgebra $\mathcal{J}_{54}$.

        \item[(b)] $\beta_3=0, \ \beta_2 \neq 0$. In this subcase we obtain superalgebras $\mathcal{J}_{55}$ (when $\xi_2=0$) and $\mathcal{J}_{56}$ (when $\xi_2 \neq 0$).
    \end{enumerate}
    
    \item[3.] $(\mu_1, \mu_2)=(0,1)$.
    In this case, we have the following results
    \begin{center}

         $J(e_1,e_1,e_2,f_1)=0 \  \Rightarrow \  \beta_1=\beta_2=0,$ \
         $J(e_1,e_1,e_2,f_2)=0 \  \Rightarrow \  \beta_3=\beta_4=0,$ \
         $J(e_1,e_1,f_1,f_2)=0 \  \Rightarrow \  \xi_1=\xi_2=0.$

    \end{center}
   This gives us the Jordan superalgebra $\mathcal{J}_{57}.$

    \item[4.] $(\mu_1, \mu_2)=(\frac{1}{2},\frac{1}{2})$. In this case, we have the following results
    \begin{center}

         $J(e_1,e_2,e_1,f_1)=0 \  \Rightarrow \  \beta_1=\beta_2=0,$ \
         $J(e_1,e_2,e_1,f_2)=0 \  \Rightarrow \  \beta_3=\beta_4=0,$ \
         $J(e_1,f_1,e_1,f_2)=0 \  \Rightarrow \  \xi_2=0,$ \
         $J(e_2,f_1,e_1,f_2)=0 \  \Rightarrow \  \xi_1=0.$

    \end{center}
    This gives us the Jordan superalgebra $\mathcal{J}_{58}.$

    \item[5.] $(\mu_1, \mu_2)=(\frac{1}{2},1)$. In this case, we have the following results
    \begin{center}

         $J(e_1,e_1,e_2,f_2)=0 \  \Rightarrow \  \beta_4=0,$ \
         $J(e_1,e_1,f_1,f_2)=0 \  \Rightarrow \  \xi_1=0,$ \
         $J(e_1,e_2,e_1,f_1)=0 \  \Rightarrow \  \beta_1=0,$ \
         $J(e_1,e_2,e_2,f_2)=0 \  \Rightarrow \  \beta_2 \beta_3=0.$

    \end{center}
    This case, being similar to the second one, gives us the Jordan superalgebras $\mathcal{J}_{59}-\mathcal{J}_{64}$.

    \item[6.] $(\mu_1, \mu_2)=(1,1)$.
    In this case, we have the following results
    \begin{center}

         $J(e_1,e_1,e_2,f_1)=0 \  \Rightarrow \  \beta_1=\beta_2=0,$ \
         $J(e_1,e_1,e_2,f_2)=0 \  \Rightarrow \  \beta_3=\beta_4=0,$ \
         $J(e_1,e_1,f_1,f_2)=0 \  \Rightarrow \  \xi_2=0,$ \
         $J(e_2,f_1,e_1,f_2)=0 \  \Rightarrow \  \xi_1=0.$

    \end{center}
    Here we obtain the Jordan superalgebra $\mathcal{J}_{65}$.
\end{enumerate}
\noindent
\textbf{ii)} Let $L_{e_{1}} \simeq       \begin{pmatrix}
        \mu_1 & 1 \\  
        0 & \mu_1 
    \end{pmatrix}$ then the rule of multiplication can be written as follows:
$$e_1^2=e_1, \ e_1 f_1=\mu_1 f_1+f_2, \ e_1 e_2 =\frac{1}{2}e_2, \ e_1 f_2=\mu_1 f_2, \ e_2^2 =0,$$  
$$e_2 f_1=\beta_1 f_1 +\beta_2 f_2, \ f_1 f_2 = \xi_1 e_1+\xi_2 e_2,  \ e_2 f_2=\beta_3 f_1 +\beta_4 f_2.$$

However, from the following 
    \begin{center}

         $J(e_1,e_1,e_1,f_1)=0 \  \Rightarrow \  1-6\mu_1+6\mu_1^2=0,$ \
         $J(e_1,e_1,e_1,f_2)=0 \  \Rightarrow \  (\mu_1-1)\mu_1(2\mu_1-1)=0,$ 

    \end{center}
    we get a contradiction.

\bigskip
\noindent
\underline{Let $\jor_0 \cong \mathcal{B}_3$.}
Here we are looking for Jordan superalgebras such that $\jor=(\mathbb{F}e_1 +\mathbb{F}e_2)+(\mathbb{F}f_1+\mathbb{F}f_2)$ with multiplication rules
$$e_1^2=e_2, \ e_1 f_1=\alpha_1 f_1 +\alpha_2 f_2, \ e_1 f_2=\alpha_3 f_1 +\alpha_4 f_2,$$
$$e_2 f_1=\beta_1 f_1 +\beta_2 f_2, \ f_1 f_2 = \xi_1 e_1+\xi_2 e_2, \ e_2 f_2=\beta_3 f_1 +\beta_4 f_2.$$

For the action of the operator $L_{e_{1}}$ on $\jor_1$ we can write the following matrix
\begin{center}
    \begin{tabular}{c}
        $\begin{pmatrix}
        \alpha_1 & \alpha_2 \\  
        \alpha_3 & \alpha_4 
        \end{pmatrix}.$ 
    \end{tabular}
\end{center}

\noindent
However, it is easy to prove that, by using a simple change of basis, the matrix of $L_{e_{1}}$ has one of the following forms:
\begin{center}
    \begin{tabular}{cc}
        $\begin{pmatrix}
        \mu_1 & 0 \\  
        0 & \mu_2 
        \end{pmatrix},$ & 
        $\begin{pmatrix}
        \mu_1 & 1 \\  
        0 & \mu_1 
        \end{pmatrix}.$ 
    \end{tabular}
\end{center}

\noindent
\textbf{i)} Let $L_{e_{1}} \simeq       \begin{pmatrix}
        \mu_1 & 0 \\  
        0 & \mu_2 
    \end{pmatrix}$ then the rule of multiplication can be written as follows:
\begin{center}

       $e_1^2=e_2,$  \ $e_1 f_1=\mu_1 f_1,$ \ $e_1 f_2=\mu_2 f_2,$ \ $e_2 f_1=\beta_1 f_1 +\beta_2 f_2,$ \ $f_1 f_2 = \xi_1 e_1+\xi_2 e_2, $ \ $e_2 f_2=\beta_3 f_1 +\beta_4 f_2.$

\end{center}

Let's assume that $\mu_1 \neq 0$. Then, from
\begin{center}

         $J(e_1,e_1,e_1,f_1)=0 \  \Rightarrow \  \beta_2=0,$ \
         $J(e_2,e_2,e_2,f_1)=0 \  \Rightarrow \  \beta_1=0,$ \
         $J(e_1,e_1,e_1,f_1)=0 \  \Rightarrow \  \mu_1=0,$

\end{center}
results, we get a contradiction. So $\mu_1=0.$

Further, let's assume that $\mu_2 \neq 0$. Then, from
\begin{center}
        $J(e_1,e_1,e_1,f_2)=0 \  \Rightarrow \  \beta_3=0,$ \
         $J(e_1,e_1,e_2,f_1)=0 \  \Rightarrow \  \beta_1=0,$
\end{center}
\begin{center}
         $J(e_2,e_2,e_2,f_2)=0 \  \Rightarrow \  \beta_4=0,$ \
         $J(e_1,e_1,e_1,f_2)=0 \  \Rightarrow \  \mu_2=0.$
\end{center}
results we get a contradiction again. So $\mu_2=0$.

\bigskip
\noindent
With obtained results we can consider the action of the operator $L_{e_{2}}$ on $\jor_1$ which can be written in one of the following forms:
\begin{center}
    \begin{tabular}{cc}
        $\begin{pmatrix}
        \tau_1 & 0 \\  
        0 & \tau_2 
        \end{pmatrix},$ & 
        $\begin{pmatrix}
        \tau_1 & 1 \\  
        0 & \tau_1 
        \end{pmatrix}.$ 
    \end{tabular}
\end{center}

\noindent
\textbf{i.1)} $L_{e_2} \simeq \begin{pmatrix}
        \tau_1 & 0 \\  
        0 & \tau_2 
        \end{pmatrix}$ gives the following results:
\begin{center}

         $J(e_1,e_1,e_2,f_1)=0 \  \Rightarrow \  \tau_1=0,$ \
         $J(e_1,e_1,e_2,f_2)=0 \  \Rightarrow \  \tau_2=0,$ 

\end{center}
by which we obtain the following superalgebra:
\begin{center}
    \begin{tabular}{l l}
       $e_1^2=e_2,$  & $f_1 f_2 = \xi_1 e_1+\xi_2 e_2. $
    \end{tabular}
\end{center}
\begin{itemize}
    \item[a)] If $\xi_1=0, \xi_2=0$ we have the superalgebra $\mathcal{J}_{66}.$
    \item[b)] If $\xi_1=0, \xi_2\neq 0,$ then by changing $f_1'=\frac{1}{\xi_2}f_1$ we get the superalgebra $\mathcal{J}_{67}.$
    \item[c)] If $\xi_1 \neq 0,$ then by changing $f_1'=\frac{1}{\xi_1}f_1$ and $e_1'=e_1+\frac{\xi_2}{\xi_1}e_2$ we obtain the superalgebra $\mathcal{J}_{68}.$
\end{itemize}

\noindent
\textbf{i.2)} $L_{e_2} \simeq \begin{pmatrix}
        \tau_1 & 1 \\  
        0 & \tau_1 
        \end{pmatrix}$ gives the following results:
\begin{center}

         $J(e_1,e_1,e_2,f_1)=0 \  \Rightarrow \  \tau_1=0,$ \
         $J(e_1,e_1,f_1,f_1)=0 \  \Rightarrow \  \xi_1=\xi_2=0,$ 

\end{center}
by which we obtain $\mathcal{J}_{69}$.

\noindent
\textbf{ii)} Let $L_{e_{1}} \simeq       \begin{pmatrix}
        \mu_1 & 1 \\  
        0 & \mu_1 
    \end{pmatrix}$ then the rule of multiplication can be written as follows:
\begin{center}

       $e_1^2=e_2,$ \ $e_1 f_1=\mu_1 f_1+f_2,$ \ $e_1 f_2=\mu_1 f_2,$ \ $e_2 f_1=\beta_1 f_1 +\beta_2 f_2,$ \ $f_1 f_2 = \xi_1 e_1+\xi_2 e_2, $ \ $e_2 f_2=\beta_3 f_1 +\beta_4 f_2.$

\end{center}

\noindent
Then we have the following results:
\begin{center}

         $J(e_1,e_1,f_1,e_1)=0 \  \Rightarrow \  \beta_3=0,$ \
         $J(e_2,e_2,e_2,f_2)=0 \  \Rightarrow \  \beta_4=0,$ \
         $J(e_1,e_1,e_1,f_2)=0 \  \Rightarrow \  \mu_1=0,$ \
         $J(e_1,e_1,f_1,e_1)=0 \  \Rightarrow \  \beta_1=0,$ \
         $J(e_1,e_1,f_1,f_1)=0 \  \Rightarrow \  \beta_2 \xi_1=0, \ \beta_2 \xi_2 = 2\xi_1.$

\end{center}
If $\beta_2=0$ then  we have $\xi_1=0$ and get the superalgebras $\mathcal{J}_{70}$ and $\mathcal{J}_{71}$.

\noindent
If $\beta_2 \neq 0$ then we have $\xi_1=\xi_2=0$. Then by changing the basis as $e_1'=e_1-\frac{1}{\beta_2}e_2, \ f_2'=\beta_2 f_2$ we obtain the superalgebra $$e_1'^2=e_2, \ e_2f_1=f_2'.$$ However, this superalgebra is isomorphic to $ \mathcal{J}_{69}$. 
\end{proof}

\begin{Th}
Up to isomorphism there are 59 Jordan superalgebras of type $(3,1)$, which are presented below with some additional information:

\renewcommand{\arraystretch}{1.2}
\begin{longtable}{l|c|l|l}
     \textnumero & Orbit & Multiplication rules & Decomposition \\
     \hline
     $\bf \mathfrak{J}_{1}$ & 15 & $e_1^2=e_1, \ e_2^2=e_2, \ e_3^2=e_3$ & $\U_1 \oplus \U_1 \oplus \U_1 \oplus \mathcal{S}_1^1$ \\
     $\bf \mathfrak{J}_{2}$ & 14 & $e_1^2=e_1, \ e_2^2=e_2, \ e_3^2=e_3, \ e_1f=f$ & $\S_2^2 \oplus \U_1 \oplus \U_1$\\
     $\bf \mathfrak{J}_{3}$ & 14 & $e_1^2=e_1, \ e_2^2=e_2, \ e_3^2=e_3, \ e_1f=\frac{1}{2}f$  & $\S_1^2 \oplus \U_1 \oplus \U_1$ \\
     $\bf \mathfrak{J}_{4}$ & 14 & $e_1^2=e_1, \ e_2^2=e_2, \ e_3^2=e_3, \ e_1f=\frac{1}{2}f, \ e_2f=\frac{1}{2}f$ & $\S_{13}^3 \oplus \U_1$ \\
     \hline
     $\bf \mathfrak{J}_{5}$ & 12 & $e_1^2=e_1, \ e_2^2=e_2$ & $\U_1 \oplus \U_1 \oplus \U_2 \oplus \mathcal{S}_1^1$ \\
     $\bf \mathfrak{J}_{6}$ & 13 & $e_1^2=e_1, \ e_2^2=e_2, \ e_1f=f$ & $\S_2^2 \oplus \U_1 \oplus \S_1^1$\\
     $\bf \mathfrak{J}_{7}$ & 13 & $e_1^2=e_1, \ e_2^2=e_2, \ e_1f=\frac{1}{2}f$  & $\S_1^2 \oplus \U_1 \oplus \S_1^1$ \\
     $\bf \mathfrak{J}_{8}$ & 13 & $e_1^2=e_1, \ e_2^2=e_2, \ e_1f=\frac{1}{2}f, \ e_2f=\frac{1}{2}f$ & $\S_{13}^3 \oplus \S_1^1$ \\
      \hline
     $\bf \mathfrak{J}_{9}$ & 7 & $e_1^2=e_1$ & $\U_1 \oplus \U_2 \oplus \U_2 \oplus \mathcal{S}_1^1$ \\
     $\bf \mathfrak{J}_{10}$ & 10 & $e_1^2=e_1, \ e_1f=f$ & $\S_2^2 \oplus \U_2 \oplus \S_1^1$\\
     $\bf \mathfrak{J}_{11}$ & 11 & $e_1^2=e_1, \ e_1f=\frac{1}{2}f$  & $\S_1^2 \oplus \U_2 \oplus \S_1^1$ \\
     \hline
     $\bf \mathfrak{J}_{12}$ & 14 & $e_1^2=e_1, \ e_1e_2=e_2, \ e_3^2=e_3$ & $\B_1 \oplus \U_1 \oplus \S_1^1$ \\
     $\bf \mathfrak{J}_{13}$ & 13 & $e_1^2=e_1, \ e_1e_2=e_2, \ e_3^2=e_3, \ e_3f=f$ & $\B_1 \oplus \S_2^2$\\
     $\bf \mathfrak{J}_{14}$ & 13& $e_1^2=e_1, \ e_1e_2=e_2, \ e_3^2=e_3, \ e_3f=\frac{1}{2}f$  & $\B_1 \oplus \S_2^1$ \\
     $\bf \mathfrak{J}_{15}$ & 12 & $e_1^2=e_1, \ e_1e_2=e_2, \ e_3^2=e_3, \ e_1f=f$  & $\S_{10}^3 \oplus \U_1$ \\
     $\bf \mathfrak{J}_{16}$ & 13 & $e_1^2=e_1, \ e_1e_2=e_2, \ e_3^2=e_3, \ e_1f=\frac{1}{2}f$  & $\S_{9}^3 \oplus \U_1$ \\
     $\bf \mathfrak{J}_{17}$ & 13 & $e_1^2=e_1, \ e_1e_2=e_2, \ e_3^2=e_3, \ e_1f=\frac{1}{2}f, \ e_3f=\frac{1}{2}f$ & Indecomposable \\
      \hline
     $\bf \mathfrak{J}_{18}$ & 11 & $e_1^2=e_1, \ e_1e_2=e_2$ & $\B_1 \oplus \U_2 \oplus \S_1^1$ \\
     $\bf \mathfrak{J}_{19}$ & 11 & $e_1^2=e_1, \ e_1e_2=e_2, \ e_1f=f$ & $\S_{10}^3 \oplus \U_2^1$\\
     $\bf \mathfrak{J}_{20}$ & 12 & $e_1^2=e_1, \ e_1e_2=e_2, \ e_1f=\frac{1}{2}f$  & $\S_9^3 \oplus \U_2^1$ \\
     \hline
     $\bf \mathfrak{J}_{21}$ & 13 & $e_1^2=e_1, \ e_1e_2=\frac{1}{2}e_2, \ e_3^2=e_3$ & $\B_2 \oplus \U_1 \oplus \S_1^1$ \\
     $\bf \mathfrak{J}_{22}$ & 12 & $e_1^2=e_1, \ e_1e_2=\frac{1}{2}e_2, \ e_3^2=e_3, \ e_3f=f$ & $\B_2 \oplus \S_2^2$\\
     $\bf \mathfrak{J}_{23}$ & 12 & $e_1^2=e_1, \ e_1e_2=\frac{1}{2}e_2, \ e_3^2=e_3, \ e_3f=\frac{1}{2}f$ & $\B_2 \oplus \S_1^1$\\
     $\bf \mathfrak{J}_{24}$ & 13 & $e_1^2=e_1, \ e_1e_2=\frac{1}{2}e_2, \ e_3^2=e_3, \ e_1f=f$ & $\S_{12}^3 \oplus \U_1$\\
     $\bf \mathfrak{J}_{25}$ & 10 & $e_1^2=e_1, \ e_1e_2=\frac{1}{2}e_2, \ e_3^2=e_3, \ e_1f=\frac{1}{2}f$ & $\S_{11}^3 \oplus \U_1$\\
     $\bf \mathfrak{J}_{26}$ & 12 & $e_1^2=e_1, \ e_1e_2=\frac{1}{2}e_2, \ e_3^2=e_3, \ e_1f=\frac{1}{2}f, \ e_3f=\frac{1}{2}f$  & Indecomposable\\
     \hline
     $\bf \mathfrak{J}_{27}$ & 10 & $e_1^2=e_1, \ e_1e_2=\frac{1}{2}e_2$ & $\B_2 \oplus \U_2 \oplus \S_1^1$ \\
     $\bf \mathfrak{J}_{28}$ & 12 & $e_1^2=e_1, \ e_1e_2=\frac{1}{2}e_2, \ e_1f=f$ & $\S_{12}^3 \oplus \U_2$\\
     $\bf \mathfrak{J}_{29}$ & 9 & $e_1^2=e_1, \ e_1e_2=\frac{1}{2}e_2, \ e_1f=\frac{1}{2}f$ & $\S_{11}^3 \oplus \U_2^1$\\
     \hline
     $\bf \mathfrak{J}_{30}$ & 11 & $e_1^2=e_2, \ e_3^2=e_3$ & $\B_3 \oplus \U_1 \oplus \S_1^1$ \\
     $\bf \mathfrak{J}_{31}$ & 12  & $e_1^2=e_2, \ e_3^2=e_3, \ e_3f=f$ & $\S_{2}^2 \oplus \U_1 \oplus \U_2$\\
     $\bf \mathfrak{J}_{32}$ & 12 & $e_1^2=e_2, \ e_3^2=e_3, \ e_3f=\frac{1}{2}f$ & $\S_{1}^2 \oplus \U_1 \oplus \U_2$\\
      \hline
     $\bf \mathfrak{J}_{33}$ & 6 & $e_1^2=e_2$ & $\B_3 \oplus \U_2 \oplus \S_1^1$ \\
     \hline
     $\bf \mathfrak{J}_{34}$ & 13 & $e_1^2=e_1, \ e_1e_2=e_2, \ e_1e_3=e_3, \ e_2^2=e_3$ & $\T_1 \oplus \S_1^1$\\
     $\bf \mathfrak{J}_{35}$ & 11 & $e_1^2=e_1, \ e_1e_2=e_2, \ e_1e_3=e_3, \ e_2^2=e_3, \ e_1f=f$ & Indecomposable\\
     $\bf \mathfrak{J}_{36}$ & 12 & $e_1^2=e_1, \ e_1e_2=e_2, \ e_1e_3=e_3, \ e_2^2=e_3, \ e_1f=\frac{1}{2}f$ & Indecomposable\\
      \hline
     $\bf \mathfrak{J}_{37}$ & 11 & $e_1^2=e_1, \ e_1e_2=e_2, \ e_1e_3=e_3$ & $\T_2 \oplus \S_1^1$\\
     $\bf \mathfrak{J}_{38}$ & 7 & $e_1^2=e_1, \ e_1e_2=e_2, \ e_1e_3=e_3, \ e_1f=f$ & Indecomposable\\
     $\bf \mathfrak{J}_{39}$ & 10 & $e_1^2=e_1, \ e_1e_2=e_2, \ e_1e_3=e_3, \ e_1f=\frac{1}{2}f$ & Indecomposable\\
      \hline
     $\bf \mathfrak{J}_{40}$ & 10 & $e_1^2=e_2, \ e_1e_2=e_3$ & $\T_3 \oplus \S_1^1$\\
      \hline
     $\bf \mathfrak{J}_{41}$ & 8 & $e_1^2=e_2, \ e_1e_3=e_2$ & $\T_4 \oplus \S_1^1$\\
     \hline
     $\bf \mathfrak{J}_{42}$ & 14 & $e_1^2=e_1, \ e_2^2=e_2, \ e_3^2=e_1+e_2, \ e_1e_3=\frac{1}{2}e_3, \ e_2e_3=\frac{1}{2}e_3$ & $\T_5 \oplus \S_1^1$\\
     $\bf \mathfrak{J}_{43}$ & 12 & $e_1^2=e_1, \ e_2^2=e_2, \ e_3^2=e_1+e_2, \ e_1e_3=\frac{1}{2}e_3, \ e_2e_3=\frac{1}{2}e_3, $ & Indecomposable\\
      & & $e_1f=\frac{1}{2}f, \ e_2f=\frac{1}{2}f$ & \\
      \hline
     $\bf \mathfrak{J}_{44}$ & 12 & $e_1^2=e_1, \ e_1e_2=\frac{1}{2}e_2, \ e_1e_3=e_3$ & $\T_6 \oplus \S_1^1$\\
     $\bf \mathfrak{J}_{45}$ & 10 & $e_1^2=e_1, \ e_1e_2=\frac{1}{2}e_2, \ e_1e_3=e_3, \ e_1f=f$ & Indecomposable\\
     $\bf \mathfrak{J}_{46}$ & 9 & $e_1^2=e_1, \ e_1e_2=\frac{1}{2}e_2, \ e_1e_3=e_3, \ e_1f=\frac{1}{2}f$ & Indecomposable\\
     \hline
     $\bf \mathfrak{J}_{47}$ & 9 & $e_1^2=e_1, \ e_1e_2=\frac{1}{2}e_2, \ e_1e_3=\frac{1}{2}e_3$ & $\T_7 \oplus \S_1^1$\\
     $\bf \mathfrak{J}_{48}$ & 9 & $e_1^2=e_1, \ e_1e_2=\frac{1}{2}e_2, \ e_1e_3=\frac{1}{2}e_3, \ e_1f=f$ & Indecomposable\\
     $\bf \mathfrak{J}_{49}$ & 4 & $e_1^2=e_1, \ e_1e_2=\frac{1}{2}e_2, \ e_1e_3=\frac{1}{2}e_3, \ e_1f=\frac{1}{2}f$ & Indecomposable\\
     \hline
     $\bf \mathfrak{J}_{50}$ & 12 & $e_1^2=e_1, \ e_1e_2=\frac{1}{2}e_2, \ e_2^2=e_3$ & $\T_8 \oplus \S_1^1$\\
     $\bf \mathfrak{J}_{51}$ & 13 & $e_1^2=e_1, \ e_1e_2=\frac{1}{2}e_2, \ e_2^2=e_3, \ e_1f=f$ & Indecomposable\\
     $\bf \mathfrak{J}_{52}$ & 11 & $e_1^2=e_1, \ e_1e_2=\frac{1}{2}e_2, \ e_2^2=e_3, \ e_1f=\frac{1}{2}f$ & Indecomposable\\
      \hline
     $\bf \mathfrak{J}_{53}$ & 13 & $e_1^2=e_1, \ e_1e_2=\frac{1}{2}e_2, \ e_2^2=e_3, \ e_1e_3=e_3$ & $\T_9 \oplus \S_1^1$\\
     $\bf \mathfrak{J}_{54}$ & 12 & $e_1^2=e_1, \ e_1e_2=\frac{1}{2}e_2, \ e_2^2=e_3, \ e_1e_3=e_3, \ e_1f=f$ & Indecomposable\\
     $\bf \mathfrak{J}_{55}$ & 11 & $e_1^2=e_1, \ e_1e_2=\frac{1}{2}e_2, \ e_2^2=e_3, \ e_1e_3=e_3, \ e_1f=\frac{1}{2}f$ & Indecomposable\\
      \hline
     $\bf \mathfrak{J}_{56}$ & 13 & $e_1^2=e_1, \ e_2^2=e_2, \ e_1e_3=\frac{1}{2}e_3, \ e_2e_3=\frac{1}{2}e_3$ & $\T_{10} \oplus \S_1^1$\\
     $\bf \mathfrak{J}_{57}$ & 13 & $e_1^2=e_1, \ e_2^2=e_2, \ e_1e_3=\frac{1}{2}e_3, \ e_2e_3=\frac{1}{2}e_3, \ e_1f=f$ & Indecomposable\\
     $\bf \mathfrak{J}_{58}$ & 12 & $e_1^2=e_1, \ e_2^2=e_2, \ e_1e_3=\frac{1}{2}e_3, \ e_2e_3=\frac{1}{2}e_3, \ e_1f=\frac{1}{2}f$ & Indecomposable\\
     $\bf \mathfrak{J}_{59}$ & 10 & $e_1^2=e_1, \ e_2^2=e_2, \ e_1e_3=\frac{1}{2}e_3, \ e_2e_3=\frac{1}{2}e_3, \ e_1f=\frac{1}{2}f, \ e_2f=\frac{1}{2}f$ & Indecomposable
     
\end{longtable}
\end{Th}

\begin{proof}

\underline{Let $\jor_0 \cong \mathcal{U}_1 \oplus \mathcal{U}_1 \oplus \mathcal{U}_1$.}
Here we are looking for Jordan superalgebras such that $\jor=(\mathbb{F}e_1 +\mathbb{F}e_2+\mathbb{F}e_3)+\mathbb{F}f_1$ with multiplication rules

\begin{center}

       $e_1^2=e_1,$  \ $e_3^2=e_3,$ \ $e_2^2 =e_2,$ \ $e_i f_1 = \beta_i f_1 , \ i=\overline{1,3}.$

\end{center}

Using the Jacobi super identity we obtain the following results:
\begin{center}
 $J(e_1,e_1,e_1,f_1)=0 \  \Rightarrow \  (\beta_1-1)\beta_1(2\beta_1-1)=0,$ \
         $J(e_2,e_2,e_2,f_1)=0 \  \Rightarrow \  (\beta_2-1)\beta_2(2\beta_2-1)=0,$
\end{center}

\begin{center}
         $J(e_3,e_3,e_3,f_1)=0 \  \Rightarrow \  (\beta_3-1)\beta_3(2\beta_3-1)=0,$ \
         $J(e_1,e_2,e_3,f_1)=0 \  \Rightarrow \  \beta_1\beta_2\beta_3=0.$

\end{center}

\begin{itemize}
    \item If $\beta_1=1$ then from $J(e_1,e_1,e_2,f_1)=0$ and $J(e_1,e_1,e_3,f_1)=0$ we get $\beta_2=0$ and $\beta_3=0$, respectively, which gives us the superalgebra $\mathfrak{J}_2$.
    \item If $\beta_1=0$ then 
        \begin{itemize}
            \item[a)] When $\beta_2=0$, from $J(e_3,e_3,e_3,f_1)=0$ we get three superalgebras with $\beta_3 \in \{0,1,\frac{1}{2}\}$. While $\beta_3=1$ gives a superalgebra that is isomorphic to $\mathfrak{J}_2$, from $\beta_3 \in \{0,\frac{1}{2}\}$ we obtain superalgebras $\mathfrak{J}_1$ and $\mathfrak{J}_3$. 
            \item[b)] When $\beta_2=1$, from $J(e_2,e_2,e_3,f_1)=0$ we get $\beta_3=0$, which gives $\mathfrak{J}_2.$
            \item[c)] When $\beta_2=\frac{1}{2}$, from $J(e_2,e_3,e_3,f_1)=0$ we get $\beta_3\in \{0,\frac{1}{2}\}$, which gives superalgebras isomorphic to $\mathfrak{J}_3$ and $\mathfrak{J}_4$.
        \end{itemize}
    \item If $\beta_1=\frac{1}{2}$ from $J(e_1,e_2,e_2,f_1)=0$ we get either $\beta_2=0$ or $\beta_2=\frac{1}{2}$. When the former occurs, we have two superalgebras with $\beta_3\in \{0,\frac{1}{2}\}$, and when the latter does we have $\beta_3=0$. However, all superalgebras obtained here are isomorphic to those of previous steps.
\end{itemize}

\noindent

\underline{Let $\jor_0 \cong \mathcal{U}_1 \oplus \mathcal{U}_1 \oplus \mathcal{U}_2$.}
Here we are looking for Jordan superalgebras such that $\jor=(\mathbb{F}e_1 +\mathbb{F}e_2+\mathbb{F}e_3)+\mathbb{F}f_1$ with multiplication rules

\begin{center}
    \begin{tabular}{l l l}
       $e_1^2=e_1$  & $e_2^2 =e_2$ &  $e_i f_1 = \beta_i f_1, \ i=\overline{1,3}.$
    \end{tabular}
\end{center}

Using the Jacobi super identity we obtain the following results:
\begin{center}
         $J(e_1,e_1,e_1,f_1)=0 \  \Rightarrow \  (\beta_1-1)\beta_1(2\beta_1-1)=0,$ 
\end{center}
\begin{center}
         $J(e_2,e_2,e_2,f_1)=0 \  \Rightarrow \  (\beta_2-1)\beta_2(2\beta_2-1)=0,$ \
         $J(e_3,e_3,e_3,f_1)=0 \  \Rightarrow \  \beta_3=0.$

\end{center}

\begin{itemize}
    \item If $\beta_1=1$ then from $J(e_1,e_1,e_2,f_1)=0$  we get $\beta_2=0$, which gives $\mathfrak{J}_6.$
    \item If $\beta_1=0$ then from $J(e_2,e_2,e_2,f_1)=0$ we get three superalgebras with $\beta_2 \in \{0,1,\frac{1}{2}\}$. New superalgebras here are $\mathfrak{J}_5$ and $\mathfrak{J}_7$.
    \item If $\beta_1=\frac{1}{2}$ from $J(e_1,e_2,e_2,f_1)=0$ we get either $\beta_2=0$ or $\beta_2=\frac{1}{2}$. The only new superlgebra here is $\mathfrak{J}_8.$
\end{itemize}

\noindent

\underline{Let $\jor_0 \cong \mathcal{U}_1 \oplus \mathcal{U}_2 \oplus \mathcal{U}_2$.}
Here we are looking for Jordan superalgebras such that $\jor=(\mathbb{F}e_1 +\mathbb{F}e_2 +\mathbb{F}e_3 )+\mathbb{F}f_1$ with multiplication rules

\begin{center}
    \begin{tabular}{l l}
       $e_1^2=e_1,$  & $e_i f_1 = \beta_i f_1, \ i=\overline{1,3}. $
    \end{tabular}
\end{center}

Using the Jacobi super identity we obtain the following results:
\begin{center}

         $J(e_1,e_1,e_1,f_1)=0 \  \Rightarrow \  (\beta_1-1)\beta_1(2\beta_1-1)=0,$

\end{center}
\begin{center}

         $J(e_2,e_2,e_2,f_1)=0 \  \Rightarrow \  \beta_2=0,$ \
         $J(e_3,e_3,e_3,f_1)=0 \  \Rightarrow \  \beta_3=0.$

\end{center}

In this case we have three superalgebras with $\beta_1\in \{0, 1, \frac{1}{2}\}$, which give us $\mathfrak{J}_9, \mathfrak{J}_{10}$ and $\mathfrak{J}_{11}.$

\noindent

\underline{Let $\jor_0 \cong \mathcal{U}_2 \oplus \mathcal{U}_2 \oplus \mathcal{U}_2$.}
Here we are looking for Jordan superalgebras such that $\jor=(\mathbb{F}e_1 +\mathbb{F}e_2 +\mathbb{F}e_3)+\mathbb{F}f_1$ with multiplication rules

\begin{center}
    \begin{tabular}{l}
       $e_i f_1 = \beta_i f_1, \ i=\overline{1,3}. $
    \end{tabular}
\end{center}
Using the Jacobi super identity we obtain the following results:
\begin{center}
         $J(e_1,e_1,e_1,f_1)=0 \  \Rightarrow \  \beta_1=0,$ \
         $J(e_2,e_2,e_2,f_1)=0 \  \Rightarrow \  \beta_2=0,$ \
         $J(e_3,e_3,e_3,f_1)=0 \  \Rightarrow \  \beta_3=0.$ 
\end{center}
So in this case we obtain a trivial superalgebra.

\noindent

\underline{Let $\jor_0 \cong \mathcal{B}_1 \oplus \mathcal{U}_1$.}
Here we are looking for Jordan superalgebras such that $\jor=(\mathbb{F}e_1 +\mathbb{F}e_2 +\mathbb{F}e_3)+\mathbb{F}f_1 $ with multiplication rules

\begin{center}

       $e_1^2=e_1,$  \ $e_3^2=e_3,$ \ $e_1 e_2 =e_2,$ \ $e_i f_1 = \beta_i f_1, \ i=\overline{1,3}. $ 

\end{center}
Using the Jacobi super identity we obtain the following results:
\begin{center}

         $J(e_2,e_2,e_2,f_1)=0 \  \Rightarrow \  \beta_2=0,$ \
         $J(e_1,e_1,e_1,f_1)=0 \  \Rightarrow \  (\beta_1-1)\beta_1(2\beta_1-1)=0,$ \
         $J(e_3,e_3,e_3,f_1)=0 \  \Rightarrow \  (\beta_3-1)\beta_3(2\beta_3-1)=0,$ \
         $J(e_1,e_1,e_3,f_1)=0 \  \Rightarrow \  \beta_1 \beta_3 (2\beta_1-1)=0,$ \
         $J(e_1,e_3,e_3,f_1)=0 \  \Rightarrow \  \beta_1 \beta_3 (2\beta_3-1)=0.$

\end{center}
\begin{itemize}
    \item If $\beta_1=0$, then $\beta_3 \in \{0,1,\frac{1}{2}\}$, which gives $\mathfrak{J}_{12}, \ \mathfrak{J}_{13}$ and $\mathfrak{J}_{14}.$
    \item If $\beta_1=1$, then $\beta_3=0$, which gives us the superalgebra $\mathcal{J}_{15}$.
    \item If $\beta_1=\frac{1}{2}$, then $\beta_3 \in \{0,\frac{1}{2}\}$, which gives us $\mathfrak{J}_{16}$ and $\mathfrak{J}_{17}$.
\end{itemize}

\noindent

\underline{Let $\jor_0 \cong \mathcal{B}_1 \oplus \mathcal{U}_2$.}
Here we are looking for Jordan superalgebras such that $\jor=(\mathbb{F}e_1 +\mathbb{F}e_2+\mathbb{F}e_3)+\mathbb{F}f_1$ with multiplication rules

\begin{center}
    \begin{tabular}{l l l}
       $e_1^2=e_1,$  & $e_1 e_2 =e_2,$  & $e_i f_1 = \beta_i f_1, \ i=\overline{1,3}.$
    \end{tabular}
\end{center}
Using the Jacobi super identity we obtain the following results:
\begin{center}

         $J(e_2,e_2,e_2,f_1)=0 \  \Rightarrow \  \beta_2=0,$ \
         $J(f_1,f_1,f_1,e_3)=0 \  \Rightarrow \  \beta_3=0,$ \
         $J(e_1,e_1,e_1,f_1)=0 \  \Rightarrow \  (\beta_1-1)\beta_1(2\beta_1-1)=0.$

\end{center}
So, we have three superalgebras with $\beta_1 \in \{0,1,\frac{1}{2}\}$, which give us $\mathfrak{J}_{18}, \ \mathfrak{J}_{19}$ and $\mathfrak{J}_{20}.$

\noindent

\underline{Let $\jor_0 \cong \mathcal{B}_2 \oplus \mathcal{U}_1$.}
Here we are looking for Jordan superalgebras such that $\jor=(\mathbb{F}e_1+\mathbb{F}e_2+\mathbb{F}e_3)+\mathbb{F}f_1$ with multiplication rules

\begin{center}

       $e_1^2=e_1,$ \ $e_3^2=e_3,$ \ $e_1 e_2 =\frac{1}{2}e_2,$ \ $e_i f_1 = \beta_i f_1, \ i=\overline{1,3}. $

\end{center}
Using the Jacobi super identity we obtain the following results:
\begin{center}

        $J(e_2,e_2,e_2,f_1)=0 \  \Rightarrow \  \beta_2=0,$ \
        $J(e_1,e_1,e_1,f_1)=0 \  \Rightarrow \  (\beta_1-1)\beta_1(2\beta_1-1)=0,$ \
        $J(e_3,e_3,e_3,f_1)=0 \  \Rightarrow \  (\beta_3-1)\beta_3(2\beta_3-1)=0,$ \
        $J(e_1,e_1,e_3,f_1)=0 \  \Rightarrow \  \beta_1 \beta_3 (2\beta_1-1)=0,$ \
        $J(e_1,e_3,e_3,f_1)=0 \  \Rightarrow \  \beta_1 \beta_3 (2\beta_3-1)=0.$

\end{center}
\begin{itemize}
    \item If $\beta_1=0$, then $\beta_3 \in \{0,1,\frac{1}{2}\}$ which gives us $\mathfrak{J}_{21}, \ \mathfrak{J}_{22}$ and $\mathfrak{J}_{23}$.
    \item If $\beta_1=1$, then $\beta_3=0$, which gives us the superalgebra $\mathfrak{J}_{24}$.

    \item If $\beta_1=\frac{1}{2}$, then $\beta_3 \in \{0,\frac{1}{2}\}$, which gives $\mathfrak{J}_{25}$ and $\mathfrak{J}_{26}$.
\end{itemize}

\noindent

\underline{Let $\jor_0 \cong \mathcal{B}_2 \oplus \mathcal{U}_2$.}
Here we are looking for Jordan superalgebras such that $\jor=(\mathbb{F}e_1 +\mathbb{F}e_2+\mathbb{F}e_3)+\mathbb{F}f_1$ with multiplication rules

\begin{center}
    \begin{tabular}{l l l}
       $e_1^2=e_1,$  & $e_1 e_2 =\frac{1}{2}e_2$ & $e_i f_1 = \beta_i f_1, \ i=\overline{1,3}.$ 
    \end{tabular}
\end{center}
Using the Jacobi super identity we obtain the following results:
\begin{center}

        $J(e_2,e_2,e_2,f_1)=0 \  \Rightarrow \  \beta_2=0,$ \
        $J(e_3,e_3,e_3,f_1)=0 \  \Rightarrow \  \beta_3=0,$ \
        $J(e_1,e_1,e_1,f_1)=0 \  \Rightarrow \  (\beta_1-1)\beta_1(2\beta_1-1)=0.$

\end{center}
So we have superalgebras $\mathfrak{J}_{27}, \ \mathfrak{J}_{28}$ and $\mathfrak{J}_{29}$ from $\beta_1 \in \{0,1,\frac{1}{2}\}$.

\noindent

\underline{Let $\jor_0 \cong \mathcal{B}_3 \oplus \mathcal{U}_1$.}
Here we are looking for Jordan superalgebras such that $\jor=(\mathbb{F}e_1 +\mathbb{F}e_2 +\mathbb{F}e_3 )+\mathbb{F}f_1$ with multiplication rules

\begin{center}
    \begin{tabular}{l l l}
       $e_1^2=e_2,$  & $e_3^2=e_3$ & $e_i f_1 = \beta_i f_1, \ i=\overline{1,3}.$ 
    \end{tabular}
\end{center}
Using the Jacobi super identity we obtain the following results:
\begin{center}

        $J(e_2,e_2,e_2,f_1)=0 \  \Rightarrow \  \beta_2=0,$ \
        $J(e_1,e_1,e_1,f_1)=0 \  \Rightarrow \  \beta_1=0,$ \
        $J(e_3,e_3,e_3,f_1)=0 \  \Rightarrow \  (\beta_3-1)\beta_3(2\beta_3-1)=0.$

\end{center}
So we have Jordan superalgebras $\mathfrak{J}_{30}, \ \mathfrak{J}_{31}$ and $\mathfrak{J}_{32}$ from $\beta_3 \in \{0,1,\frac{1}{2}\}$.

\noindent

\underline{Let $\jor_0 \cong \mathcal{B}_3 \oplus \mathcal{U}_2$.}
Here we are looking for Jordan superalgebras such that $\jor=(\mathbb{F}e_1 +\mathbb{F}e_2 +\mathbb{F}e_3)+\mathbb{F}f_1$ with multiplication rules

\begin{center}
    \begin{tabular}{l l}
       $e_1^2=e_2,$  & $e_i f_1 = \beta_i f_1, \ i=\overline{1,3}.$  
    \end{tabular}
\end{center}
Using the Jacobi super identity we obtain the following results:
\begin{center}

        $J(e_2,e_2,e_2,f_1)=0 \  \Rightarrow \  \beta_2=0,$ \
        $J(e_1,e_1,e_1,f_1)=0 \  \Rightarrow \  \beta_1=0,$\
        $J(e_3,e_3,e_3,f_1)=0 \  \Rightarrow \  \beta_3=0.$

\end{center}
So the superalgebra in this case is $\mathfrak{J}_{33}$.

\noindent

\underline{Let $\jor_0 \cong \mathcal{T}_1$.}
Here we are looking for Jordan superalgebras such that $\jor=(\mathbb{F}e_1 +\mathbb{F}e_2+\mathbb{F}e_3)+\mathbb{F}f_1$ with multiplication rules

\begin{center}
    \begin{tabular}{l l l l l}
       $e_1^2=e_1$, & $e_2^2 =e_3$, & $e_1 e_2 =e_2$,  & $e_1 e_3=e_3$, & $e_i f_1 = \beta_i f_1, \ i=\overline{1,3}. $
    \end{tabular}
\end{center}

Using the Jacobi super identity we obtain the following results:
\begin{center}

         $J(e_3,e_3,e_3,f_1)=0 \  \Rightarrow \  \beta_3=0,$ \
         $J(e_2,e_2,e_2,f_1)=0 \  \Rightarrow \  \beta_2=0,$ \
         $J(e_1,e_1,e_1,f_1)=0 \  \Rightarrow \  (\beta_1-1)\beta_1(2\beta_1-1)=0.$

\end{center}

Thus, from $\beta_1 \in \{0,1,\frac{1}{2}\}$, we get $\mathfrak{J}_{34}, \ \mathfrak{J}_{35}$ and $\mathfrak{J}_{36}.$

\noindent

\underline{Let $\jor_0 \cong \mathcal{T}_2$.}
Here we are looking for Jordan superalgebras such that $\jor=(\mathbb{F}e_1 +\mathbb{F}e_2+ \mathbb{F}e_3 )+ \mathbb{F}f_1$ with multiplication rules

\begin{center}

       $e_1^2=e_1,$ \ $e_1 e_3=e_3,$ \ $e_1 e_2 =e_2,$ \ $e_i f_1 = \beta_i f_1, \ i=\overline{1,3}.$ 

\end{center}

Using the Jacobi super identity we obtain the following results:
\begin{center}

         $J(e_2,e_2,e_2,f_1)=0 \  \Rightarrow \  \beta_2=0,$ \
         $J(e_3,e_3,e_3,f_1)=0 \  \Rightarrow \  \beta_3=0,$ \
         $J(e_1,e_1,e_1,f_1)=0 \  \Rightarrow \  (\beta_1-1)\beta_1(2\beta_1-1)=0.$

\end{center}

This we obtain $\mathfrak{J}_{37}, \ \mathfrak{J}_{38}$ and $\mathfrak{J}_{39}$ from $\beta_1 \in \{0,1,\frac{1}{2}\}.$

\noindent

\underline{Let $\jor_0 \cong \mathcal{T}_3$.}
Here we are looking for Jordan superalgebras such that $\jor=(\mathbb{F}e_1 +\mathbb{F}e_2+\mathbb{F}e_3)+\mathbb{F}f_1$ with multiplication rules

\begin{center}
    \begin{tabular}{l l}
       $e_1^2=e_2,$  & $e_i f_1 = \beta_i f_1, \ i=\overline{1,3}.$
    \end{tabular}
\end{center}

Using the Jacobi super identity we obtain the following results:
\begin{center}

         $J(e_2,e_2,e_2,f_1)=0 \  \Rightarrow \  \beta_2=0,$ \
         $J(e_3,e_3,e_3,f_1)=0 \  \Rightarrow \  \beta_3=0,$ \
         $J(e_1,e_1,e_1,f_1)=0 \  \Rightarrow \  \beta_1=0.$

\end{center}
Thus we have $\mathfrak{J}_{40}.$

\noindent

\underline{Let $\jor_0 \cong \mathcal{T}_4$.}
Here we are looking for Jordan superalgebras such that $\jor=(\mathbb{F}e_1 +\mathbb{F}e_2+ \mathbb{F}e_3)+ \mathbb{F}f_1$ with multiplication rules

\begin{center}
    \begin{tabular}{l l l}
       $e_1^2=e_2,$  & $e_1 e_3=e_2,$ & $e_i f_1 = \beta_i f_1, \ i=\overline{1,3}.$  
    \end{tabular}
\end{center}

Using the Jacobi super identity we obtain only the following result:
\begin{center}

         $J(e_2,e_2,e_2,f_1)=0 \  \Rightarrow \  \beta_2=0,$ \
         $J(e_3,e_3,e_3,f_1)=0 \  \Rightarrow \  \beta_3=0,$ \
         $J(e_1,e_1,e_1,f_1)=0 \  \Rightarrow \  \beta_1=0.$

\end{center}
Thus we obtain $\mathfrak{J}_{41}.$

\noindent

\underline{Let $\jor_0 \cong \mathcal{T}_5$.}
Here we are looking for Jordan superalgebras such that $\jor=(\mathbb{F}e_1 +\mathbb{F}e_2+\mathbb{F}e_3)+\mathbb{F}f_1$ with multiplication rules

\begin{center}

       $e_1^2=e_1$ \ $e_3^2=e_1+e_2,$ \ $e_1 e_3=\frac{1}{2}e_3$  \ $e_2^2 =e_2,$ \
       $e_2 e_3=\frac{1}{2} e_3$ \ $e_i f_1 = \beta_i f_1, \ i=\overline{1,3}.$

\end{center}

Using the Jacobi super identity we obtain only the following result:
\begin{center}
        \begin{tabular}{l}
         $J(e_1,e_1,e_1,f_1)=0 \  \Rightarrow \  (\beta_1-1)\beta_1(2\beta_1-1)=0.$
         \end{tabular}
\end{center}
If $\beta_1=0$ then from 
\begin{center}

         $J(e_1,e_1,e_3,f_1)=0 \  \Rightarrow \  \beta_3=0,$ \
         $J(e_1,f_1,e_3,e_3)=0 \  \Rightarrow \  \beta_2=0,$

\end{center}
we get $\mathfrak{J}_{42}$.

\noindent
If $\beta_1=1$ then from 
\begin{center}
         $J(e_1,e_1,e_2,f_1)=0 \  \Rightarrow \  \beta_2=0,$ \
         $J(e_1,e_1,e_3,f_1)=0 \  \Rightarrow \  \beta_3=0,$ \
         $J(e_1,f_1,e_3,e_3)=0 \  \Rightarrow \  -\frac{1}{2}(1+\beta_2-2\beta_3^2)=0.$

\end{center}
we get a contradiction. So there is no superalgebra in this case.

\noindent
If $\beta_1=\frac{1}{2}$ then from 
\begin{center}

         $J(e_1,e_3,e_1,f_1)=0 \  \Rightarrow \  \beta_3=0,$ \
         $J(e_3,e_3,e_1,f_1)=0 \  \Rightarrow \  \beta_2=\frac{1}{2}.$ 

\end{center}

we have $\mathfrak{J}_{43.}$

\noindent

\underline{Let $\jor_0 \cong \mathcal{T}_6$.}
Here we are looking for Jordan superalgebras such that $\jor=(e_1 \mathbb{F}+e_2 \mathbb{F}+e_3 \mathbb{F})+f_1 \mathbb{F}$ with multiplication rules

\begin{center}

       $e_1^2=e_1,$ \ $e_1 e_2 =\frac{1}{2}e_2,$ \ $e_1 e_3=e_3,$ \ $e_i f_1 = \beta_i f_1, \ i=\overline{1,3}. $

\end{center}

Using the Jacobi super identity we obtain the following results:
\begin{center}
   
         $J(e_1,e_1,e_1,f_1)=0 \  \Rightarrow \  (\beta_1-1)\beta_1(2\beta_1-1)=0,$ \
         $J(e_2,e_2,e_2,f_1)=0 \  \Rightarrow \  \beta_2=0,$ \
         $J(f_1,e_3,e_3,e_3)=0 \  \Rightarrow \  \beta_3=0.$
   
\end{center}
Thus we get $\mathfrak{J}_{44}, \ \mathfrak{J}_{45}$ and $\mathfrak{J}_{46}$ from $\beta_1=0, \beta_1=1$ and $\beta_1=\frac{1}{2}$, respectively.

\noindent

\underline{Let $\jor_0 \cong \mathcal{T}_7$.}
Here we are looking for Jordan superalgebras such that $\jor=(\mathbb{F}e_1 +\mathbb{F}e_2+\mathbb{F}e_3)+\mathbb{F}f_1$ with multiplication rules

\begin{center}

       $e_1^2=e_1,$ \ $e_1 e_2 =\frac{1}{2}e_2,$ \ $e_1 e_3=\frac{1}{2}e_3,$ \ $e_i f_1 = \beta_i f_1  \ i=\overline{1,3}.$

\end{center}

Using the Jacobi super identity we obtain the following results:
\begin{center}
        \begin{tabular}{l}
         $J(e_2,e_2,e_2,f_1)=0 \  \Rightarrow \  \beta_2=0,$ \\
         $J(e_3,e_3,e_3,f_1)=0 \  \Rightarrow \  \beta_3=0,$ \\
         $J(e_1,e_1,e_1,f_1)=0 \  \Rightarrow \  (\beta_1-1)\beta_1(2\beta_1-1)=0.$
         \end{tabular}
\end{center}
So, we get $\mathfrak{J}_{47}, \ \mathfrak{J}_{48}$ and $\mathfrak{J}_{49}$ superalgebras from $\beta_1=0, \beta_1=1$ and $\beta_1=\frac{1}{2}$, respectively.

\noindent

\underline{Let $\jor_0 \cong \mathcal{T}_8$.}
Here we are looking for Jordan superalgebras such that $\jor=(\mathbb{F}e_1 +\mathbb{F}e_2+ \mathbb{F}e_3)+ \mathbb{F}f_1$ with multiplication rules

\begin{center}

       $e_1^2=e_1,$  \ $e_1 e_2 =\frac{1}{2}e_2,$ \ $e_2^2 =e_3,$ \ $e_i f_1 = \beta_i f_1, \ i=\overline{1,3}. $

\end{center}

Using the Jacobi super identity we obtain the following results:
\begin{center}

         $J(e_3,e_3,e_3,f_1)=0 \  \Rightarrow \  \beta_3=0,$ \
         $J(e_2,e_2,e_2,f_1)=0 \  \Rightarrow \  \beta_2=0,$ \
         $J(e_1,e_1,e_1,f_1)=0 \  \Rightarrow \  (\beta_1-1)\beta_1(2\beta_1-1)=0.$

\end{center}
Thus we get $\mathfrak{J}_{50}, \ \mathfrak{J}_{51}$ and $\mathfrak{J}_{52}$.

\noindent

\underline{Let $\jor_0 \cong \mathcal{T}_9$.}
Here we are looking for Jordan superalgebras such that $\jor=(\mathbb{F}e_1 +\mathbb{F}e_2+ \mathbb{F}e_3)+ \mathbb{F}f_1 $ with multiplication rules

\begin{center}
    \begin{tabular}{l l l l l}
       $e_1^2=e_1$,  & $e_1 e_2 =\frac{1}{2}e_2$, & $e_1 e_3=e_3$, & $e_2^2 =e_3$, & $e_i f_1 = \beta_i f_1, \ i=\overline{1,3}. $ 
    \end{tabular}
\end{center}

Using the Jacobi super identity we obtain the following results:
\begin{center}

         $J(e_3,e_3,e_3,f_1)=0 \  \Rightarrow \  \beta_3=0,$ \
         $J(e_2,e_2,e_2,f_1)=0 \  \Rightarrow \  \beta_2=0,$ \
         $J(e_1,e_1,e_1,f_1)=0 \  \Rightarrow \  (\beta_1-1)\beta_1(2\beta_1-1)=0.$

\end{center}
Thus we obtain $\mathfrak{J}_{53}, \ \mathfrak{J}_{54}$ and $\mathfrak{J}_{55}$

\noindent

\underline{Let $\jor_0 \cong \mathcal{T}_{10}$.}
Here we are looking for Jordan superalgebras such that $\jor=(\mathbb{F}e_1 +\mathbb{F}e_2+ \mathbb{F}e_3 )+ \mathbb{F}f_1$ with multiplication rules

\begin{center}
 
       $e_1^2=e_1$, \ $e_1 e_3=\frac{1}{2}e_3$, \ $e_2^2 =e_2$, \  $e_2 e_3=\frac{1}{2}e_3$, \  $e_i f_1 = \beta_i f_1, \ i=\overline{1,3}. $ 

\end{center}

\noindent
Using the Jacobi super identity we obtain the following results:
\begin{center}

         $J(e_3,e_3,e_3,f_1)=0 \  \Rightarrow \  \beta_3=0,$ \
         $J(e_1,e_1,e_1,f_1)=0 \  \Rightarrow \  (\beta_1-1)\beta_1(2\beta_1-1)=0.$

\end{center}

\noindent
If $\beta_1=0$ then 
\begin{center}
        \begin{tabular}{l}
         $J(e_2,e_2,e_2,f_1)=0 \  \Rightarrow \  (\beta_2-1)\beta_2(2\beta_2-1)=0.$ 
         \end{tabular}
\end{center}

So we have three superalgebras isomorphic to $\mathfrak{J}_{56}, \ \mathfrak{J}_{57}$ and $\mathfrak{J}_{58}$.

\noindent
If $\beta_1=1$ then  $J(e_1,e_1,e_2,f_1)=0 \  \Rightarrow \  \beta_2=0$  gives us $\mathfrak{J}_{57}$.

\noindent
If $\beta_1=\frac{1}{2}$ then  $J(e_1,e_2,e_2,f_1)=0 \  \Rightarrow \  \frac{1}{2} \beta_2(2\beta_2 -1)=0.$ 
So we have $\mathfrak{J}_{58}$ and $\mathfrak{J}_{59}$.
\end{proof}

\section{Irreducible components}

\begin{Th}\label{geo1}
The variety of $4$-dimensional Jordan superalgebras of type  $(1,3)$ has 
dimension  $12$ and it has  $11$  irreducible components defined by  
\begin{center}
$\mathcal{C}_1=\overline{\{ {\bf J}_{5}\}},$ \
$\mathcal{C}_2=\overline{\{ {\bf J}_{7}\}},$ \ 
$\mathcal{C}_3=\overline{\{ {\bf J}_{8}\}},$ \  
$\mathcal{C}_4=\overline{\{ {\bf J}_{9}\}},$ \
$\mathcal{C}_5=\overline{\{ {\bf J}_{11}\}},$ \\

$\mathcal{C}_6=\overline{\{ {\bf J}_{12}\}},$ \
$\mathcal{C}_7=\overline{\{ {\bf J}_{14}\}},$ \ 
$\mathcal{C}_8=\overline{\{ {\bf J}_{15}\}},$ \  
$\mathcal{C}_9=\overline{\{ {\bf J}_{16}\}},$ \
$\mathcal{C}_{10}=\overline{\{ {\bf J}_{17}\}},$ \ $\mathcal{C}_{11}=\overline{\{ {\bf J}_{19}\}},$ \\

\end{center}

In particular, all of them are rigid superalgebras.
 
\end{Th}

\begin{proof}
After carefully  checking  the dimensions of orbit closures of the more important for us superalgebras, we have 

\begin{longtable}{rcl}
      
$\dim  \mathcal{O}({\bf J}_{5})=\dim  \mathcal{O}({\bf J}_{12})=\dim  \mathcal{O}({\bf J}_{14})$&$=$&$12,$ \\ 
$\dim  \mathcal{O}({\bf J}_{8})=
\dim  \mathcal{O}({\bf J}_{9})=
\dim  \mathcal{O}({\bf J}_{11})= 
\dim  \mathcal{O}({\bf J}_{17})= 
\dim  \mathcal{O}({\bf J}_{19}) $&$=$&$10,$ \\ 

$\dim  \mathcal{O}({\bf J}_{16})$&$=$&$9,$ \\ 

$\dim  \mathcal{O}({\bf J}_{7})$&$=$&$7,$ \\ 

$\dim  \mathcal{O}({\bf J}_{15})$&$=$&$4.$ 
\end{longtable}   

If $E_{f_1}^t, E_{f_2}^t, E_{f_3}^t, E_{e}^t$ is a {\it parametric basis} for ${\bf A}\to {\bf B}$, then we denote a degeneration by ${\bf A}\xrightarrow{(E_{f_1}^t, E_{f_2}^t, E_{f_3}^t, E_{e}^t)} {\bf B}$.

\begin{longtable}{lcl|lcl} \hline

${\bf J}_{3}$ & $\xrightarrow{ (tf_1, f_2+f_3+te, f_3+te, te)}$ & ${\bf J}_{1}$ &  ${\bf J}_{5}$ & $\xrightarrow{ (tf_1, tf_2, f_3,e)}$ & ${\bf J}_{2}$ 
\\  \hline

${\bf J}_{5}$ & $\xrightarrow{ (f_1-f_3, f_2, tf_3-te,e)}$ & ${\bf J}_{3}$ & 

${\bf J}_{5}$ & $\xrightarrow{ (tf_1+f_2, tf_2, f_3,e)}$ & ${\bf J}_{4}$
\\  \hline

${\bf J}_{12}$ & $\xrightarrow{ (f_1+2f_2+2f_3, tf_2+2tf_3, t^2f_3, te)}$ & ${\bf J}_{6}$ 

&
${\bf J}_{14}$ & $\xrightarrow{ (f_1, f_2, tf_3, e)}$ & ${\bf J}_{10}$ 

 \\  \hline

${\bf J}_{14}$ & $\xrightarrow{ (f_1, f_2, tf_3, e)}$ & ${\bf J}_{13}$ & 
${\bf J}_{19}$ & $\xrightarrow{ (tf_1, f_2, f_3, e)}$ & ${\bf J}_{18}$\\  \hline
\end{longtable}

Below we list all the important reasons for necessary non-degenerations.

\begin{longtable}{lcl|l}
\hline
    \multicolumn{4}{c}{Non-degenerations reasons} \\
\hline

${\bf J}_{5}$ & $\not \rightarrow  $ & 

$\begin{array}{llll}
{\bf J}_{7}, {\bf J}_{8}, {\bf J}_{9}, 
{\bf J}_{11},{\bf J}_{15}, {\bf J}_{16}, {\bf J}_{17},{\bf J}_{19}\\
\end{array}$ 
& 
$\mathcal R=\left\{\begin{array}{l}
\mbox{According to Lemma~\ref{lema:inv} (2)}
\end{array}\right\}
$\\

\hline

${\bf J}_{12}$ & $\not \rightarrow  $ & 

$\begin{array}{llll}
{\bf J}_{7},{\bf J}_{8}, {\bf J}_{9}, {\bf J}_{11}, {\bf J}_{15}, {\bf J}_{16}, {\bf J}_{17}, {\bf J}_{19}\\
\end{array}$ 
& 
$\mathcal R=\left\{\begin{array}{l}
JJ\subset \{e,f_2,f_3\}, \
c_{44}^4=2c_{24}^2, \ c_{44}^4=c_{34}^3
\end{array}\right\}
$\\

\hline

${\bf J}_{14}$ & $\not \rightarrow  $ & 

$\begin{array}{llll}
{\bf J}_{7}, {\bf J}_{8}, {\bf J}_{9}, {\bf J}_{11}, {\bf J}_{15}, {\bf J}_{16}, {\bf J}_{17}, {\bf J}_{19} \\
\end{array}$ 
& 
$\mathcal R=\left\{\begin{array}{lllll}
JJ\subset \{e,f_2,f_3\}, \
c_{44}^4=c_{34}^3, \ c_{42}^2=c_{34}^3
\end{array}\right\}
$\\

\hline

${\bf J}_{8}$ & $\not \rightarrow  $ & 

$\begin{array}{llll}
{\bf J}_{7}, {\bf J}_{15}, {\bf J}_{16}\\
\end{array}$ 
& 
$\mathcal R=\left\{\begin{array}{l}
c_{ij}^1=0, \ 2c_{34}^3=c_{44}^4
\end{array}\right\}
$\\

\hline

${\bf J}_{9}$ & $\not \rightarrow  $ & 

$\begin{array}{llll}
{\bf J}_{7}, {\bf J}_{15}, {\bf J}_{16}\\
\end{array}$ 
& 
$\mathcal R=\left\{\begin{array}{l}
c_{ij}^1=0, \ c_{34}^3=c_{44}^4
\end{array}\right\}
$\\

\hline

${\bf J}_{11}$ & $\not \rightarrow  $ & 

$\begin{array}{llll}
{\bf J}_{7},  {\bf J}_{15}, {\bf J}_{16}\\
\end{array}$ 
& 
$\mathcal R=\left\{\begin{array}{l}
c_{ij}^1=0, \ 2c_{34}^3=c_{44}^4
\end{array}\right\}
$\\

\hline

${\bf J}_{17}$ & $\not \rightarrow  $ & 

$\begin{array}{llll}
{\bf J}_{7},  {\bf J}_{15}, {\bf J}_{16}\\
\end{array}$ 
& 
$\mathcal R=\left\{\begin{array}{l}

A_2A_2\subset A_2, \ 2c_{14}^1=c_{44}^4, \ c_{24}^2=c_{44}^4, \ c_{34}^3=c_{44}^4

\end{array}\right\}
$\\

\hline

${\bf J}_{19}$ & $\not \rightarrow  $ & 

$\begin{array}{llll}
{\bf J}_{7},  {\bf J}_{15}, {\bf J}_{16}\\
\end{array}$ 
& 
$\mathcal R=\left\{\begin{array}{l}
A_2A_2\subset A_2, \
c_{14}^1=c_{44}^4, \ c_{24}^2=c_{44}^4, \ c_{34}^3=c_{44}^4
\end{array}\right\}
$\\

\hline

${\bf J}_{16}$ & $\not \rightarrow  $ & 

$\begin{array}{llll}
{\bf J}_{7},  {\bf J}_{15}\\
\end{array}$ 
& 
$\mathcal R=\left\{\begin{array}{l}
A_2A_2\subset A_2, \
2c_{14}^1=c_{44}^4, \ 2c_{24}^2=c_{44}^4, \ c_{34}^3=c_{44}^4
\end{array}\right\}
$\\

\hline

${\bf J}_{7}$ & $\not \rightarrow  $ & 

$\begin{array}{llll}
{\bf J}_{15}\\
\end{array}$ 
& 
$\mathcal R=\left\{\begin{array}{l}
JJ\subset \{e\} \\
\end{array}\right\}
$\\

\hline

\end{longtable}

Here $c_{ij}^{k}$ coefficients are structural constants in the $x_1=f_1, \ x_2=f_2, \ x_3=f_3, \ x_4=e$ basis.

\end{proof}

\begin{Th}\label{geo2}
The variety of $4$-dimensional Jordan superalgebras of type  $(2,2)$ has 
dimension  $13$ and it has  $25$  irreducible components defined by  
\begin{center}
$\mathcal{C}_1=\overline{\{ {\bf \mathcal{J}}_{1}\}},$ \
$\mathcal{C}_2=\overline{\{ {\bf \mathcal{J}}_{2}\}},$ \ 
$\mathcal{C}_3=\overline{\{ {\bf \mathcal{J}}_{3}\}},$ \  
$\mathcal{C}_4=\overline{\{ {\bf \mathcal{J}}_{5}\}},$ \
$\mathcal{C}_5=\overline{\{ {\bf \mathcal{J}}_{6}\}},$ \\

$\mathcal{C}_6=\overline{\{ {\bf \mathcal{J}}_{8}\}},$ \
$\mathcal{C}_7=\overline{\{ {\bf \mathcal{J}}_{9}\}},$ \ 
$\mathcal{C}_8=\overline{\{ {\bf \mathcal{J}}_{10}\}},$ \  
$\mathcal{C}_9=\overline{\{ {\bf \mathcal{J}}_{11}\}},$ \
$\mathcal{C}_{10}=\overline{\{ {\bf \mathcal{J}}_{12}\}},$ \ $\mathcal{C}_{11}=\overline{\{ {\bf \mathcal{J}}_{13}\}},$ \
$\mathcal{C}_{12}=\overline{\{ {\bf \mathcal{J}}_{14}\}},$ \
$\mathcal{C}_{13}=\overline{\{ {\bf \mathcal{J}}_{16}^{t}\}},$ \
$\mathcal{C}_{14}=\overline{\{ {\bf \mathcal{J}}_{24}\}},$ \
$\mathcal{C}_{15}=\overline{\{ {\bf \mathcal{J}}_{32}\}},$ \
$\mathcal{C}_{16}=\overline{\{ {\bf \mathcal{J}}_{42}\}},$ \
$\mathcal{C}_{17}=\overline{\{ {\bf \mathcal{J}}_{49}\}},$ \
$\mathcal{C}_{18}=\overline{\{ {\bf \mathcal{J}}_{50}\}},$ \
$\mathcal{C}_{19}=\overline{\{ {\bf \mathcal{J}}_{54}\}},$ \
$\mathcal{C}_{20}=\overline{\{ {\bf \mathcal{J}}_{56}\}},$ \
$\mathcal{C}_{21}=\overline{\{ {\bf \mathcal{J}}_{57}\}},$ \
$\mathcal{C}_{22}=\overline{\{ {\bf \mathcal{J}}_{58}\}},$ \
$\mathcal{C}_{23}=\overline{\{ {\bf \mathcal{J}}_{62}\}},$ \
$\mathcal{C}_{24}=\overline{\{ {\bf \mathcal{J}}_{64}\}},$ \
$\mathcal{C}_{25}=\overline{\{ {\bf \mathcal{J}}_{65}\}}.$
\end{center}

In particular, 24 of them are rigid superalgebras.
 
\end{Th}

\begin{proof}
After carefully  checking  the dimensions of orbit closures of the more important for us superalgebras, we have 

\begin{longtable}{rcl}
      
$\dim  \mathcal{O}({\bf \mathcal{J}}_{2})=
\dim  \mathcal{O}({\bf \mathcal{J}}_{3})=
\dim  \mathcal{O}({\bf \mathcal{J}}_{5})= 
\dim  \mathcal{O}({\bf \mathcal{J}}_{6})$&$=$&$13,$ \\ 
$\dim  \mathcal{O}({\bf \mathcal{J}}_{9})=
\dim  \mathcal{O}({\bf \mathcal{J}}_{11})=

\dim  \mathcal{O}({\bf \mathcal{J}}_{56})$&$=$&$13,$ \\ 

$\dim \mathcal{O}({\bf \mathcal{J}}_{1})=
\dim \mathcal{O}({\bf \mathcal{J}}_{10})=
\dim \mathcal{O}({\bf \mathcal{J}}_{12})=
\dim \mathcal{O}({\bf \mathcal{J}}_{13})=
\dim  \mathcal{O}({\bf \mathcal{J}}_{14})=
\dim \mathcal{O}({\bf \mathcal{J}}_{16}^{t})$&$=$&$12,$ \\ 

$\dim  \mathcal{O}({\bf \mathcal{J}}_{24})=\dim \mathcal{O}({\bf \mathcal{J}}_{32})=
\dim \mathcal{O}({\bf \mathcal{J}}_{42})=
\dim \mathcal{O}({\bf \mathcal{J}}_{54})=
\dim \mathcal{O}({\bf \mathcal{J}}_{57})=
\dim \mathcal{O}({\bf \mathcal{J}}_{64})$&$=$&$12,$ \\

$\dim  \mathcal{O}({\bf \mathcal{J}}_{8})=

\dim  \mathcal{O}({\bf \mathcal{J}}_{49})=
\dim \mathcal{O}({\bf \mathcal{J}}_{62})$&$=$&$11,$\\

$\dim  \mathcal{O}({\bf \mathcal{J}}_{50})=
\dim  \mathcal{O}({\bf \mathcal{J}}_{65})$&$=$&$10,$\\

$\dim  \mathcal{O}({\bf \mathcal{J}}_{58})$&$=$&$4.$
\end{longtable}   

If $E_{e_1}^t, E_{e_2}^t, E_{f_1}^t, E_{f_2}^t$ is a {\it parametric basis} for ${\bf A}\to {\bf B}$, then we denote a degeneration by ${\bf A}\xrightarrow{(E_{e_1}^t, E_{e_2}^t, E_{f_1}^t, E_{f_2}^t)} {\bf B}$. 

\begin{longtable}{lcl|lcl|lcl} \hline

$\bf \mathcal{J}_{8}$ & $\xrightarrow{ (e_1,e_2,tf_1, f_2)}$ & $\bf \mathcal{J}_{7}$ & 

$\bf \mathcal{J}_{16}^{0}$ & $\xrightarrow{ (e_1, e_2, f_1, tf_2)}$ & $\bf \mathcal{J}_{15}$

&
$\bf \mathcal{J}_{68}$ & $\xrightarrow{ (e_1, \frac{1}{t}e_2, f_1,f_2)}$ & $\bf \mathcal{J}_{17}$ \\  \hline 

$\bf \mathcal{J}_{23}$ & $\xrightarrow{ (te_1, e_2, f_1, f_2)}$ & $\bf \mathcal{J}_{18}$ &

$\bf \mathcal{J}_{24}$ & $\xrightarrow{ (te_1, e_2, f_1, f_2)}$ & $\bf \mathcal{J}_{20}$ & 

$\bf \mathcal{J}_{22}$ & $\xrightarrow{ (e_1, e_2, tf_1,f_2)}$ & $\bf \mathcal{J}_{21}$ 

 \\  \hline

$\bf \mathcal{J}_{24}$ & $\xrightarrow{ (e_1, t e_2, tf_1, f_2)}$ & $\bf \mathcal{J}_{22}$ &

$\bf \mathcal{J}_{24}$ & $\xrightarrow{ (e_1, e_2, tf_1, tf_2)}$ & $\bf \mathcal{J}_{23}$

&

$\bf \mathcal{J}_{13}$ & $\xrightarrow{ (e_1, t e_2, f_1, f_2)}$ & $\bf \mathcal{J}_{25}$  \\  \hline

$\bf \mathcal{J}_{14}$ & $\xrightarrow{ (e_1, te_2, f_1, f_2)}$ & $\bf \mathcal{J}_{26}$

&

 $\bf \mathcal{J}_{30}$ & $\xrightarrow{ (e_1,te_2,f_1, tf_2)}$ & $\bf \mathcal{J}_{28}$ &

$\bf \mathcal{J}_{30}$ & $\xrightarrow{ (e_1, \frac{1}{t}e_2, f_1, f_2)}$ & $\bf \mathcal{J}_{29}$

 \\  \hline
 
  $\bf \mathcal{J}_{16}^{t}$ & $\xrightarrow{ (e_1, te_2, f_1, f_2)}$ & $\bf \mathcal{J}_{30}$ &
  
 $\bf \mathcal{J}_{32}$ & $\xrightarrow{ (e_1, \frac{1}{t}e_2, tf_1, f_2)}$ & $\bf \mathcal{J}_{31}$

&

$\bf \mathcal{J}_{35}$ & $\xrightarrow{ (e_1,e_2,f_1, tf_2)}$ & $\bf \mathcal{J}_{34}$ \\  \hline
  
$\bf \mathcal{J}_{1}$ & $\xrightarrow{ (e_1+e_2, te_2, f_1, f_2)}$ & $\bf \mathcal{J}_{36}$

&

 $\bf \mathcal{J}_{40}$ & $\xrightarrow{ (e_1, e_2, f_1, tf_2)}$ & $\bf \mathcal{J}_{39}$ &
  
$\bf \mathcal{J}_{42}$ & $\xrightarrow{ (e_1, te_2, tf_1, f_2)}$ & $\bf \mathcal{J}_{40}$ 

\\  \hline

$\bf \mathcal{J}_{42}$ & $\xrightarrow{ (e_1,e_2,tf_1, tf_2)}$ & $\bf \mathcal{J}_{41}$ &
  
$\bf \mathcal{J}_{5}$ & $\xrightarrow{ (e_1+e_2, te_2, f_2, f_1)}$ & $\bf \mathcal{J}_{43}$

&

$\bf \mathcal{J}_{45}$ & $\xrightarrow{ (e_1, e_2, tf_1,f_2)}$ & $\bf \mathcal{J}_{44}$ \\  \hline 
  
$\bf \mathcal{J}_{47}$ & $\xrightarrow{ (e_1, te_2, f_1, tf_2)}$ & $\bf \mathcal{J}_{45}$

&

$\bf \mathcal{J}_{47}$ & $\xrightarrow{ (e_1, \frac{1}{t}e_2, f_1, f_2)}$ & $\bf \mathcal{J}_{46}$ &

$\bf \mathcal{J}_{49}$ & $\xrightarrow{ (e_1, e_2, tf_1, tf_2)}$ & $\bf \mathcal{J}_{48}$

\\  \hline

$\bf \mathcal{J}_{52}$ & $\xrightarrow{ (e_1, e_2, f_1, tf_2)}$ & $\bf \mathcal{J}_{51}$ &

$\bf \mathcal{J}_{54}$ & $\xrightarrow{ (e_1, e_2, tf_1, tf_2)}$ & $\bf \mathcal{J}_{53}$ 

&

$\bf \mathcal{J}_{56}$ & $\xrightarrow{ (e_1, e_2, tf_1, tf_2)}$ & $\bf \mathcal{J}_{55}$
\\  \hline 

$\bf \mathcal{J}_{60}$ & $\xrightarrow{ (e_1, e_2, f_1, tf_2)}$ & $\bf \mathcal{J}_{59}$
&
$\bf \mathcal{J}_{64}$ & $\xrightarrow{ (e_1, te_2,tf_1, f_2)}$ & $\bf \mathcal{J}_{60}$ &

$\bf \mathcal{J}_{62}$ & $\xrightarrow{ (e_1, e_2, tf_1, tf_2)}$ & $\bf \mathcal{J}_{61}$

\\  \hline

$\bf \mathcal{J}_{64}$ & $\xrightarrow{ (e_1, e_2,tf_1,tf_2)}$ & $\bf \mathcal{J}_{63}$ &

$\bf \mathcal{J}_{67}$ & $\xrightarrow{ (e_1, e_2, f_1, tf_2)}$ & $\bf \mathcal{J}_{66}$

&

$\bf \mathcal{J}_{71}$ & $\xrightarrow{ (e_1, e_2, tf_1, tf_2)}$ & $\bf \mathcal{J}_{70}$  

\\  \hline

\end{longtable}

\begin{longtable}{lcl}
\hline
$\bf \mathcal{J}_{6}$ & $\xrightarrow{ (e_1, e_2+tf_1+t^2f_2,f_1+tf_2,tf_2)}$ & $\bf \mathcal{J}_{4}$\\
\hline
$\bf \mathcal{J}_{42}$ & $\xrightarrow{ (t^{\frac23}e_2 - \frac{t^{\frac{13}{12}}}{2}f_2, te_1+\frac{2+t}{2}e_2+ t^{\frac{7}{12}}f_2, t^{\frac13} f_1, t^{\frac13}f_2)}$ & $\bf \mathcal{J}_{19}$ \\
\hline$\bf \mathcal{J}_{28}$ & $\xrightarrow{ (e_1+f_1+f_2, \frac{1}{t}e_2, f_1+f_2, f_2)}$ & $\bf \mathcal{J}_{27}$ \\
\hline
$\bf \mathcal{J}_{5}$ & $\xrightarrow{ (\frac{t}{3}e_1+(1+\frac{2t}{3})e_2, \sqrt{t}e_1+\sqrt{t}e_2, \sqrt{t}f_1+(2t-3)f_2, f_1+4\sqrt{t}f_2)}$ & $\bf \mathcal{J}_{33}$ \\
\hline
$\bf \mathcal{J}_{6}$ & $\xrightarrow{ (-te_1+(1+t)e_2, \sqrt{t}e_1+\sqrt{t}e_2,f_1, (1+2t)f_2)}$ & $\bf \mathcal{J}_{35}$\\
\hline
$\bf \mathcal{J}_{3}$ & $\xrightarrow{ ((1-t)e_1+(1+t)e_2, \frac{(t-1)\sqrt{t}}{1+t}e_1+\sqrt{t}e_2,f_1, f_2)}$ & $\bf \mathcal{J}_{37}$\\
\hline
$\bf \mathcal{J}_{2}$ & $\xrightarrow{ ((1+2t-t^2)e_1+(1+t^2)e_2, \frac{\sqrt{t}(-1-2 t+t^2)}{1+t^2}e_1+\sqrt{t}e_2, f_1, f_2)}$ & $\bf \mathcal{J}_{38}$ \\
\hline
$\bf \mathcal{J}_{16}^{1+t}$ & $\xrightarrow{e_1+e_2+t^2f_2, \ te_2, \ \frac{1}{t}f_1, \ tf_2}$ & $\bf \mathcal{J}_{47}$\\
\hline
$\bf \mathcal{J}_{54}$ & $\xrightarrow{ (e_1, te_2+f_1+t^3f_2, \frac{1}{t}f_1, t^2f_2)}$ & $\bf \mathcal{J}_{52}$\\
\hline

$\bf \mathcal{J}_{68}$ & $\xrightarrow{ (te_1+e_2, t^2 e_2, -t^3 f_1, f_2)}$ & $\bf \mathcal{J}_{67}$\\
\hline
$\bf \mathcal{J}_{16}^{-1}$ & $\xrightarrow{ (te_1-te_2, \frac{2t^2}{1+t}e_2, f_1, tf_2)}$ & $\bf \mathcal{J}_{68}$ \\
\hline
$\bf \mathcal{J}_{9}$ & $\xrightarrow{ (-\frac{t}{\sqrt{2}}e_1+\frac{t}{\sqrt{2}}e_2, t^2e_2, f_1+2f_2, t^2f_2)}$ & $\bf \mathcal{J}_{69}$ \\
\hline
$\bf \mathcal{J}_{24}$ & $\xrightarrow{ (-t^2e_1+e_2+t^2f_1, t^2e_2+2t^2f_2, tf_1, tf_2)}$ & $\bf \mathcal{J}_{71}$\\
\hline
\end{longtable}

Below we list all the important reasons for necessary non-degenerations. All other nondegenerations which are not in this table, can be inferred from Theorem~\ref{2d} and Lemma~\ref{lema:inv}(2). Since the even parts of $\mathcal{J}_{56}, \mathcal{J}_{54}, \mathcal{J}_{57}, \mathcal{J}_{64}, \mathcal{J}_{50}, \mathcal{J}_{65}, \mathcal{J}_{58}$ superalgebras coincide with $\mathcal{B}_2$, we conclude that these superalgebras do not degenerate to others whose even part is not isomorphic to $\mathcal{B}_2$ and vice versa.

\begin{longtable}{lcl|l}
\hline
    \multicolumn{4}{c}{Non-degenerations reasons} \\
\hline

$\bf \mathcal{J}_{56}$ & $\not \rightarrow  $ & 
$\begin{array}{lllll}
\bf \mathcal{J}_{50}, \bf \mathcal{J}_{54}, \bf \mathcal{J}_{57}, \\
\bf \mathcal{J}_{58}, \bf \mathcal{J}_{62}, \bf \mathcal{J}_{64}, \bf \mathcal{J}_{65}
\end{array}$  
& 
$\mathcal R=\left\{\begin{array}{lllll}
\mbox{$ c_{41}^4-c_{12}^2-c_{13}^3=0,  c_{24}^3c_{43}^2-(c_{24}^2)^2=0,$}\\[1mm]
\mbox{$c_{22}^2=0, \ c_{11}^1-c_{12}^2-c_{13}^3-c_{14}^4=0$}\\
\end{array}\right\}
$\\
\hline

$\bf \mathcal{J}_{54}$ & $\not \rightarrow  $ & 
$\bf \mathcal{J}_{50}, \bf \mathcal{J}_{58}, \bf \mathcal{J}_{62}, \bf \mathcal{J}_{65}$ 
& 
$\mathcal R=\left\{\begin{array}{lllll}
\mbox{$c_{12}^1=0, \ c_{11}^1-c_{12}^2-c_{13}^3-c_{14}^4=0$}
\end{array}\right\}
$\\
\hline

$\bf \mathcal{J}_{57}$ & $\not \rightarrow  $ & 
$\bf \mathcal{J}_{50}, \bf \mathcal{J}_{58}, \bf \mathcal{J}_{62}, \bf \mathcal{J}_{65}$ 
& 
$\mathcal R=\left\{\begin{array}{lllll}
\mbox{$c_{12}^1=0, \ c_{13}^3=0, \ c_{11}^1-c_{14}^4=0$}
\end{array}\right\}
$\\
\hline

$\bf \mathcal{J}_{64}$ & $\not \rightarrow  $ & 
$\bf \mathcal{J}_{50}, \bf \mathcal{J}_{58}, \bf \mathcal{J}_{62}, \bf \mathcal{J}_{65}$ 
& 
$\mathcal R=\left\{\begin{array}{lllll}
c_{12}^1=0, \ 2c_{11}^1-c_{12}^2-c_{13}^3-c_{14}^4=0,\\[1mm] 
c_{13}^3+c_{31}^3=c_{11}^1, \ c_{34}^2c_{24}^3+(c_{24}^2)^2=0\\[1mm] 
\end{array}\right\}
$\\
\hline

$\bf \mathcal{J}_{62}$ & $\not \rightarrow  $ & 
$\bf \mathcal{J}_{50}, \bf \mathcal{J}_{58}, \bf \mathcal{J}_{65}$ 
& 
$\mathcal R=\left\{\begin{array}{lllll}
c_{12}^1=0, \ 2c_{11}^1-c_{12}^2-c_{13}^3-c_{14}^4=0
\end{array}\right\}
$\\
\hline

$\bf \mathcal{J}_{50}$ & $\not \rightarrow  $ & 
$\bf \mathcal{J}_{58}$ 
& 
$\mathcal R=\left\{\begin{array}{lllll}
\mbox{$c_{12}^1=0, \ c_{13}^3=0$}
\end{array}\right\}
$\\
\hline

$\bf \mathcal{J}_{65}$ & $\not \rightarrow  $ & 
$\bf \mathcal{J}_{58}$ 
& 
$\mathcal R=\left\{\begin{array}{lllll}
\mbox{$c_{12}^1=0, \ c_{11}^1-c_{13}^3=0$}
\end{array}\right\}
$\\
\hline

$\bf \mathcal{J}_{2}$ & $\not \rightarrow  $ & 
$\begin{array}{lllll}
\bf \mathcal{J}_{1}, \bf \mathcal{J}_{8}, \bf \mathcal{J}_{10}, \\[1mm] 
\bf \mathcal{J}_{12}, \bf \mathcal{J}_{13}, \bf \mathcal{J}_{14}, \\[1mm] 
\bf \mathcal{J}_{16}, \bf \mathcal{J}_{24}, \bf \mathcal{J}_{32}, \\[1mm] 
\bf \mathcal{J}_{42}, \bf \mathcal{J}_{49}\\[1mm] 
\end{array}$ 
& 
$\mathcal R=\mbox{$\left\{\begin{array}{lllll}
A_1A_3\subset A_4, \ c_{12}^2=c_{14}^4, \ c_{24}^4=c_{22}^2, \\[1mm]
c_{11}^2 c_{23}^4=c_{13}^4(c_{14}^4-c_{11}^1)
\end{array}\right\}$}
$\\
\hline

$\bf \mathcal{J}_{3}$ & $\not \rightarrow  $ & 
$\begin{array}{lllll}
\bf \mathcal{J}_{1}, \bf \mathcal{J}_{8}, \bf \mathcal{J}_{10}, \\[1mm] 
\bf \mathcal{J}_{12}, \bf \mathcal{J}_{13}, \bf \mathcal{J}_{14}, \\[1mm] 
\bf \mathcal{J}_{16}, \bf \mathcal{J}_{24}, \bf \mathcal{J}_{32}, \\[1mm] 
\bf \mathcal{J}_{42}, \bf \mathcal{J}_{49}\\[1mm] 
\end{array}$ 
& 
$\mathcal R=\left\{\begin{array}{lllll}
A_1A_3\subset A_4, \ c_{12}^2=2c_{14}^4, \ 2c_{24}^4=c_{22}^2, \\[1mm]
c_{11}^2 c_{23}^4=c_{13}^4(2c_{14}^4-c_{11}^1)
\end{array}\right\}
$\\
\hline

$\bf \mathcal{J}_{5}$ & $\not \rightarrow  $ & 
$\begin{array}{lllll}
\bf \mathcal{J}_{1}, \bf \mathcal{J}_{8}, \bf \mathcal{J}_{10}, \\[1mm] 
\bf \mathcal{J}_{12}, \bf \mathcal{J}_{13}, \bf \mathcal{J}_{14}, \\[1mm] 
\bf \mathcal{J}_{16}, \bf \mathcal{J}_{24}, \bf \mathcal{J}_{32}, \\[1mm] 
\bf \mathcal{J}_{42}, \bf \mathcal{J}_{49}\\[1mm] 
\end{array}$ 
& 
$\mathcal R=\left\{\begin{array}{lllll}
2c_{14}^4=c_{12}^2, \ 2c_{24}^4=c_{22}^2, \ c_{12}^1=0, \ c_{14}^3=0, \ c_{34}^2=0,\\[1mm]
c_{13}^3=c_{12}^2, \ c_{23}^3=c_{22}^2, \ c_{22}^1=0, \ c_{24}^3=0, \ c_{34}^1=0,\\[1mm]
c_{22}^2 c_{11}^2+(c_{11}^1-c_{12}^2)c_{12}^2=0,  \ c_{11}^2 c_{23}^4=c_{13}^4(2c_{14}^4-c_{11}^1)\\[1mm]
\end{array}\right\}
$\\
\hline

$\bf \mathcal{J}_{6}$ & $\not \rightarrow  $ & 
$\begin{array}{lllll}
\bf \mathcal{J}_{1}, \bf \mathcal{J}_{8}, \bf \mathcal{J}_{10}, \\[1mm] 
\bf \mathcal{J}_{12}, \bf \mathcal{J}_{13}, \bf \mathcal{J}_{14}, \\[1mm] 
\bf \mathcal{J}_{16}, \bf \mathcal{J}_{24}, \bf \mathcal{J}_{32}, \\[1mm] 
\bf \mathcal{J}_{42}, \bf \mathcal{J}_{49}\\[1mm] 
\end{array}$ 
& 
$\mathcal R=\left\{\begin{array}{lllll}
c_{22}^1=0, \ c_{12}^1=0, \ c_{34}^1=0, \ c_{24}^4=2c_{22}^2-c_{23}^3\\[1mm]
c_{23}^4+c_{32}^4=0, \ c_{24}^3+c_{42}^3=0,\ 
c_{34}^2 c_{23}^4=(c_{34}^4)^2, \\[1mm]
c_{22}^2 (c_{34}^2)^2 c_{11}^2+(c_{14}^4 c_{34}^2 - c_{14}^2 c_{34}^4) (c_{11}^1 c_{34}^2 - c_{14}^4 c_{34}^2 + c_{14}^2 c_{34}^4)=0 \\[1mm]
\end{array}\right\}
$\\
\hline

$\bf \mathcal{J}_{9}$ & $\not \rightarrow  $ & 
$\begin{array}{lllll}
\bf \mathcal{J}_{1}, \bf \mathcal{J}_{8}, \bf \mathcal{J}_{10}, \\[1mm] 
\bf \mathcal{J}_{12}, \bf \mathcal{J}_{13}, \bf \mathcal{J}_{14}, \\[1mm] 
\bf \mathcal{J}_{16}, \bf \mathcal{J}_{24}, \bf \mathcal{J}_{32}, \\[1mm] 
\bf \mathcal{J}_{42}, \bf \mathcal{J}_{49} \\[1mm] 
\end{array}$ 
& 
$\mathcal R=\left\{\begin{array}{lllll}
\mbox{$A_1A_3 \subset A_4, \ c_{12}^1=0, \ 2c_{24}^4=c_{22}^2, \ 2c_{14}^4=c_{11}^1+c_{12}^2$}
\end{array}\right\}
$\\
\hline

$\bf \mathcal{J}_{11}$ & $\not \rightarrow  $ & 
$\begin{array}{lllll}
\bf \mathcal{J}_{1}, \bf \mathcal{J}_{8}, \bf \mathcal{J}_{10}, \\[1mm] 
\bf \mathcal{J}_{12}, \bf \mathcal{J}_{13}, \bf \mathcal{J}_{14}, \\[1mm] 
\bf \mathcal{J}_{16}, \bf \mathcal{J}_{24}, \bf \mathcal{J}_{32}, \\[1mm] 
\bf \mathcal{J}_{42}, \bf \mathcal{J}_{49} \\[1mm] 
\end{array}$ 
& 
$\mathcal R=\left\{\begin{array}{lllll}
c_{12}^1=0, \ c_{14}^3=0, \  c_{34}^1=0, \ c_{23}^3=c_{22}^2, \ 2c_{14}^4=c_{11}^1+c_{12}^2,\\[1mm]
c_{22}^1=0, \ c_{24}^3=0, \ c_{34}^2=0, \ c_{13}^3=c_{12}^2,  \  2c_{24}^4=c_{22}^2\\[1mm]
\end{array}\right\}
$\\
\hline

$\bf \mathcal{J}_{1}$ & $\not \rightarrow  $ & 
$\bf \mathcal{J}_{8}, \bf \mathcal{J}_{49}$ 
& 
$\mathcal R=\left\{\begin{array}{lllll}
A_1A_3=0
\end{array}\right\}
$\\
\hline

$\bf \mathcal{J}_{10}$ & $\not \rightarrow  $ & 
$\bf \mathcal{J}_{8}, \bf \mathcal{J}_{49}$ 
& 
$\mathcal R=\left\{\begin{array}{lllll}
\mbox{$c_{22}^1=0,\ c_{34}^1=0, \ c_{34}^2=0, \ 2c_{23}^3=c_{22}^2$}
\end{array}\right\}
$\\
\hline

$\bf \mathcal{J}_{12}$ & $\not \rightarrow  $ & 
$\bf \mathcal{J}_{8}, \bf \mathcal{J}_{49}$ 
& 
$\mathcal R=\left\{\begin{array}{lllll}
\mbox{$c_{22}^1=0, \ c_{34}^1=0, \ c_{34}^2=0$}
\end{array}\right\}
$\\
\hline

$\bf \mathcal{J}_{13}$ & $\not \rightarrow  $ & 
$\bf \mathcal{J}_{8}, \bf \mathcal{J}_{49}$ 
& 
$\mathcal R=\left\{\begin{array}{lllll}
\mbox{$c_{12}^1=0,\ c_{34}^1=0, \ c_{34}^2=0, \ c_{22}^1=0, \ 2c_{13}^3=c_{12}^2, \ 2c_{14}^4=c_{11}^1+c_{12}^2$}
\end{array}\right\}
$\\
\hline

$\bf \mathcal{J}_{14}$ & $\not \rightarrow  $ & 
$\bf \mathcal{J}_{8}, \bf \mathcal{J}_{49}$ 
& 
$\mathcal R=\left\{\begin{array}{lllll}
\mbox{$c_{22}^1=0, \ c_{34}^1=0, \ c_{34}^2=0, \  c_{23}^3=c_{22}^2$}
\end{array}\right\}
$\\
\hline

$\bf \mathcal{J}_{16}$ & $\not \rightarrow  $ & 
$\bf \mathcal{J}_{8}$ 
& 
$\mathcal R=\left\{\begin{array}{lllll}
c_{23}^{3}+c_{24}^{4}=c_{22}^2+c_{12}^1,\ 
c_{11}^{1}+c_{12}^{2}+c_{14}^{4}=c_{13}^{3}+2c_{31}^{3}\\
\end{array}\right\}
$\\
\hline

$\bf \mathcal{J}_{16}$ & $\not \rightarrow  $ & 
$\bf \mathcal{J}_{49}$ 
& 
$\mathcal R=\left\{\begin{array}{lllll}
c_{22}^1=0, \ c_{12}^1+c_{21}^1=0, \ c_{23}^{4}+c_{32}^{4}=0 ,\ c_{24}^3+c_{42}^3=0\\

\end{array}\right\}
$\\
\hline

$\bf \mathcal{J}_{42}$ & $\not \rightarrow  $ & 
$\bf \mathcal{J}_{49}$ 
& 
$\mathcal R=\left\{\begin{array}{lllll}
\mbox{$c_{22}^1=0, \ c_{22}^2=0, \ c_{14}^2c_{24}^2=-c_{14}^3c_{34}^2, \ c_{31}^3=c_{11}^1-c_{13}^3$}
\end{array}\right\}
$\\
\hline

\end{longtable}

Here $c_{ij}^{k}$ coefficients are structural constants in the $x_1=e_1, \ x_2=e_2, \ x_3=f_1, \ x_4=f_2$ basis.

\end{proof}

\begin{Th}\label{geo3}
The variety of $4$-dimensional Jordan superalgebras of type  $(3,1)$ has 
dimension  $15$ and it has  $21$  irreducible components defined by  
\begin{center}
$\mathcal{C}_1=\overline{\{ {\bf \mathfrak{J}}_{1}\}},$ \
$\mathcal{C}_2=\overline{\{ {\bf \mathfrak{J}}_{2}\}},$ \ 
$\mathcal{C}_3=\overline{\{ {\bf \mathfrak{J}}_{3}\}},$ \  
$\mathcal{C}_4=\overline{\{ {\bf \mathfrak{J}}_{4}\}},$ \
$\mathcal{C}_5=\overline{\{ {\bf \mathfrak{J}}_{42}\}},$ \\

$\mathcal{C}_6=\overline{\{ {\bf \mathfrak{J}}_{21}\}},$ \
$\mathcal{C}_7=\overline{\{ {\bf \mathfrak{J}}_{51}\}},$ \ 
$\mathcal{C}_8=\overline{\{ {\bf \mathfrak{J}}_{53}\}},$ \  
$\mathcal{C}_9=\overline{\{ {\bf \mathfrak{J}}_{57}\}},$ \
$\mathcal{C}_{10}=\overline{\{ {\bf \mathfrak{J}}_{22}\}},$ \ $\mathcal{C}_{11}=\overline{\{ {\bf \mathfrak{J}}_{23}\}},$ \
$\mathcal{C}_{12}=\overline{\{ {\bf \mathfrak{J}}_{26}\}},$ \
$\mathcal{C}_{13}=\overline{\{ {\bf \mathfrak{J}}_{43}\}},$ \
$\mathcal{C}_{14}=\overline{\{ {\bf \mathfrak{J}}_{54}\}},$ \
$\mathcal{C}_{15}=\overline{\{ {\bf \mathfrak{J}}_{58}\}},$ \
$\mathcal{C}_{16}=\overline{\{ {\bf \mathfrak{J}}_{55}\}},$ \
$\mathcal{C}_{17}=\overline{\{ {\bf \mathfrak{J}}_{47}\}},$ \
$\mathcal{C}_{18}=\overline{\{ {\bf \mathfrak{J}}_{48}\}},$ \
$\mathcal{C}_{19}=\overline{\{ {\bf \mathfrak{J}}_{49}\}},$ \
$\mathcal{C}_{20}=\overline{\{ {\bf \mathfrak{J}}_{24}\}},$ \
$\mathcal{C}_{21}=\overline{\{ {\bf \mathfrak{J}}_{25}\}},$ \

\end{center}

In particular, all of them are rigid superalgebras.
 
\end{Th}

\begin{proof}
After carefully  checking  the dimensions of orbit closures of the more important for us superalgebras, we have 

\begin{longtable}{rcl}
      
$\dim  \mathcal{O}({\bf \mathfrak{J}}_{1})$&$=$&$15,$ \\ 
$\dim  \mathcal{O}({\bf \mathfrak{J}}_{2})=
\dim  \mathcal{O}({\bf \mathfrak{J}}_{3})=
\dim  \mathcal{O}({\bf \mathfrak{J}}_{4})= 
\dim  \mathcal{O}({\bf \mathfrak{J}}_{42})$&$=$&$14,$ \\ 

$\dim  \mathcal{O}({\bf \mathfrak{J}}_{21})=
\dim  \mathcal{O}({\bf \mathfrak{J}}_{51})=
\dim  \mathcal{O}({\bf \mathfrak{J}}_{53})=
\dim  \mathcal{O}({\bf \mathfrak{J}}_{24})=
\dim  \mathcal{O}({\bf \mathfrak{J}}_{57})$&$=$&$13,$ \\ 

$\dim \mathcal{O}({\bf \mathfrak{J}}_{22})=
\dim \mathcal{O}({\bf \mathfrak{J}}_{23})=
\dim \mathcal{O}({\bf \mathfrak{J}}_{26})=
\dim \mathcal{O}({\bf \mathfrak{J}}_{43})=
\dim \mathcal{O}({\bf \mathfrak{J}}_{54})=
\dim \mathcal{O}({\bf \mathfrak{J}}_{58})$&$=$&$12,$ \\

$\dim  \mathcal{O}({\bf \mathfrak{J}}_{55})$&$=$&$11,$\\

$\dim  \mathcal{O}({\bf \mathfrak{J}}_{25})$&$=$&$10,$\\

$\dim  \mathcal{O}({\bf \mathfrak{J}}_{47})=
\dim  \mathcal{O}({\bf \mathfrak{J}}_{48})$&$=$&$9,$\\

$\dim  \mathcal{O}({\bf \mathfrak{J}}_{49})$&$=$&$4.$
\end{longtable}   

If $E_{e_1}^t, E_{e_2}^t, E_{e_3}^t, E_{f}^t$ is a {\it parametric basis} for ${\bf A}\to {\bf B}$, then we denote a degeneration by ${\bf A}\xrightarrow{(E_{f_1}^t, E_{f_2}^t, E_{f_3}^t, E_{e}^t)} {\bf B}$.

\begin{longtable}{lcl|lcl} \hline

$\bf \mathfrak{J}_{1}$ & $\xrightarrow{ (e_1, e_2, te_3, f)}$ & $\bf \mathfrak{J}_{5}$ & 

$\bf \mathfrak{J}_{2}$ & $\xrightarrow{ (e_1, e_2, te_3, f)}$ & $\bf \mathfrak{J}_{6}$

\\  \hline

$\bf \mathfrak{J}_{3}$ & $\xrightarrow{ (e_1, e_2, te_3, f)}$ & $\bf \mathfrak{J}_{7}$ & 

$\bf \mathfrak{J}_{4}$ & $\xrightarrow{ (e_1, e_2, te_3, f)}$ & $\bf \mathfrak{J}_{8}$
\\  \hline

$\bf \mathfrak{J}_{1}$ & $\xrightarrow{ (e_1, te_2, te_3, f)}$ & $\bf \mathfrak{J}_{9}$ & 

$\bf \mathfrak{J}_{2}$ & $\xrightarrow{ (e_1, te_2, te_3, f)}$ & $\bf \mathfrak{J}_{10}$
\\  \hline

$\bf \mathfrak{J}_{3}$ & $\xrightarrow{ (e_1, te_2, te_3, f)}$ & $\bf \mathfrak{J}_{11}$ & 

$\bf \mathfrak{J}_{1}$ & $\xrightarrow{ (e_1+e_2, te_2, e_3, f)}$ & $\bf \mathfrak{J}_{12}$
\\  \hline

$\bf \mathfrak{J}_{2}$ & $\xrightarrow{ (e_2+e_3, te_2, e_1, f)}$ & $\bf \mathfrak{J}_{13}$ & 
$\bf \mathfrak{J}_{3}$ & $\xrightarrow{ (e_2+e_3, te_2, e_1, f)}$ & $\bf \mathfrak{J}_{14}$

\\  \hline

$\bf \mathfrak{J}_{2}$ & $\xrightarrow{ (e_1+e_2, te_2, e_3, f)}$ & $\bf \mathfrak{J}_{15}$ & 
$\bf \mathfrak{J}_{3}$ & $\xrightarrow{ (e_1+e_2, te_2, e_3, f)}$ & $\bf \mathfrak{J}_{16}$

\\  \hline
$\bf \mathfrak{J}_{4}$ & $\xrightarrow{ (e_2+e_3, te_2, e_1, f)}$ & $\bf \mathfrak{J}_{17}$
 & 
$\bf \mathfrak{J}_{1}$ & $\xrightarrow{ (e_1+e_2, te_2, te_3, f)}$ & $\bf \mathfrak{J}_{18}$

\\  \hline
$\bf \mathfrak{J}_{2}$ & $\xrightarrow{ (e_1+e_2, te_2, te_3, f)}$  & $\bf \mathfrak{J}_{19}$
 & 
$\bf \mathfrak{J}_{3}$ & $\xrightarrow{ (e_1+e_2, te_2, te_3, f)}$ & $\bf \mathfrak{J}_{20}$

\\  \hline

$\bf \mathfrak{J}_{21}$ & $\xrightarrow{ (e_1, e_2, te_3, f)}$  & $\bf \mathfrak{J}_{27}$
  & 
$\bf \mathfrak{J}_{51}$ & $\xrightarrow{ (e_1,te_2,e_3, f)}$ & $\bf \mathfrak{J}_{28}$
\\  \hline

$\bf \mathfrak{J}_{26}$ & $\xrightarrow{ (e_1, e_2, te_3, f)}$ & $\bf \mathfrak{J}_{29}$
&
$\bf \mathfrak{J}_{1}$ & $\xrightarrow{ ((t-t^2)e_1+te_2, t^3e_2, e_3, f)}$  & $\bf \mathfrak{J}_{30}$
 
\\  \hline

$\bf \mathfrak{J}_{2}$ & $\xrightarrow{ (te_2+t^2e_3, t^2e_2+t^4e_3, e_1, f)}$ & $\bf \mathfrak{J}_{31}$

& 
$\bf \mathfrak{J}_{3}$ & $\xrightarrow{ (te_2+t^2e_3, t^2e_2+t^4e_3, e_1, f)}$ & $\bf \mathfrak{J}_{32}$

\\  \hline
 
$\bf \mathfrak{J}_{1}$ & $\xrightarrow{ ((t-t^2)e_1+te_2, t^3e_2,te_3, f)}$ & $\bf \mathfrak{J}_{33}$

&

$\bf \mathfrak{J}_{1}$ & $\xrightarrow{ (e_1+e_2+e_3, (t-t^2)e_2+te_3, t^3e_3, f)}$ & $\bf \mathfrak{J}_{34}$

\\  \hline

$\bf \mathfrak{J}_{2}$ & $\xrightarrow{ (e_1+e_2+e_3, (t-t^2)e_2+te_3, t^3e_3, f)}$ & $\bf \mathfrak{J}_{35}$
&
$\bf \mathfrak{J}_{3}$ & $\xrightarrow{ (e_1+e_2+e_3, (t-t^2)e_2+te_3, t^3e_3, f)}$ & $\bf \mathfrak{J}_{36}$
 
\\  \hline

$\bf \mathfrak{J}_{1}$ & $\xrightarrow{ (e_1+e_2+e_3, te_2,te_3, f)}$ & $\bf \mathfrak{J}_{37}$
& 
$\bf \mathfrak{J}_{2}$ & $\xrightarrow{ (e_1+e_2+e_3, te_2,te_3, f)}$ & $\bf \mathfrak{J}_{38}$

\\  \hline

$\bf \mathfrak{J}_{3}$ & $\xrightarrow{ (e_1+e_2+e_3, te_2,te_3, f)}$ & $\bf \mathfrak{J}_{39}$

& 
$\bf \mathfrak{J}_{34}$ & $\xrightarrow{ (te_1+e_2,te_2+e_3,te_3, f)}$ & $\bf \mathfrak{J}_{40}$

\\  \hline
 
$\bf \mathfrak{J}_{40}$ & $\xrightarrow{ (te_1+\frac{1+t}{2}e_2,te_3, e_2+e_3, f)}$ & $\bf \mathfrak{J}_{41}$
&
$\bf \mathfrak{J}_{21}$ & $\xrightarrow{ (e_1+e_3,e_2,te_3, f)}$ & $\bf \mathfrak{J}_{44}$

\\  \hline

$\bf \mathfrak{J}_{22}$ & $\xrightarrow{ (e_1+e_3,e_2,te_3, f)}$ & $\bf \mathfrak{J}_{45}$
 & 
$\bf \mathfrak{J}_{23}$ & $\xrightarrow{ (e_1+e_3,e_2,te_3, f)}$ & $\bf \mathfrak{J}_{46}$

\\  \hline

$\bf \mathfrak{J}_{42}$ & $\xrightarrow{ (e_1,te_3,t^2e_2, f)}$ & $\bf \mathfrak{J}_{50}$

 & 
$\bf \mathfrak{J}_{43}$ & $\xrightarrow{ (e_1,te_3,t^2e_2, f)}$ & $\bf \mathfrak{J}_{52}$

\\  \hline

$\bf \mathfrak{J}_{42}$ & $\xrightarrow{ (e_1,e_2,te_3, f)}$ & $\bf \mathfrak{J}_{56}$

 & 
$\bf \mathfrak{J}_{43}$ & $\xrightarrow{ (e_1,e_2,te_3, f)}$ & $\bf \mathfrak{J}_{59}$

\\  \hline

\end{longtable}

Below we list all the important reasons for necessary non-degenerations. All other nondegenerations which are not in this table, can be inferred from Theorem~\ref{3d} and Lemma~\ref{lema:inv}(2).

\begin{longtable}{lcl|l}
\hline
    \multicolumn{4}{c}{Non-degenerations reasons} \\
\hline

$\bf \mathfrak{J}_{1}$ & $\not \rightarrow  $ & 
$\bf \mathfrak{J}_{2},  \bf \mathfrak{J}_{3}, \bf \mathfrak{J}_{4}$ 
& 
$\mathcal R=\left\{\begin{array}{lllll}
\mbox{$A_1A_4 = 0$}
\end{array}\right\}
$\\
\hline

$\bf \mathfrak{J}_{42}$ & $\not \rightarrow  $ & 
$\bf \mathfrak{J}_{43}, \bf \mathfrak{J}_{51}, \bf \mathfrak{J}_{57}, \bf \mathfrak{J}_{58}$ 
& 
$\mathcal R=\left\{\begin{array}{lllll}
\mbox{$A_1A_4 = 0$}
\end{array}\right\}
$\\
\hline

$\bf \mathfrak{J}_{21}$ & $\not \rightarrow  $ & 
$\bf \mathfrak{J}_{22}, \bf \mathfrak{J}_{23},  \bf \mathfrak{J}_{25}, \bf \mathfrak{J}_{26}$ 
& 
$\mathcal R=\left\{\begin{array}{lllll}
\mbox{$A_1A_4 = 0$}
\end{array}\right\}
$\\
\hline

$\bf \mathfrak{J}_{24}$ & $\not \rightarrow  $ & 
$\bf \mathfrak{J}_{22}, \mathfrak{J}_{23}, \mathfrak{J}_{25}, \mathfrak{J}_{26}$ 
& 
$\mathcal R=\left\{\begin{array}{lllll}
\mbox{$A_2A_2 \subset A_3,\ A_2A_4=0, \ c_{11}^1=c_{14}^4$}
\end{array}\right\}
$\\
\hline

$\bf \mathfrak{J}_{47}$ & $\not \rightarrow  $ & 
$\bf \mathfrak{J}_{49}$ 
& 
$\mathcal R=\left\{\begin{array}{lllll}
\mbox{$A_1A_4 = 0$}
\end{array}\right\}
$\\
\hline

$\bf \mathfrak{J}_{53}$ & $\not \rightarrow  $ & 
$\bf \mathfrak{J}_{54}, \ \bf \mathfrak{J}_{55}$ 
& 
$\mathcal R=\left\{\begin{array}{lllll}
\mbox{$A_1A_4 = 0$}
\end{array}\right\}
$\\
\hline

$\bf \mathfrak{J}_{48}$ & $\not \rightarrow  $ & 
$\bf \mathfrak{J}_{49}$ 
& 
$\mathcal R=\left\{\begin{array}{lllll}
\mbox{$c_{12}^1=0, \ c_{13}^1=0, \ c_{11}^1=c_{14}^4$}
\end{array}\right\}
$\\
\hline

$\bf \mathfrak{J}_{54}$ & $\not \rightarrow  $ & 
$\bf \mathfrak{J}_{55}$ 
& 
$\mathcal R=\left\{\begin{array}{lllll}
\mbox{$c_{34}^4=0, \ c_{24}^4=0, \ c_{11}^1=c_{14}^4$}
\end{array}\right\}
$\\
\hline

$\bf \mathfrak{J}_{57}$ & $\not \rightarrow  $ & 
$\bf \mathfrak{J}_{58}$ 
& 
$\mathcal R=\left\{\begin{array}{lllll}
\mbox{$c_{34}^4=0, \ c_{24}^4=0, \ c_{11}^1=c_{14}^4$}
\end{array}\right\}
$\\
\hline

$\bf \mathfrak{J}_{22}$ & $\not \rightarrow  $ & 
$\bf \mathfrak{J}_{25}$ 
& 
$\mathcal R=\left\{\begin{array}{lllll}
\mbox{$c_{22}^1=0, \ c_{22}^2=0, \ c_{33}^1=0, \ c_{33}^2=0, \ c_{33}^3=c_{34}^4$}
\end{array}\right\}
$\\
\hline

$\bf \mathfrak{J}_{23}$ & $\not \rightarrow  $ & 
$\bf \mathfrak{J}_{25}$ 
& 
$\mathcal R=\left\{\begin{array}{lllll}
\mbox{$c_{22}^2=0, \ c_{23}^1=0, \ c_{23}^2=0, \ c_{33}^2=0, \ c_{24}^4=\frac{1}{2}c_{23}^3, \ c_{34}^4=\frac{1}{2}c_{33}^3$}
\end{array}\right\}
$\\
\hline

$\bf \mathfrak{J}_{26}$ & $\not \rightarrow  $ & 
$\bf \mathfrak{J}_{25}$ 
& 
$\mathcal R=\left\{\begin{array}{lllll}
\mbox{$c_{22}^2=0, \ c_{23}^1=0, \ c_{23}^2=0, \ c_{33}^2=0, \ c_{24}^4=\frac{1}{2}c_{23}^3, \ c_{34}^4=\frac{1}{2}c_{33}^3$}
\end{array}\right\}
$\\
\hline

\end{longtable}

Here $c_{ij}^{k}$ coefficients are structural constants in the $x_1=e_1, \ x_2=e_2, \ x_3=e_3, \ x_4=f$ basis.

\end{proof}

\end{document}